\documentclass[10pt,a4paper]{article}

\usepackage{amsmath}
\usepackage{amsfonts}
\usepackage{amsthm}
\usepackage{amssymb}
\usepackage[utf8]{inputenc}
\usepackage{color}
\usepackage{verbatim}

\usepackage{enumerate}

\newtheorem{mthm}{Theorem}
\newtheorem{thm}{Theorem}[section]
\newtheorem{lem}{Lemma}[section]
\newtheorem{hyp}{Hypothesis}

\newtheorem{cor}[lem]{Corollary}
\newtheorem{prop}[thm]{Proposition}

\newtheorem{rmk}{Remark}[section]

\theoremstyle{remark}

\newcommand{\argmin}{\operatorname{arg min}}

\newcommand{\Z}{\mathbb{Z}}
\newcommand{\Rm}{\mathbb{R}}
\newcommand{\R}{\mathbb{R}}
\newcommand{\NN}{\mathbb{N}}

\newcommand{\mC}{\ensuremath{\mathcal{C}}}
\newcommand{\mR}{\ensuremath{\mathcal{R}}}
\newcommand{\mU}{\ensuremath{\mathcal{U}}}

\newcommand{\mS}{\ensuremath{\mathcal{S}}}
\newcommand{\mF}{\ensuremath{\mathcal{F}}}
\newcommand{\mP}{\ensuremath{\mathcal{P}}}
\newcommand{\mD}{\ensuremath{\mathcal{D}}}

\newcommand{\mH}{\ensuremath{\mathcal{H}}}
\newcommand{\mV}{\ensuremath{\mathcal{V}}}

\newcommand{\Qm}{\ensuremath{\mathbb{Q}}}

\newcommand{\Nm}{\ensuremath{\mathbb{N}}}
\newcommand{\Zm}{\ensuremath{\mathbb{Z}}}

\newcommand{\mK}{\ensuremath{\mathcal{K}}}
\newcommand{\mB}{\ensuremath{\mathcal{B}}}
\newcommand{\mA}{\ensuremath{\mathcal{A}}}

\newcommand{\mI}{\ensuremath{\mathcal{I}}}

\newcommand{\mG}{\ensuremath{\mathcal{G}}}

\newcommand{\mQ}{\ensuremath{\mathcal{Q}}}
\newcommand{\mN}{\ensuremath{\mathcal{N}}}
\newcommand{\Tm}{\ensuremath{\mathbb{T}}}
\newcommand{\T}{\ensuremath{\mathbb{T}}}

\newcommand{\eps}{\ensuremath{\epsilon}}

\newcommand{\tmA}{\tilde{\mathcal{A}}}
\newcommand{\tmN}{\tilde{\mathcal{N}}}
\newcommand{\tmS}{\tilde{\mathcal{S}}}
\newcommand{\tmH}{\tilde{\mathcal{H}}}
\newcommand{\tmE}{\tilde{\mathcal{E}}}
\newcommand{\mE}{\mathcal{E}}
\newcommand{\cmE}{\check{\mathcal{E}}}
\newcommand{\dist}{\mathrm{dist}}

\def\lto{\longrightarrow}
\def\to{\longrightarrow}
\def\lmto{\longmapsto}

\def\leq{\leqslant}
\def\geq{\geqslant}
\def\le{\leqslant}
\def\ge{\geqslant}

\def\bdef{\begin{definition}}
\def\endef{\end{definition}}
\def\bthm{\begin{mthm}}
\def\ethm{\end{mthm}}
\def\bthm{\begin{thm}}
\def\ethm{\end{thm}}
\def\blm{\begin{lemma}}
\def\elm{\end{lemma}}
\def\brm{\begin{rmk}}
\def\erm{\end{rmk}}
\def\bprop{\begin{proposition}}
\def\eprop{\end{proposition}}
\def\bcor{\begin{cor}}
\def\ecor{\end{cor}}

\def\om{\omega}

\def\epsilon{\varepsilon}
\def\A{\mathcal A}
\def\R{\mathbb R}

\def\T{\mathbb T}

\def\Z{\mathbb Z}

\def\cC{\mathcal C}

\def\B{\mathcal B}

\def\Gm{\Gamma}
\def\Lb{\Lambda}
\def\th{\theta}
\def\dt{\delta}
\def\lb{\lambda}

\newtheorem{definition}{Definition}

\author{P. Bernard\footnote{Universit{\'e} Paris-Dauphine (\texttt{patrick.bernard\@ ceremade.dauphine.fr})},
 V. Kaloshin\footnote{University of Maryland at College Park (\texttt{vadim.kaloshin\@ gmail.com})},
 K. Zhang\footnote{University of Toronto (\texttt{kzhang\@ math.utoronto.edu})}}

\title{Arnold diffusion in arbitrary degrees of freedom \\
and  normally hyperbolic invariant cylinders}

\begin{document}
\maketitle

\begin{abstract}
We prove a form of Arnold diffusion in the a priori stable case. 
Let 
$$H_0(p)+\epsilon H_1(\th,p,t),\quad \th\in \T^n,\ p\in B^n,\ t\in \T=\R/\T
$$ 
be a nearly integrable system of arbitrary degrees of freedom 
$n\ge 2$ with a strictly convex $H_0$. We show  that for 
a ``generic'' $\epsilon H_1$, 
there exists an orbit $(\theta,p)(t)$ satisfying
\[
\|p(t)-p(0) \| >l(H_1)>0,
\]
where $l(H_1)$ is independent of $\epsilon$. 
The diffusion orbit travels along  a co-dimension one resonance, 
and the only obstruction to our construction 
is a finite set of additional resonances. 

For the proof we use a combination geometric and variational methods, 
and manage to adapt tools which have recently been developed in 
the a priori unstable case.
\end{abstract}

\markboth{P. Bernard, V. Kaloshin, K. Zhang}{Arnold diffusion along crumpled 
normally hyperbolic invariant cylinders}

\section{Introduction}
On the phase space $\Tm^n\times B^n$,
we consider the Hamiltonian system generated by the $C^r$
time-periodic  Hamiltonian
$$
H_\epsilon(\th,p,t)=H_0(p)+\epsilon H_1(\th,p,t), 
\quad (\theta,p,t)\in \Tm^n\times B^n\times \Tm,
$$
where $\Tm=\Rm/\Zm,\ B^n$ is the unit ball in $\Rm^n$ 
around the origin, and $\epsilon\geq 0$ is a small parameter.
The equations 
$$
\dot \theta=\partial_p H_0+\epsilon \partial_p H_1\quad , \quad 
\dot p=-\epsilon \partial_\th H
$$
imply that the momenta $p$ are constant in the case $\epsilon=0$.
A question of general interest in Hamiltonian dynamics is 
to understand the evolution of these momenta when 
$\epsilon>0$ is small (see e.g. \cite{Ar1,Ar2,AKN}). 
In the present paper, we assume that $H_0$ is convex, 
and ,  more precisely,

\begin{equation}\label{Dd}
(1/D)\ I\leq \partial^2_pH_0 \leq D\ I,
\end{equation}
and prove that a certain form of Arnold's diffusion occur for many perturbations.
We assume that $r\geq 4$ and denote by $\mS^r$ the unit sphere in
$C^r(\Tm^n\times B^n\times \Tm)$.

\begin{mthm}\label{main}
There exist  two continuous functions
$\ell$ and $\epsilon_0$ on $\mS^r$, which are positive on an open and 
dense set \ $\mU\subset \mS^r$, and an open and dense subset
$\mV_1$ of
$$
\mV:=\{H_0+\epsilon H_1:
H_1\in {\mathcal U} ,\ \ 0<\epsilon<\epsilon_0(H_1)\}
$$
such that the following property holds for each Hamiltonian 
$H\in \mV_1$:

There exists an orbit $(\theta(t),p(t))$ of $H_\eps$
and a time $T\in \Nm$ such that
\[
\|p(T)-p(0)\|>\ell(H_1).
\]
\end{mthm}

The key point in this statement is that $\ell(H_1)$ does not depend on 
$\epsilon\in ]0, \epsilon_0(H_1)[$. In section~\ref{sec:diff-path}, 
we give a more detailed description of the diffusion path.  
Moreover, an improved version of the main theorem provides 
an explicit lower bound on $l(H_1)$ (see 
Theorem \ref{normal-form} 
and Remark \ref{length}).

The present work is in large part inspired by the work of Mather 
\cite{Ma3,Ma4,Ma5}. In \cite{Ma3}, Mather announced a much 
stronger version of Arnold diffusion for $n=2$. Our set 
$\mathcal{V}$ is what Mather called a {\it cusp residual} set.
As in Mather's work the instability phenomenon thus holds in 
an open dense subset of a cusp residual set. Our result 
is, however, quite different. We obtain a much more 
restricted form of instability, which holds for any $n\geq 2$. 
The restricted character of the diffusion comes from the fact that 
we do not really solve the problem of double resonance (but only 
finitely many, independent from $\epsilon$, double resonances 
are really problematic). The proof of Mather's result is partially 
written (see \cite{Ma4}), and he has given lectures about some 
parts of the proof \cite{Ma5}.
\footnote{ After a preliminary version of this paper was completed 
for $n=2$ the problem of double resonance was solved and existence 
of a strong form of Arnold diffusion is given in \cite{KZ2}. }

The study of Arnold diffusion was initiated by the seminal paper
of Arnold, \cite{Ar1}, where he  describes a diffusion 
phenomenon on a specific example involving two independent 
perturbations. A lot of work has then been devoted to describe 
more general situations where similar constructions could be 
achieved. A unifying aspect of all these situations  is the 
presence of a normally hyperbolic cylinder, as was understood 
in \cite{Mo} and \cite{DLS}, see also \cite{DGLS,DH,T1,T2,CY1,CY2,Be1}.
These general classes of situations have been referred to as 
\textit{ a priori unstable}.

The Hamiltonian $H_{\epsilon}$ studied here is, on the contrary, 
called {\it a priori stable}, because no hyperbolic structure is 
present in the unperturbed system $H_0$. Our method will, 
however, rely on the existence of a normally hyperbolic
invariant cylinder. The novelty here thus consists in proving that
a priori unstable methods do apply in the a priori stable case.
Application of normal forms to construct normally 
$3$-dimensional hyperbolic invariant cylinders in a priori stable 
situation in $3$ degrees of freedom had already been discussed 
in \cite{KZZ} and in \cite{Mar}. The existence of normally 
hyperbolic cylinders with a length independent from $\epsilon$ 
in the a priori stable case, in arbitrary dimension, have been proved in \cite{Be3}, see also \cite{Be5}. In the present paper, we obtain  
an explicit lower bound on the  length of such a cylinder. 
The quantity  $l(H_1)$ in the statement of Theorem  \ref{main} 
is closely related to this lower bound (see also Remark \ref{length}). 
Let us mention some additional  works of interest  around the problem 
of Arnold's diffusion \cite{Be4,Be6,BB,BBB,Bs1,Bs2,Bo,BK,CL1,CL2,Cr,GR1,GR2,
KS,KL1,KL2,KLS,LM,MS,Zha,Zhe,X}  and many others.

\subsection{Reduction to normal form}
\label{sec:diff-path}

As is usual in the theory of instability, we  build our unstable orbits
around a resonance.
A frequency $\omega \in \Rm^n$ is said {\it resonant}
if there exists $k\in \Zm^{n+1}$, $k\neq 0$,  such that
$
k\cdot (\omega,1) =0.
$
The set of such integral vectors $k$ forms a submodule 
$\Lb$ of $\Zm^{n+1}$, 
and the dimension of this module (which is 
also the dimension of the vector subspace of $\Rm^{n+1}$ 
it generates) is called the order, or the dimension of 
the resonant frequency $\omega$.

In order to apply our proof, we have to consider a resonance 
of order $n-1$ or, equivalently, of codimension $1$. 
For definiteness and simplicity, we choose once and for all to work with the resonance
$$\omega^s=0,$$
 where
$$\omega=(\omega^s,\omega^f)\in \Rm^{n-1}\times \Rm.
$$
Similarly, we use the notations
$$
\theta=(\theta^s,\theta^f)\in \Tm^{n-1}\times \Tm,\quad
p=(p^s,p^f)\in \Rm^{n-1}\times \Rm,
$$
which are the slow and fast variables associated to our resonance (see Section \ref{sec:normal-form} for 
definitions).
More precisely, we will be working around the manifold
defined by the equation 
$$
\partial_{p^s} H_0(p)=0
$$
in the phase space. In view of (\ref{Dd}), this equation defines a $C^{r-1}$
curve $\Gamma$ in $\Rm^n$, which can also be described parametrically as 
the graph of a $C^{r-1}$ function $p_*^s(p^f):\Rm^{n-1}\lto \Rm.$
We will also use the notation $p_*(p^f):= (p_*^s(p^f),p^f))$.


%

We define the averaged perturbation corresponding to the resonance $\Gamma$,
\[
Z(\theta^s, p):= \iint H_1(\th^s,p^s,\th^f,p^f,t)\, d\th^f\, dt.
\]
If the perturbation $H_1(\theta,p,t)$ is expanded as
$$
H_1(\theta,p,t)=H_1(\theta^s,\theta^f,p,t)=
\sum_{k^s\in \Zm^{n-1},k^f\in \Zm,l\in \Zm} h_{[k^s,k^f,l]}(p)
e^{2i\pi(k^s\cdot \theta^s+k^f\cdot \theta^f+l\cdot t)},
$$
 then
$$
Z(\theta^s,p)=
\sum_{k^s} h_{[k^s,0,0]}(p)
e^{2i\pi(k^s\cdot \theta^s)}.
$$
Our first generic assumption, which defines the set $\mU\subset \mS^r$
in Theorem \ref{main} is on the shape of $Z$. We assume that 
there exists a subarc $\Gamma_1\subset \Gamma$ such that :
\begin{hyp}\label{HZ}
There exists  a real number $\lambda\in ]0,1/2[$ such that, for each $p\in \Gamma_1$,
there exists $\theta^s_*(p)\in \Tm^{n-1}$ such that  the inequality
\begin{equation}\tag{$HZ\lambda$}\label{HZl}
Z(\theta^s,p)\leq Z(\theta^s_*(p),p)-\lambda d^2(\theta^s,\theta^s_*(p))
\end{equation}
holds for each $\theta^s $.
\end{hyp}

\subsection{Single maximum}
This condition implies that for each $p\in \Gamma_1$
the averaged perturbation $Z(\th,p)$ has a unique 
non-degenerate maximum at $\th^s_*(p)$. In Section \ref{sec:bifurcation} 
we relax this condition and allow 
bifurcations from one global maxima to a different one.
The set of functions $Z\in C^r(\Tm^{n-1}\times B^n)$
satisfying Hypothesis \ref{HZ} on some arc 
$\Gamma_1\subset \Gamma$ is open and dense for each 
$r\geq 2$. 
As a consequence, the set  $ \mU$ of functions 
$H_1\in \mS^r$ (the unit sphere in 
$C^r(\Tm^n\times B^n\times \Tm)$) whose average $Z$ 
satisfies Hypothesis \ref{HZ} on some  arc 
$\Gamma_1\subset \Gamma$ is open and dense in 
$\mS^r$ if $r\geq 2$.

The general principle of averaging theory is that the dynamics of $H_{\epsilon}$
is approximated by the dynamics of the averaged Hamiltonian
$H_0+\varepsilon Z$ in a neighborhood of $\Tm^n\times \Gamma$.
The applicability of this principle is limited by the presence of 
additional resonances, that is points $p\in \Gamma$ such that the 
remaining frequency $\partial_{p^f} H_0$ is rational.
Although additional resonances are dense in $\Gamma$, only finitely many of
them, called {punctures}, are really problematic.
More precisely, denoting by $U_{\epsilon^{1/3}}(\Gamma_1)$
the $\epsilon^{1/3}$-neighborhood of $\Gamma_1$ in $B^n$
and by 
$\mR(\Gamma_1,\varepsilon, \delta)\subset C^r(\Tm^n\times B^n\times \Tm)$
the set of functions $R(\th,p,t): \Tm^n\times B^n\times \Tm\lto \Rm$
such that
$$
\|R\|_{C^2( \Tm^n\times U_{\epsilon^{1/3}}(\Gamma_1)\times \Tm)}\leq \delta.
$$
We  will prove in section \ref{sec:normal-form} that :

\begin{prop}\label{prop-nf}
For each $\delta\in ]0,1[$, there exists a locally finite subset 
$\mP_{\delta}\subset \Gamma$ and $\epsilon_1\in ]0,\delta[$, such that :

For each compact arc $\Gamma_1\subset \Gamma$ disjoint from 
$\mP_{\delta}$, each $H_1\in \mS^r$, and each 
$\epsilon\subset]0,\epsilon_1[$, there exists
a \text{$C^r$} smooth canonical change of coordinates
$$
\Phi :\Tm^n\times B \times \Tm
\lto \Tm^n\times \Rm^n\times \Tm
$$
satisfying $\|\Phi-id\|_{C^0}\leq \sqrt{\epsilon} $ and such that, in 
the new coordinates, the Hamiltonian $H_0+\epsilon H_1$ takes the form
\begin{equation}\label{eq:normal-form-bis}
 N_\epsilon =  H_0(p) + \epsilon Z(\theta^s, p) + \epsilon R(\theta, p, t),
\end{equation}
with   $R\in \mR(\Gamma_1,\epsilon, \delta)$.
\end{prop}

The key aspects of this result is that the set $\mP_{\delta}$ is locally finite and independent from $\epsilon$.
Because it is essential to have these properties of $\mP_{\delta}$, the conclusions on the smallness of $R$
are not very strong. Yet they are sufficient to obtain:

\begin{thm}\label{mainnormal}
Let us consider the $C^r$ Hamiltonian 
\begin{equation}\label{eq:normal-form}
 N_\epsilon(\th,p,t)=H_0(p)+\varepsilon Z(\th^s,p)+\varepsilon R(\th,p,t),
\end{equation}
and assume that $\|Z\|_{C^2}\leq 1$ and that   (\ref{HZl}) holds on some arc 
$\Gamma_1\subset \Gamma$ of the form
$$
\Gamma_1:=\{(p_*(p^f)), p^f\in[a_-,a_+]\}.
$$
Then there exist { constants $\delta>0$ and $\eps_0$}, 
which depends only on $n$, $H_0$, and $\lambda$, and 
such that, for each $\epsilon\in ]0, \eps_0[$, the following 
property holds for {an open} dense subset of functions 
$R\in \mR(\Gamma_1,\epsilon, \delta)$ (for the $C^r$ topology):

There exists an orbit $(\th(t),p(t))$ and an integer $T\in \Nm$ 
such that
$\|p(0)-p_*(a_-)\|< \sqrt \eps$ 
and $\|p(T)- p_*(a_+)\|< \sqrt \eps$.
\end{thm}

\subsection{{Derivation of Theorem~\ref{main} using Proposition~\ref{prop-nf} and  Theorem~\ref{mainnormal}}}

Given $l>0$, we denote by $\mD^r(l)$ the set of $C^r$ Hamiltonians with 
the following property: There exists an orbit $(\theta(t),p(t))$ 
and an integer $T$ such that  $\|p(T)-p(0)\|>l$.
The set $\mD^r(l)$ is clearly open.

We denote by $\mD^r(l)$ the set of $C^r$ Hamiltonians with the following property:
There exists an orbit $(\theta(t),p(t))$ and an integer $T$ such that 
$\|p(T)-p(0)\|>l$.
The set $\mD^r(l)$ is clearly open. 


We now prove the existence of a continuous function $\epsilon_0$ on $\mS^r$
which is positive on $\mU$ and such that each Hamiltonian $H_{\epsilon}=H_0+\epsilon H_1$
with $H_1\in \mU$ and $\epsilon<\epsilon_0(H_1)$ belongs to the closure of $\mD^r(\epsilon_0(H_1))$.

For each $H_1\subset \mU$, there exists 
a compact arc $\Gamma_1\subset \Gamma$
and a number $\lambda \in ]0,1/4[$ such that the corresponding
averaged perturbation  $Z$ satisfies 
Hypothesis \ref{HZ} on $\Gamma_1$ with  constant $2\lambda$.
We then consider the real $\delta$ given by
Theorem~\ref{mainnormal} (applied with the parameter $\lambda$).
By possibly reducing the arc $\Gamma_1$, we  can assume in addition that 
this arc is disjoint from the set $\mP_{\delta}$ of punctures for this 
$\delta$.
The following properties then hold:
\begin{itemize}
\item The averaged perturbation $Z$ satisfies Hypothesis \ref{HZ} on $\Gamma_1$
with  a constant $\lambda'>\lambda$.
\item The parameter $\delta$ is associated to $\lambda$ by Theorem~\ref{mainnormal}.
\item The arc $\Gamma_1$ is disjoint from the set $\mP_{\delta}$ of punctures.
\end{itemize}

We say that $(\Gamma_1, \lambda, \delta)$ is a compatible set of data
if they satisfy the second and third point above.
Then, we denote by $\mU(\Gamma_1,\lambda, \delta)$ the set of $H_1\in \mS^r$
which satisfy the first point.
This is an open set, and we just proved that the union on all compatible
sets of data of these open sets covers $\mU$.

To each compatible set of data $(\Gamma_1, \lambda, \delta)$ we associate the positive numbers
$\ell:=\|p_--p_+\|/2$, where $p_{\pm}$ are the extremities of $\Gamma_1$,  and 
$\epsilon_2(\Gamma_1, \lambda, \delta) := \min( \epsilon_1, \ell^2/5, \ell)$, where $\epsilon_1$ is associated to $\delta$
by Proposition~ \ref{prop-nf}.

Using a partition of the unity, we can build a continuous function $\epsilon_0$ 
on $\mS^r$ which is positive on $\mU$ and have the following property:
For each $H_1\in \mU$, there exists a compatible set of data $(\Gamma_1, \lambda, \delta)$
such that $H_1\in \mU(\Gamma_1, \lambda, \delta)$ and 
$\epsilon_0(H_1)\leq \epsilon_2(\Gamma_1, \lambda, \delta)$.

For this function $\epsilon_0$, we claim that 
each Hamiltonian $H_{\epsilon}=H_0+\epsilon H_1$
with $H_1\in \mU$ and $0<\epsilon<\epsilon_0(H_1)$ belongs to the closure of $\mD^r(\epsilon_0(H_1))$.

Assuming the claim, we finish the proof of Theorem~\ref{main}.
For $l>0$, let us denote by $\mV(l)$ the open set of Hamiltonians of the form $H_0+\epsilon H_1$, where $H_1\in \mU$
satisfies $\epsilon_0(H_1)>l$ and $\epsilon\in ]0, \epsilon_0(H_1)[$.
The claim implies that $\mD(l)$ is dense in $\mV(l)$ for each $l>0$.
The conclusion of the Theorem (with $l(H_1):= \epsilon_0(H_1)$) then holds with the open set 
$\mV_1:= \cup_{l>0} (\mV(l)\cap \mD(l))$, which is open and dense in $\mV=\cup_{l>0}\mV(l)$.

To prove the claim,
we consider a Hamiltonian
$H_{\epsilon}=H_0+\epsilon H_1$, with $H_1\in \mU$ and $\epsilon\in ]0,\epsilon_0(H_1)[$.
We take a  compatible set of data $(\Gamma_1, \lambda, \delta)$ such that 
$H_1\in \mU(\Gamma_1, \lambda, \delta)$ and 
$\epsilon_0(H_1)\leq \epsilon_2(\Gamma_1, \lambda, \delta)$.
We  apply Proposition~\ref{prop-nf} to find a 
change of coordinates $\Phi$ which 
transforms the Hamiltonian $H_0+\epsilon H_1$ to a Hamiltonian in the normal
form 
$\Phi^* H_{\epsilon}=N_{\epsilon}$ with $R\in \mR(\Gamma_1, \epsilon, \delta)$.
The change of coordinates $\Phi$ is fixed for the sequel of this discussion, as well as $\epsilon$.
By Theorem~\ref{mainnormal}, the Hamiltonian $N_{\epsilon}$ can be approximated
in the $C^r$ 
norm by Hamiltonians $\tilde N_{\epsilon}$ admitting an  orbit 
$(\theta(t),p(t))$ such that  $p(0)=p_-$ and $p(T)=p_+$ for some $T\in \Nm$.
Let us denote by $\tilde H_{\epsilon}:= (\Phi^{-1})^*\tilde N_{\epsilon}$ the expression 
in the original coordinates of $\tilde N_{\epsilon}$. It can be made arbitrarily $C^r$-close
to $H_{\epsilon}$ by taking $\tilde N_{\epsilon}$ sufficiently close to $N_{\epsilon}$.
Since $\|\Phi-Id\|_{C^0}\leq \sqrt{\epsilon}$, the extended $\tilde H_{\epsilon}$-orbit 
$(x(t),y(t), t\text{ mod }1):=\Phi(\theta(t),p(t), t\text{ mod }1)$
satisfies $\|p(0)-p_-\|\leq \sqrt \epsilon$
and  $\|p(T)-p_-\|\leq \sqrt \epsilon$,
hence 
$$
\|y(T)-y(0)\|\geq \|p_+-p_-\|-2\sqrt{\epsilon}> \ell\geq \epsilon_0(H_1).
$$
In other words, we have $\tilde H_{\epsilon}\in \mD(\epsilon_0(H_1))$.
We have proved that $H_{\epsilon}$ belongs to the closure of $\mD(\epsilon_0(H_1))$.
This ends the proof of Theorem~\ref{main}.
\qed

The Hamiltonian in normal form $N_{\epsilon}$ has the typical structure of
what is called an a priori unstable system under  Hypothesis \ref{HZ}.
 Actually, under the additional assumption that
$\|R\|_{C^2}\leq \delta$, with $\delta$ sufficiently 
small with respect to  $\epsilon$,
the conclusion of Theorem \ref{mainnormal} would follow from 
the various works on the a priori unstable case, 
see \cite{Be1,CY1,CY2,DLS,GR2, T1,T2}. The difficulty here is 
the weak hypothesis made on the smallness of $R$, and, in particular, 
the fact that $\epsilon$ is allowed to be much smaller than $\delta$.

\subsection{Proof of Theorem \ref{mainnormal}}
\label{sec:proof}
We give a proof based on several intermediate results that will 
be established in the further sections of the paper. The first step is to  
establish the existence  of a normally hyperbolic cylinder. It is detailed 
in Section \ref{sec:NHIC}. As a consequence of the difficulties of our 
situation, we get only a rough control on this cylinder, as was already 
the case in \cite{Be3}. Some $C^1$ norms might blow up when 
$\epsilon\rightarrow 0$ (see (\ref{blow-up})).

The second step consists in building unstable orbits along this cylinder 
under additional generic assumptions. In the a priori unstable case, 
where a regular cylinder is present, several 
methods have been developed. Which of them can be extended 
to the present situation is unclear. Here we  manage to extend 
the variational approach of \cite{Be1,CY1,CY2} (which are based 
on Mather's work). We use the framework of \cite{Be1}, but also essentially 
appeal to ideas from \cite{Mag} and \cite{CY2} for the proof of one of 
the key genericity results. A self-contained proof of the required genericity 
with many new ingredients is presented in Section \ref{sec:variational}.

{The second step consists of three main steps: 
\begin{itemize}
\item Along a resonance $\Gm$ prove existence a normally 
hyperbolic cylinder $\mC$ and derive its properties 
(see Theorem \ref{intro-nhic-mult}). 
\item Show that this cylinlder $\mC$ contains a family of 
Ma\~n\'e sets $\tilde \mN(c),\ c \in \Gamma$, each being 
of Aubry-Mather type, i.e. a Lipschitz graph over the circle
(see Theorem \ref{intro-local}). 
\item Using the notion of a forcing class \cite{Be1} generically 
construct orbits diffusing along this cylinder $\mC$
(see Theorem \ref{intro-gen}). 
\end{itemize}
\subsubsection{Existence and properties of a normally 
hyperbolic cylinder $\mC$}}

\begin{thm}\label{intro-nhic-mult}
Let us consider the $C^r$ Hamiltonian system (\ref{eq:normal-form})
and assume that $Z$ satisfies (\ref{HZl}) on some arc $\Gamma_1\subset \Gamma$
of the form 
$$
\Gamma_1:=\{(p_*(p^f)), p^f\in[a_-,a_+]\}.
$$
Then there exist constants   $C>1>\kappa>\delta>0$, which depend only on $n$, $H_0$, 
and $\lambda$, and such that, for each $\epsilon$  in $ ]0, \delta[$,  
the following property holds for each function $R\in \mR(\Gamma_1,\epsilon, \delta)$:

There exists  a $C^2$ map
$$
(\Theta^s, P^s)(\theta^f, p^f, t):\T \times
[a_- -\kappa{\epsilon}^{1/3}, a_+ +\kappa{\epsilon}^{1/3}]\times\Tm\lto \T^{n-1}\times \Rm^{n-1}$$
such that the cylinder
$$\mC=\big \{ (\theta^s, p^s)=(\Theta^s,P^s)(\theta^f, p^f, t);
\quad p^f\in 
[a_--\kappa{\epsilon}^{1/3}, a_{+} + \kappa{\epsilon}^{1/3}],\quad  (\theta^f,t)\in \T\times \T\big \}$$
is weakly invariant with respect to  $N_\epsilon$ in the sense that 
the Hamiltonian vector field is tangent to $\mC$.
The cylinder $\mC$ is contained in  the set
\begin{align*}
W:=\big\{&(\theta,p,t); p^f\in
[a_- -\kappa{\epsilon}^{1/3}, a_{+} +\kappa{\epsilon}^{1/3}], \\
&\|\theta^s-\theta^s_*(p^f)\|\leq \kappa,
\quad
\|p^s-p^s_*(p^f)\|\le\kappa \sqrt{\epsilon}
\big\},
\end{align*}
and it contains all the full orbits  of
$N_{\epsilon}$ contained in $W$. We have the estimate
\begin{align} \label{blow-up}
\left\|\frac{ \partial \Theta^s}{\partial p^f}\right\|\leq 
C\left(1+\sqrt{\frac{\delta}{\epsilon}}\right)
\quad &,\quad 
 \left\|\frac{ \partial \Theta^s}{\partial (\theta^f,t)}\right\|\leq C(\sqrt \epsilon +\sqrt \delta),\\
 \quad 
\left\|\frac{ \partial P^s}{\partial p^f}\right\|\leq C
\quad &, \quad
\left\|\frac{ \partial P^s}{\partial (\theta^f,t)}\right\|\leq C\sqrt \epsilon,
\end{align}
$$
{\|\Theta^s(\theta^f,p^f,t)-\theta ^s_*(p^f)\|\leq C\sqrt \delta.}
$$
\end{thm}

A similar,  weaker, result is proved in  \cite{Be3}. 
The present  statement contains better quantitative estimates.
It follows from  Theorem  \ref{nhic-mult}
below, which makes these estimates even more explicit. 
The terms $\kappa \epsilon^{1/3}$ come from the fact that we only estimate $R$ on 
the $\epsilon ^{1/4}$-neighborhood of $\Gamma_1$, see the definition of $\mR(\Gamma_1,\epsilon, \delta)$.



For convenience of notations we extend our system from $\Tm^n\times B^n \times \T$
to $\Tm^n\times \Rm^n \times \T$.  
It is more pleasant in many occasions to consider 
the time-one Hamiltonian flow $\phi$ and the discrete system 
that it generates on $\Tm^n\times\Rm^n$. 
We will thus consider the cylinder
$$
\mC_0=\{(q,p)\in \Tm^n \times \Rm^n : (q,p,0) \in \mC\}.
$$
We will think  of this cylinder as being  $\phi$-invariant, although this is not 
precisely true, due to the possibility that orbits may escape through the boundaries.
If $r$ is large enough, it is possible to prove the existence of a really invariant
cylinder closed by KAM invariant circles, but this is not useful here.

The presence of this normally hyperbolic invariant cylinder is another 
similarity with the a priori unstable case. The difference is that we only 
have rough control on the present cylinders, with some estimates blowing up
when $\epsilon \lto 0$. As we will see, variational methods can still be used 
to build unstable orbits along the cylinder. We will  use the variational 
mechanism of \cite{Be1}. Variational methods for this problem were initiated 
by Mather, see \cite{Ma2} in an abstract setting.
In a quite different direction, they were also used by Bessi to study 
the Arnold's example of \cite{Ar1}, see \cite{Bs1}. 

\subsubsection{Weak KAM and Mather theory}
We will use standard notations of weak KAM and Mather theory,
we recall here the most important ones for the convenience of the reader.
We mostly use Fathi's presentation in terms of weak KAM solutions, see \cite{Fa},
and also \cite{Be1} for the non-autonomous case.
We consider the Lagrangian function $L(\theta,v,t)$ associated to $N_{\epsilon}$
(see Section \ref{sec:localization} for the definition)
and, for each $c\in  \Rm^n$, the function 
$$
G_c(\theta_0,\theta_1):= \min _{\gamma} \int_0^1 L(\gamma(t), \dot \gamma(t),t) 
-c\cdot \dot \gamma(t)  dt,
$$
where the minimum is taken on the set of $C^1$ curves
$\gamma:[0,1]\lto \Tm^n$ such that $\gamma(0)=\theta_0, \gamma(1)=\theta_1$.
It is a classical fact that this minimum exists, and that the minimizers is the projection 
of a Hamiltonian orbit.
A (discrete) {\it weak KAM solution at cohomology $c$} is a function $u\in C(\Tm^n,\Rm)$
such that 
$$
u(\theta)=\min_{v\in \Rm^n} \left[ u(\theta-v)+ 
G_c(\theta-v,\theta)+\alpha(c)\right]
$$
where $\alpha(c)$ is the only real constant such that such a function $u$ exists.
For each curve $\gamma(t):\Rm\lto \Tm^n$ and each $S<T$ in $\Zm$ 
we thus have the inequalities
$$
u(\gamma(T))-u(\gamma(S))
\leq
G_c(\gamma(S),\gamma(T))+(T-S)\alpha(c)
\leq
\int_S^T L(\gamma(t), \dot \gamma(t),t) 
-c\cdot \dot \gamma(t) +\alpha(c) dt.
$$
A curve $\theta (t):\Rm\lto \Tm^n$ is said {\it calibrated} by $u$
if 
$$
u(\theta(T))-u(\theta(S))=  \int_S^T L(\theta(t), \dot \theta(t),t) 
-c\cdot \dot \theta(t) +\alpha(c) dt,
$$
for each $S<T$ in $\Zm$.
The curve $\theta(t)$ is then the projection of a Hamiltonian  orbit $(\theta(t),p(t))$, 
such an orbit is called a {\it calibrated orbit}. We denote by 
$$\tilde \mI (u,c)\subset \Tm^n \times \Rm ^n$$ 
the union on all calibrated orbits $(\theta,p)(t)$ of the sets $(\theta,p)(\Zm)$, or equivalently of the sets $(\theta,p)(0)$.
In other words, these are the initial conditions the orbits of which are calibrated by $u$. 
By definition, the set $\tilde \mI(u,c)$ is invariant under the time one Hamiltonian flow $\varphi$,
it is moreover compact and not empty.
We also denote by 
$$s\tilde \mI (u,c)\subset  \Tm^n \times \Rm ^n\times \Tm$$ 
the suspension  of $\tilde \mI (u,c)$, or in other words the set of points of the form
$((\theta(t), p(t),t\mod 1)$ for each $t\in \Rm$ and each calibrated orbit $(\theta,p)$.
The set $s\tilde \mI(u,c)$ is compact and invariant under the extended Hamiltonian flow.
Note that $s\tilde \mI(u,c)\cap \{t=0\}= \tilde \mI(u,c)\times \{0\}$.
The projection 
$$
\mI(u,c)\subset \Tm^n
$$
of $\tilde \mI(u,c)$ on $\Tm^n$ is  the union of points $\theta(0)$ where $\theta$ is a calibrated curve.
The projection 
$$
s\mI(u,c)\subset  \Tm^n \times \Tm
$$
of $s\tilde \mI(u,c)$ on $\Tm^n\times \Tm$ is  the union of points $(\theta(t), t\mod 1)$ where $t\in \Rm$ and $\theta$ is a calibrated curve.
It is an important result of Mather theory that $s\tilde \mI(u,c)$ is a Lipschitz graph above $s\mI(u,c)$ 
(hence $\tilde \mI(u,c)$ is a Lipschitz graph above $\mI(u,c)$ ).
We finally  define the Aubry and Ma\~n\'e sets by 
\begin{align} \label{aubry-mane-sets}
\tilde \mA(c)=\cap_u \tilde \mI(u,c)\ , \ s\tilde \mA(c)=\cap_u s\tilde \mI(u,c)\ , \ \tilde \mN(c)=\cup_u\tilde \mI(u,c)
\ , \  s\tilde \mN(c)=\cup_u s\tilde \mI(u,c),
\end{align}
where the union and the intersection are taken on the set  of all weak KAM solutions
$u$ at cohomology $c$. 
When a clear distinction is needed, we will call the sets $s\tilde \mA (c), s\tilde \mN (c)$ the suspended
Aubry (and Ma\~né) sets.
We denote by $s\mA(c)$ and $s\mN(c)$ the projections on $\Tm^n\times \Tm$,
of $s\tilde \mA(c)$ and $s\tilde \mN(c)$. Similarly,  $\mA(c)$ and $\mN(c)$ are the projections on $\Tm^n$
of $\tilde \mA(c)$ and $\tilde \mN(c)$.
The Aubry set $\tilde \mA(c)$ is  compact, non-empty and invariant under the time one flow. It is 
a Lipschitz graph above the projected Aubry set $\mA(c)$.
The Ma\~n\'e set $\tilde \mN(c)$ is compact and invariant.
Its orbits (under the time-one flow) either belong, or are bi-asymptotic, to 
$\tilde \mA(c)$.

In \cite{Be1}, an equivalence relation is introduced 
on the cohomology $H^1(\Tm^n,\Rm)=\Rm^n$, called {\it forcing relation}.
It will not be useful for the present exposition to recall
the precise definition of this forcing relation. What is important is that,
if $c$ and $c'$ belong to the same forcing class, then there exists an orbit
$(\theta(t),p(t))$ and an integer $T\in \Nm$ such that $p(0)=c$ and $p(T)=c'$. 
We will establish here that, in the presence of generic additional
assumptions, the resonant arc $\Gamma_1$ is contained in a forcing class,
which implies the conclusion of Theorem \ref{mainnormal}, but also the existence
of various types of orbits, see \cite{Be1}, Section 5, for more details.
To prove that $\Gamma_1$ is contained in a forcing class, it is enough 
to prove that each of its points is in the interior of its forcing class. 
This can be achieved using the mechanisms exposed in \cite{Be1}, 
called the Mather mechanism and the Arnold mechanism, under 
appropriate  informations on the sets 
$$
\tilde \mA(c)\subset  \tilde \mI(u,c)\subset \tilde \mN(c),\qquad 
c\in \Gamma_1.
$$

\subsubsection{Localization and a graph theorem}
The first step is to relate these sets to the normally hyperbolic cylinder
$\mC_0$ as follows:

\begin{thm}\label{intro-local}
In the context of Theorem~\ref{intro-nhic-mult}, 
we can assume by possibly reducing the constant $\delta>0$ that the following
 additional property holds for each function $R\in \mR(\Gamma_1,\epsilon, \delta)$ 
with $\epsilon\in ]0,\delta[$:

 For each $c\in \Gamma_1$, the Ma\~n\'e set $\tilde \mN(c)$ is contained in the 
 cylinder $\mC_0$.
 Moreover, the restriction of the coordinate map $\theta^f:\Tm^n\times \Rm^n\lto \Tm$ 
 to $\tilde \mI(u,c)$
 is a Bi-Lipschitz homeomorphism for each Weak KAM solution $u$ at cohomology $c$.
\end{thm}

\proof
The proof is based on estimates on 
Weak KAM solutions that will be established
in Section \ref{sec:localization}.
Let $\kappa$ be as given by Theorem~\ref{intro-nhic-mult}.
 Theorem \ref{thm-loc} 
(which is stated and proved in Section \ref{sec:localization})
implies that  
the suspended Ma\~ne set $s\tilde\mN(c)$ is contained in the set 
$$
\{ \|\theta^s - \theta^s_*(c^f)\| \le \kappa, \quad
\|p^s - p^s_*(c^f)\| \le \kappa \sqrt{\epsilon} , 
\quad
|p^f-c^f|\leq \kappa \sqrt{\epsilon}
\}
$$
provided $R\in \mR(\Gamma_1,\epsilon, \kappa^{16})$  and 
 $\epsilon\in ]0,\epsilon_0[$ (a constant depending on $\kappa$).
As a consequence, this inclusion holds for $R\in \mR(\Gamma_1,\epsilon, \delta)$ 
and $\epsilon \in ]0,\delta[$, with $\delta =\min(\kappa^{16}, \epsilon_0)$.
The suspended Ma\~n\'e set  
$s\tilde \mN(c)$ is then contained in the domain called $W$ in the statement 
of Theorem \ref{intro-nhic-mult}. 
It is thus contained in $\mC$, hence $\tilde \mN(c)\subset \mC_0$.

Let us consider a Weak KAM solution $u$ of $N_{\epsilon}$ at cohomology $c$
and prove the projection part of the statement.
Let $(\theta_i,p_i), i=1,2$ be two points
in $\tilde \mI(u,c)$.
By Theorem \ref{thm-lip}, we have
$$\|p_2-p_1\|\leq 9 \sqrt{D\epsilon}\|\theta_2-\theta_1\|
\leq 9\sqrt{D\epsilon}(\|\theta_2^f-\theta_1^f\|+\|\theta_2^s-\theta_1^s\|).
$$
Since the points belong to $\mC_0$, the last estimate in 
Theorem \ref{intro-nhic-mult} implies that
$$
\|\theta_2^s-\theta_1^s\|\leq C({1+\sqrt{\delta/\epsilon}})
(\|\theta_2^f-\theta_1^f\|+\|p_2-p_1\|).
$$
We get
$$
\|p_2-p_1\|\leq {9C\sqrt{D}}\big(2\sqrt{\epsilon}+\sqrt{\delta}\big)
\|\theta_2^f-\theta_1^f\|+
{9C\sqrt{D}} \big(\sqrt{\epsilon}+\sqrt{\delta}\big)\|p_2-p_1\|.
$$
If $\delta$ is small enough and $\epsilon<\delta$, then  
$$
{9C\sqrt{D}} \big(\sqrt{\epsilon}+\sqrt{\delta}\big)\leq
{9C\sqrt{D}} \big(2\sqrt{\epsilon}+\sqrt{\delta}\big)\leq
\frac{1}{2}
$$
hence 
$$
\|p_2-p_1\|\leq {9C\sqrt{D}}\big(2\sqrt{\epsilon}+\sqrt{\delta}\big)\|\theta_2^f-\theta_1^f\|+
 \frac{1}{2}\|p_2-p_1\|,
$$
thus 
$$
\|p_2-p_1\|\leq 
{9C\sqrt{D}}\big(4\sqrt{\epsilon}+2\sqrt{\delta}\big)\|\theta_2^f-\theta_1^f\|
\leq \|\theta_2^f-\theta_1^f\|.
$$
\qed

{\subsubsection{Structure of Aubry sets inside the cylinder and 
existence of diffusing orbits}}
This \text{last} result, in conjunction with the theory of circle 
homeomorphisms, has strong consequences:

All the orbits of $\tilde \mA_0(c)$ have the same rotation number 
$\rho(c)=(\rho^f(c),0)$, with $\rho^f(c)\in \Rm$.
Since the sub-differential $\partial \alpha(c)$ of the convex function 
$\alpha$ is the rotation set of $\tilde \mA(c)$, we conclude that the function 
$\alpha$ is differentiable at each point of $\Gamma_1$,
with $d\alpha(c)=(\rho^s(c),0)$.

When $\rho^s(c)$ is rational,  the Mather minimizing measures are supported 
on periodic orbits.

When $\rho^s(c)$ is irrational, the invariant set $\tilde \mA(c)$ 
is uniquely ergodic. As a consequence, there  exists one and only one weak KAM solution
(up to the addition of an additive constant), hence
$\tilde \mN(c)=\tilde \mA(c)$.

In the  irrational case, we will have to consider homoclinic orbits.
Such orbits can be dealt with by
considering  the two-fold covering
\begin{align*}
 \xi:\Tm^n&\lto \Tm^n\\
\theta=(\theta^f, \theta^s_1,\theta^s_2,\cdots,\theta^s_{n-1})
&\lmto \xi(\theta)=
 (\theta^f, 2\theta^s_1,\theta^s_2,\cdots,\theta^s_{n-1}).
\end{align*}
The idea of using a covering to study homoclinic orbits comes from Fathi, see 
\cite{Fa2}.
This covering lifts to a symplectic covering
\begin{align*}
\Xi:\Tm^n\times \Rm^n&\lto \Tm^n\times \Rm^n\\
(\theta,p)=(\theta, p^f, p^s_1,p^s_2,\ldots,p^s_{n-1})
&\lmto \Xi(\theta,p)=
(\xi(\theta),p^f, p^s_1/2,p^s_2,\ldots,p^s_{n-1}),
 \end{align*}
and we define the lifted Hamiltonian $\tilde N=N\circ \Xi$.
It is known, see \cite{Fa2, CP, Be1} that
$$
\tilde \mA_{H\circ \Xi}(\xi^*c)=\Xi^{-1}\big(\tilde \mA_{H}(c)\big)
$$
where $\xi ^*c =(c^f,c^s_1/2,c^s_2,\ldots,c^s_{n-1})$.
On the other hand, the inclusion
$$
\tilde \mN_{N\circ \Xi}(\xi^*c)\supset \Xi^{-1}\big(\tilde \mN_{N}(c)\big)
=\Xi^{-1}\big(\tilde \mA_{N}(c)\big)
$$
is not an equality.
More precisely, in the present situation,   the set $\tilde \mA_{ N\circ \Xi}(\tilde c)$
is the union of two disjoint homeomorphic copies of the circle
 $\tilde \mA_{ N}(\tilde c)$, and $\tilde \mN_{N\circ \Xi}(\tilde c)$
contains heteroclinic connections between these copies (which are the liftings of orbits 
homoclinic to $\tilde \mA_N(c)$).
More can be said if we are allowed to make a small perturbation 
to avoid degenerate situations.
We recall that a metric space is called totally disconnected if its
only connected subsets are its points.
The hypothesis of total disconnectedness in the following statement can be seen 
as a weak form of transversality of the stable and unstable manifolds
of the invariant circle $\tilde \mA_N(c)$. 

\begin{thm}\label{intro-gen}
In the context of Theorems \ref{intro-nhic-mult} and \ref{intro-local},
the following
property holds for a  dense subset of functions $R\in \mR(\Gamma_1,\epsilon, \delta_0)$
(for the $C^r$ topology):
Each $c\in \Gamma_1$ is in one of the following cases:
\begin{enumerate}
 \item $\theta^f(\mI(u,c))\subsetneq \T$ for each weak KAM solution
 $u$ at cohomology $c$.
 \item $\rho(c)$ is irrational,  $\theta^f(\mN_N(c))=\T$ (hence, $\tilde \mN_N(c)$
 is an invariant circle), and 
  $\tilde \mN_{N\circ \Xi}(\xi^*c)-\Xi^{-1}(\tilde \mN_{N}(c))$ is 
totally disconnected.
\end{enumerate}
The arc $\Gamma_1$ is then contained in a forcing class, hence the conclusion 
of Theorem \ref{mainnormal} holds.
\end{thm}

\proof 
By general results on Hamiltonian dynamics, the set 
$\mR_1\subset \mR(\Gamma_1,\epsilon, \delta_0)$ of functions 
$R$ such that the flow map $\phi$ does not admit any non-trivial 
invariant circle of rational rotation number is $C^r$-dense.
This condition  holds for example if $N$ is Kupka Smale (in 
the Hamiltonian sense, see \cite{RR} for example).

Since the coordinate map $\theta^f$
is a homeomorphism in restriction to $\tilde \mI(u,c)$,
this set is an invariant circle if  $\theta^f(\mI(u,c))=\T$.
If $R\in \mR_1$, this implies that the rotation number $\rho^f(c)$
is irrational.
In other words, for $R\in \mR_1$, condition 1 can be violated only
at points $c$
when $\rho^f(c)$ is irrational, and then $\tilde \mI(u,c)=\tilde \mA(c)=\tilde \mN(c)$
is an invariant circle.

When $R\in \mR_1$, it is possible to perturb $R$ away from $\mC_0$ in such a way that 
$\tilde \mN_{N\circ \Xi}(\xi^*c)-\Xi^{-1}(\tilde \mN_{N}(c))$
is totally disconnected for each value of $c$ such that $ \tilde \mN(c)$
is an invariant circle. This second perturbation procedure is not easy because
there are uncountably many such values of $c$. This is the result of Theorem~\ref{thm:totally-disc}.  A result of this kind was obtained in \cite{CY2}, here we give a self-contained proof with many new ingredients, see Section~\ref{sec:variational}.

We now explain, under the additional condition (1 or 2),
how the variational mechanisms of \cite{Be1}
can be applied to prove that $\Gamma_1$
is contained in a forcing class.
It is enough to prove that each point $c\in \Gamma_1$
is in the interior of its forcing class. We treat separately the two cases.

In the first case, we can apply the Mather mechanism, see (0.11) in \cite{Be1}.
In that paper, the subspace $Y(u,c)\subset \Rm^n$, defined as the set of cohomology classes
of closed one-forms whose support is disjoint from $\mI(u,c)$, is associated to each 
weak KAM solution $u$ at cohomology $c$
(in \cite{Be1}, the notation $R(\mG)$ is used).
In the present case, we know that the map $\theta^f$
restricted to $\tilde \mI(u,f)$ is a bi-Lipschitz homeomorphism which is not onto.
We conclude that $Y(u,c)=\Rm^n$.
Since this holds for each weak KAM solution $u$, we conclude  that
$$
Y(c):= \cap_u Y(u,c)=\Rm^n.
$$
The result called  Mather mechanism in \cite{Be1}
states that there is a small ball $B\subset Y(c)$ centered at $0$ in $Y$ such that
the forcing class of $c$ contains $c+B$. In the present situation,
we conclude that $c$ is in the interior of its forcing class.

In the second case, we can apply the Arnold's Mechanism, see Section 9 in \cite{Be1}.
We work with the Hamiltonian $N\circ \Xi$ lifted to the two-fold cover.
By Proposition (7.3) in \cite{Be1}, it is enough to prove that
 $\xi^*c$ is in the interior  of its forcing class for the lifted Hamiltonian
$N\circ \Xi$;
this implies that $c$ is in the interior of its forcing class for $N$.

The preimage 
$\Xi^{-1}\big(\tilde \mN_{N}(c)\big)$ is the union of two closed curves $\tilde \mS_1$
and $\tilde \mS_2$. The set $\tilde \mN_{N\circ \Xi}(\xi^*c)$
contains these two curves, as well as a set $\tilde \mH_{12}$ of 
heteroclinic connections from $\tilde \mS_1$ to $\tilde \mS_2$, and a set 
$\tilde \mH_{21}$ of heteroclinic connections
from $\tilde \mS_2$ to $\tilde \mS_1$.
Theorem (9.2) in \cite{Be1} states that $\xi^*c$ is in the interior of its
forcing class provided  $\tilde \mH_{12}$
and $\tilde \mH_{21}$ are totally disconnnected.
Actually, the hypothesis is stated in \cite{Be1} in a slightly different way,
we explain in Appendix \ref{sec:disc} that total disconnectedness actually
implies the hypothesis of \cite{Be1}.
We conclude 
that each $c\in \Gamma_1$ is in the interior of its forcing class.
Since $\Gamma_1$ is connected, it is   contained in a single forcing class.
It is then a simple consequence of the definition of the forcing relation,
see \cite{Be1}, Section 5, that the conclusion of Theorem \ref{mainnormal} holds.
This ends the proof of Theorem \ref{mainnormal}, using the results proved 
in the rest of the paper.
\qed

\subsection{Bifurcation points and a longer diffusion path}
\label{sec:bifurcation}


This section discusses some improvements on  Theorems \ref{main} and \ref{mainnormal}.
There are two limitations to the size of the resonant arc 
$\Gamma_1\subset \Gamma$ to which the above construction can be applied.

The first limitation comes from the assumption that hypothesis (\ref{HZl}) 
should hold on $\Gamma_1$. Given a resonant arc $\Gamma_2\subset \Gamma$,  
it is generic to satisfy this condition on a certain subarc 
$\Gamma_1\subset \Gamma_2$, but it is not generic to satisfy (\ref{HZl}) 
on the whole of $\Gamma_2$. The presence of values of $c\in \Gamma_2$ 
such that $Z(.,c)$ has two nondegenerate maxima can't be excluded.
%
In this section, we explain how a  modification  on the 
proof of Theorem \ref{mainnormal} allows to get rid of this limitation.

The second limitation comes from the normal form theorem, and from 
the impossibility to incorporate a finite set of additional resonances 
(punctures) in the domain of our normal forms. This limitation  is serious, 
and bypassing it would require a specific work
around additional resonances which will not be discussed here.
Some preprints on this issue appeared after the first version of
the present works, see \cite{C,KZ1,KZ2} (the latter ones being sequels to 
the present work, and the first one is independent).
Here, the best we can achieve is to prove existence of diffusion orbits between 
two consecutive punctures. The number of punctures 
is independant from $\epsilon$, it depends on the parameter 
$\delta$ in Theorem~\ref{mainnormal}, which can be computed using 
the non-degeneracy parameter $\lambda$, see Remark~\ref{length}.

In order to get rid of the first limitation, we consider a second hypothesis on $Z$:
\begin{hyp}\label{HZ2}
 There exists a  real number $\lambda>0$ and two points 
$\vartheta_1^s,\vartheta_2^s$ in 
 $\Tm^{n-1}$ such that the balls $B(\vartheta^s_1,3\lambda)$ and
 $B(\vartheta^s_1,3\lambda)$ are disjoint and such that, for each $p\in \Gamma_1$,
 there exists two local maxima
$\theta^s_1(p)\in B(\vartheta^s_1,\lambda)$ and 
$\theta^s_2(p)\in B(\vartheta^s_2,\lambda)$ of the function $Z(.,p)$ 
in $ \Tm^{n-1}$ satisfying  
\begin{align*}
\partial^2_{\theta^s} Z(\theta^s_1(p),p)\leq \lambda I
\quad,\quad 
\partial^2_{\theta^s} Z(\theta^s_2(p),p)\leq \lambda I,\qquad \qquad \qquad \\
 Z(\theta^s,p) \le \max \{Z(\theta^f_1(p),p),Z(\theta^f_2(p),p)\}- \lambda
\big(\min\{d(\theta^s-\theta^s_1), d(\theta^s-\theta^s_2)\}\big)^2
\end{align*}
for each $p\in \Gamma_1$ and each $\theta^s\in \Tm^{n-1}$.
\end{hyp}

Given an arc $\Gamma_2\in \Rm^n$, the following property is generic 
in $C^r(\Tm^{n-1}\times \Rm^n,\Rm)$:

The arc $\Gamma_2$ is a finite union of subarcs such that either Hypothesis 
\ref{HZ} or Hypothesis \ref{HZ2} holds on each of these subarcs, with a common constant $\lambda>0$.

We have the following improvement on Theorem \ref{mainnormal}:

\begin{prop}\label{bifurcation-genericity}
For the system (\ref{eq:normal-form}), assume that there exists $\lambda >0$ such that  for each $c\in \Gamma_1$, 
  either Hypothesis 1 or 2 hold for each $c\in \Gamma_1$. Then there exists  $\delta>0$, 
  which depend only on $n$, $H_0$, 
and $\lambda$, and such that, for each $\epsilon \in ]0, \delta[$, the following
property holds for a  dense subset of functions $R\in \mR(\Gamma_1,\epsilon, \delta)$
(for the $C^r$ topology):

There exists an orbit $(\th(t),p(t))$ 
and an integer $T\in \Nm$ such that 
$p(0)=p_*(a_-)$ and 
$p(T)=p_*(a_+)$.
\end{prop}

\begin{proof}[Proof of Proposition~\ref{bifurcation-genericity}]
We use the same framework as in the proof of Theorem \ref{mainnormal},
so it is enough to prove that each element of $\Gamma_1$ is in the interior of its forcing class.

Observe first that 
Theorem \ref{nhic-mult} can be applied to prove the existence
of two invariant cylinders $\mC^1$ and $\mC^2$ in the extended phase space
$\Tm^n\times \Rm^n \times \Tm$.
Moreover, we can chose the parameter $\kappa$ smaller than $ \lambda$, 
in such a way that 
$$
\theta^s(\mC_1)\subset  B(\vartheta^s_1,2\lambda)
\quad, \quad 
\theta^s(\mC_2)\subset  B(\vartheta^s_2,2\lambda).
$$
As earlier,
we denote by $\mC^1_0$ and $\mC^2_0$
the intersections with the section $\{t=0\}$.
By  Theorem \ref{thm-double}, we have
$$
\tilde \mA(c)\subset \mC^1_0\cup \mC^2_0
$$
for each $c\in \Gamma_1$.
Let us now introduce two smooth  functions $F_i(\theta^s):\Tm^{n-1}\lto [0,1]$,
$i\in \{1,2\}$, with the property that 
$F_1=1$  in $  B(\vartheta^s_2,2\lambda)$,
$F_1=0$ outside of $  B(\vartheta^s_2,3\lambda)$, 
$F_2=1$  in $  B(\vartheta^s_1,2\lambda)$
and $F_2=0$ outside of $  B(\vartheta^s_1,3\lambda)$

Considering the modified Hamiltonians $N- F_i$ will help the description 
of the Mather sets of $N$. One can check by inspection in the proofs 
(using that $F_i$ does not depend on  $p$) that Theorem~\ref{thm-loc} applies to $N-F_i$, and allows to conclude 
that the Ma\~n\'e set $\tilde \mN_i(c)$ of $N-F_i$ is contained in $\mC^i_0$.
Let us denote by $\alpha_i(c)$ the $\alpha$ function of $N-F_i$. 
These objects are closely related  to Mather's local Aubry sets.

\begin{lem}
For each $c\in \Gamma_1$, $\alpha_i(c)$ are differentiable 
at $c$, and $ \alpha(c)=\max\{\alpha_1(c), \alpha_2(c)\}.$
 Moreover,
 \begin{itemize}
  \item If $\alpha(c)=\alpha_1(c)>\alpha_2(c)$, then $\tilde \mN(c)=\tilde \mN_1(c)$,
  \item If $\alpha(c)=\alpha_2(c)>\alpha_1(c)$, then $\tilde \mN(c)=\tilde \mN_2(c)$,
  \item If $\alpha(c)=\alpha_1(c)=\alpha_2(c)$, then 
  $\tilde \mN_1(c)\cup\tilde \mN_2(c)\subsetneq \tilde \mN(c)$.
 \end{itemize}

\end{lem}

\proof
The functions $\alpha_i(c)$ are $C^1$ for the same reason as $\alpha(c)$
is $C^1$ in the one peak case.

Since $N-N_i\leq N$, we have $\alpha_i(c)\leq \alpha(c)$.
On the other hand, we know that  
$$
\alpha(c)=\max _{\mu} \big( c\cdot \rho(\mu)-\int p \partial_p N-N d\mu\big),
$$
where the minimum is taken on the set of invariant measures $\mu$.
Since we know that $
\tilde \mA(c)\subset \mC^1_0\cup \mC^2_0
$,
and since the maximizing measures are supported on the Aubry set, 
we conclude that each ergodic maximizing measure is supported either on $\mC^1$
or on $\mC^2$. If the measure is supported in $\mC^i$,
then we have 
$$
\alpha_i(c)\geq c\cdot \rho(\mu)-\int p \partial_p N-N+F_i d\mu=
c\cdot \rho(\mu)-\int p \partial_p N-N d\mu=\alpha(c).
$$
This proves the equality 
$\alpha(c)=\max \{\alpha_1(c),\alpha_2(c)\}.
 $

As is explained in the proof of Theorem~\ref{thm-double},
there are two possibilities for the Ma\~n\'e set $\tilde \mN(c)$:
either it is contained in one of the $\mC^i_0$, or it intersects both of them,
and then also contains connections (because it is necessarily chain transitive).

 If the Ma\~n\'e set $\tilde \mN(c)$ intersects $\mC^i_0$, then the intersection
is a compact  invariant set, which thus support an invariant measure. 
This measure must be maximizing the functional 
$c\cdot \rho(\mu)-\int p \partial_p N-N d\mu$,
and thus also the functional 
$c\cdot \rho(\mu)-\int p \partial_p N-N +F_i d\mu$.
As a consequence, we must have $\alpha(c)=\alpha_i(c)$.
\qed

 We can prove by the variational mechanisms of \cite{Be1} that a point $c$ 
 is in the interior of its forcing class in the following  three cases:
 
First case, the  Ma\~n\'e $\tilde \mN(c)$ set is contained in one of  
the cylinders $\mC^i_0$, and it does not contain any invariant circle. 
Then the Mather mechanism applies as in the single peak case, and  
$c$ is contained in the interior of its forcing class.
 
Second case, the Ma\~n\'e set is an invariant circle (then necessarily 
contained in one of the cylinders $\mC^i_0$), it is uniquely ergodic, and
$\tilde\mN_{N\circ \Xi}(c)-\Xi^{-1} (\tilde \mN(c))$ is totally disconnected.
Then the Arnold's mechanism  applies as in the single peak case, 
and $c$ is contained in the interior of its forcing class.
 
Third case, the sets $\tilde \mN_i(c)$ are both non-empty and uniquely 
ergodic, and $\tilde\mN(c)-\big( \tilde\mN_1(c)\cup \tilde\mN_2(c)\big)$ 
is totally disconnected. Then the Arnold's mechanism applies directly 
(without taking a cover), and $c$ is contained in the interior of its forcing class.
 
Each $c\in \Gamma_1$ is in one of these three cases provided 
the following set of additional conditions holds:

\begin{itemize}
 \item The sets $\tilde \mN_i(c)$ are uniquely ergodic.
 \item The equality $\alpha_1(c)=\alpha_2(c)$ has finitely many solutions on 
 $\Gamma_1$.
 \item The set $\tilde \mN(c)-\big(\tilde \mN_1(c)\cup \tilde \mN_2(c)\big)$ 
 is totally disconnected (and not empty) when $\alpha_1(c)=\alpha_2(c)$.
 \item The set $\tilde\mN_{N\circ \Xi}
 (c)-\Xi^{-1} (\tilde \mN(c))$ is totally disconnected whenever
 $\tilde \mN(c)$ is an invariant circle.
 \end{itemize}

 Let us now explain how these conditions can be imposed by  a  $C^r$
 perturbation of $R$.
 
 We first consider a perturbation $R_1$ of $R$ such that, for each rational number 
 $\rho \in \Qm\times \{0\}$, there exists a unique Mather minimizing measure
 of rotation number $\rho$. Such a condition is known to be generic
 (because it concerns only countably many rotation numbers)
 see \cite{Mn,CP,BC,Be7}.
 
 We then consider a perturbation $R_2$ of the form $R_1-s F_1$, with a small $s>0$.
 It is easy to see that the functions $\alpha^2_i(c), c\in \Gamma_1$ associated to 
 the Hamiltonian $H_0+\epsilon Z+\epsilon R_2$ are 
 $$
 \alpha^2_1(c)=\alpha^1_1(c)
 \quad,\quad
 \alpha^2_2(c)=\alpha^1_2(c)+s
 $$
 where $\alpha_i^1(c)$ are the functions associated to $H_0+\epsilon Z+\epsilon R_1$.
 By Sard's theorem, there exist arbitrarily small regular values
 $s$ of the difference $\alpha^1_1-\alpha_2^1$. If $s$ is such a value,
 then $0$ is a regular value of the difference $\alpha^2_1-\alpha_2^2$, hence
 the equation $\alpha^2_1(c)=\alpha_2^2(c)$ has only finitely many solutions on $\Gamma$.
 Note that the perturbation is locally constant around the cylinders 
 $\mC^i$, hence  this  second perturbation does not destroy the first property.
 
 We then perform  new perturbations supported away from $\mC^i$, which 
 preserve the first two properties.
 The third property is not hard to obtain since it now concerns only finitely many values
 of $c$. The last property is obtained using arguments of Section \ref{sec:variational}.

 We have proved that the Hamiltonian $R$ can be perturbed in such a way that 
 each point of $\Gamma_1$ is in the interior of its forcing class.
 \end{proof}
 
\section{Normal forms}\label{sec:normal-form}

The goal of the present section is to prove Proposition~\ref{prop-nf}
which allows to reduce Theorem \ref{main}
to Theorem \ref{mainnormal}.
This reduction to the normal form does not use the convexity assumption.
 We put  the initial Hamiltonian $H_{\epsilon}$ in 
normal form around a compact subarc $\Gamma_2$ of the resonance
$$\Gamma=\{p^s=p_*(p^f)\}=\{(p\in \Rm^n, \partial_{p^s}H_0=0\}.
$$
{This global normal form is obtained by using mollifiers to glue local normal 
forms that depends on the arithmetic properties of the frequencies. This allows 
a simpler proof for instability, as we avoid the need to justify transitions between 
different local coordinates. }

{ Recall that study a resonance of order $n-1$ or, equivalently, of codimension $1$.  
The resonance of order $n-1$ is given by a lattice 
$\Lambda$ span by $n-1$ linearly independent vectors 
$k_1,\dots,k_{n-1}\in (\Zm^n\setminus 0) \times \Zm$. 
Denote by $\theta^s_j=k_j \cdot \theta, 
\ \omega^s_j=k_j \cdot \nabla H_0(p),\, j=1,\dots,n-1,$  
and $\theta^s=(\theta^s_1,\dots,\theta^s_{n-1})$ 
the slow angles and by $\om^s=(\om^s_1,\dots,\om^s_{n-1})$ the slow actions resp.
Choose a complement angle 
$\theta^f$ so that 
$(\theta^s,\theta^f)\in \Tm^{n-1}\times \Tm$ form a basis.
}

For  $p\in \Gamma$ we have  $\omega(p)=(0, \partial_{p^f}H_0(p))$.
 We say that
 $p$ has an additional resonance if the remaining frequency
$\partial_{p^f}H_0(p)$ is rational.
In order to reduce the system to an appropriate normal form,
we must remove some additional resonances.
More precisely, we denote by  $\mD(K,s)\subset B$  the set of momenta $p$ such that
\begin{itemize}
 \item $\|\partial_{p^s} H_0(p)\|\leq s$, and
\item $|k^f\partial_{p^f} H_0(p) +k^t|\geq 3Ks\quad$ for each 
$(k^f,k^t)\in \Zm^2$ satisfying $\max(|k^f|,|k^t|)\in ]0,K]$.
\end{itemize}

The following result, which does not use the convexity of $H_0$,
 is a refinement of 
Proposition~\ref{prop-nf}:

\begin{thm}\label{normal-form}[Normal Form]
Let $H_0(p)$ be a $C^4$ Hamiltonian.
For each $\delta \in ]0,1[$, there exist positive parameters $K_0, \epsilon_0, \beta$
such that, for each $C^4$ Hamiltonian $H_1$ with $\|H_1\|_{C^4}\leq 1$ and each
$
K\geq K_0$, 
$\epsilon \leq \epsilon_0,
$
there exists a  smooth change of coordinates
$$
\Phi :\Tm^n\times B \times \Tm
\lto \Tm^n\times \Rm^n\times \Tm
$$
satisfying $\|\Phi-id\|_{C^0}\leq \sqrt{\epsilon} $ and  $\|\Phi-id\|_{C^2}\leq \delta$
and such that, in the new coordinates, the Hamiltonian $H_0+\epsilon H_1$
takes the form
\begin{equation*}
 N_\epsilon =  H_0(p) + \epsilon Z(\theta^s, p) + \epsilon R(\theta, p, t),
\end{equation*}
with   $\|R\|_{C^2}\le \delta$
on $\Tm^n\times \mathcal{D}(K,\beta \epsilon^{1/4})\times \Tm$.
We can take $K_0=c\delta^{-2}, \beta=c\delta^{-1-n}, \epsilon_0=\delta^{6n+5}/c$,
where $c>0$ is some constant depending only on $n$ and $\|H_0\|_{C^4}$.
\end{thm}
The proof actually builds   a symplectic diffeomorphism
$\tilde \Phi$ of $\Tm^{n+1}\times \Rm^{n+1}$
of the form
$$
\tilde \Phi (\theta,p,t,e)=\big(\Phi(\theta, p,t),e+f(\theta,p,t)\big)
$$
and such that
$$
N_{\epsilon}+e=(H_{\epsilon}+e)\circ \tilde \Phi.
$$
We have the estimates $\|\tilde \Phi-id\|_{C^0}\leq \sqrt{\epsilon} $ and
 $\|\tilde \Phi-id\|_{C^2}\leq \delta$.

\begin{rmk}\label{length}[Distance between punctures]
  On the interval, the distance between 2 adjacent rationals 
with denominator at most $K$ is $1/K^2$.  Choose $K=K_0$ as in 
Theorem~\ref{normal-form}, the distance between adjacent punctures is 
at least  $D^{-1}/K^2\ge D^{-1}c^{-1}\delta^4$.

{The length of $\Gamma_1$ is determined by the choice of $\delta$, 
which can be chosen optimally in Theorem~\ref{intro-nhic-mult} and 
Theorem~\ref{thm-loc}. Upon inspection of the proof, it is 
not difficult to determine that $\delta$ can be chosen to a power of 
$\lambda$, which shows the distance between punctures is polynomial 
in $\lambda$. }
\end{rmk}

To prove Theorem~\ref{normal-form} we proceed in $3$ steps.
We first obtain a global normal form $N_{\epsilon}$
adapted to all resonances.
 We then show that this normal form takes the desired form on the domain
$
\mathcal{D}(K,s).
$
 However, the averaging procedure lowers smoothness,
in particular, the technique requires the smoothness $r\ge n+5$.
 To obtain a result that does not require this relation between $r$ and $n$,
we use a smooth approximation trick that goes back to Moser.

\subsection{A global normal form adapted to all resonances.}

We first state a result for autonomous systems.
The time periodic version will come as a corollary.
Consider the Hamiltonian $H_\epsilon(\phi, J)=H_0(J)+\epsilon H_1(\phi, J)$,
where $(\phi,J)\in \T^m \times \R^m$
(later, we will take $m=n+1$). Let $B=\{|J|\le 1\}$ be the unit ball in $\R^m$.
Given any integer vector $k\in \Z^m\setminus\{0\}$, let $[k]=\max\{|k_i|\}$.
To avoid zero denominators in some calculations, we make the unusual convention
that $[(0, \cdots, 0)]=1$.
We fix once and for all a bump function  $\rho:\R\to \R$ be a $C^\infty$
such that
$$\rho(x)=
\begin{cases}
  1, & |x|\le 1 \\
  0, & |x|\ge 2
\end{cases}
$$
and $0<\rho(x) < 1$ in between. For each  $\beta>0$ and $k\in \Z^m$,
we define the function
$\rho_k(J)=\rho(\frac{k\cdot \partial_J H_0}{\beta \epsilon^{1/4}[k]})$,
where $\beta>0$ is a parameter.

\begin{thm}\label{autonomous}
There exists a constant $c_m>0$, which depends only on $m$, such that the
following holds.
Given:
\begin{itemize}
 \item A $C^4$ Hamiltonian $H_0(J)$,
\item A $C^r$ Hamiltonian
$H_1(\varphi,J)$ with $\|H_1\|_{C^r}=1$,
\item Parameters $r\geq m+4$, $\delta \in ]0,1[$, $\epsilon\in ]0,1[$,
$\beta>0$, $K>0$,
\end{itemize}
satisfying
\begin{itemize}
 \item $K \geq c_m \delta^{\frac{-1}{r-m-3}}$,
\item $\beta \geq c_m (1+\|H_0\|_{C^4})\delta^{-1/2},$
\item $\beta\epsilon^{1/4}\leq \|H_0\|_{C^4}$,
\end{itemize}
there exists a $C^2$ symplectic diffeomorphism {$\Phi: \T^m \times B \to \T^m \times \R^m$} such that,
in the new coordinates, the Hamiltonian $H_{\epsilon}=H_0+\epsilon H_1$ takes the form
$$
H_{\epsilon}\circ \Phi=H_0+\epsilon R_1+\epsilon R_2
$$
with
\begin{itemize}
\item $R_1=\sum_{k\in \Z^m, |k|\le K}\rho_k(J)h_k(J)e^{2\pi i (k\cdot \phi)}$, 
here $h_k(J)$ is the $k^{th}$ coefficient for the Fourier expansion of $H_1$,
\item $\|R_2\|_{C^2}\le \delta$,
\item  $\|\Phi-id\|_{C^0}\le \delta\sqrt{\epsilon}$ and $\|\Phi-id\|_{C^2}\le \delta.$
\end{itemize}
If both $H_0$ and $H_1$ are smooth, then so is $\Phi$.
\end{thm}

We  now  prove Theorem~\ref{autonomous}.
To avoid cumbersome notations, we will denote by $c_m$ various different constants
depending only on the dimension $m$.
We have the following basic estimates about  the Fourier series of
a function $g(\phi, J)$. Given a multi-index
$\alpha=(\alpha_1, \cdots, \alpha_m)$, we denote
$|\alpha|=\alpha_1+\cdots+ \alpha_m$. Denote also
$\kappa_m=\sum_{\Zm^m}[k]^{-m-1}$.

\begin{lem}\label{fourier-est}
  For $g(\phi,J)\in C^r(\T^m\times B)$, we have
  \begin{enumerate}
  \item
If $l\leq r$, we have
$\|g_k(J)e^{2\pi i (k\cdot \varphi)}\|_{C^l}\le [k]^{l-r}\|g\|_{C^r}$.

  \item Let $g_k(J)$ be a series of functions such that the inequality
$\|\partial_{J^{\alpha}}g_k\|_{C^0}\leq M[k]^{-|\alpha|-m-1}$  holds for 
each multi-index $\alpha$ with $|\alpha|\le l$, for some $M>0$. Then, we have 
\newline
$\|\sum_{k\in \Z^m}g_k(J)e^{2\pi i(k\cdot\varphi)}\|_{C^l}\le c \kappa_m M$.
  \item Let $\Pi^+_Kg=\sum_{|k|>K}g_k(J)e^{2\pi i (k\cdot \phi)}$. Then for $l\leq r-m-1$,
we have   $\|\Pi^+_K g\|_{C^l}\le \kappa_m K^{m-r+l+1}\|g\|_{C^r}$.
  \end{enumerate}
\end{lem}
\begin{proof}
  1. Let us assume that $k\neq 0$ and take $j$
  such that $k_j=[k]$.
  Let $\alpha$ and $\eta$ be two multi-indices
  such that $|\alpha+\eta|\leq l$.
  Finally, let $b=r-l$, and let $\beta$ be the multi-index
   $\beta=(0,\ldots,0,b,0,\ldots,0)$, where $\beta_j=b$.
We have
$$
g_k(J)e^{2\pi i(k,\varphi)}=
\int_{\Tm^m} g(\theta, J)e^{2i\pi(k,\varphi-\theta)}d\theta
=
\int_{\Tm^m} g(\theta+\varphi, J)e^{-2i\pi(k,\theta)}d\theta,
$$
  hence
\begin{align*}
\partial_{\varphi^{\alpha}J^{\eta}}\big(
g_k(J)e^{2i\pi(k,\varphi)}\big)
&=
\int_{\Tm^m} \partial_{\varphi^{\alpha}J^{\eta}}g(\theta+\varphi, J)e^{-2i\pi(k,\theta)}d\theta,\\
&=
\int_{\Tm^m} \frac{\partial_{\varphi^{\alpha+\beta}J^{\eta}}
g(\theta+\varphi, J)}
{(2i\pi k_j)^b}
e^{-2i\pi(k,\theta)}d\theta.
\end{align*}
  Since $|\alpha+\beta+\eta|\leq r$, we conclude that
   $$
\|g_k(J)   e^{2i\pi(k,\varphi)}\|_{C^l}\leq
\|g\|_{C^r}/(2\pi[k])^b\leq \|g\|_{C^r}[k]^{l-r}.
   $$

2. We have
$\|g_k(J)e^{2i\pi(k\cdot \varphi)}\|_{C^l}\leq
$
$$\|\sum_{k\in \Z^m}h_k(J)e^{2\pi i(k\cdot\varphi)}\|_{C^l}
\le \sum_{k\in \Z^m}c_l|k|^{-r+l}M\le c_l \kappa_m M, $$
recall that $\kappa_m=\sum_{k\in \Z^m}|k|^{-m-1}$.

3. Using 1., we get
\begin{align*} \|\Pi^+_Kg\|_{C^l}
&\le
 \sum_{|k|>K} [k]^{l-r}\|g\|_{C^r}
 \le
 \|g\|_{C^r} K^{m-r+l+1}\sum_{|k|>K}[k]^{-m-1}\\
 &\le
 \|g\|_{C^r} K^{m-r+l+1}\sum_{k\in \Zm^m}[k]^{-m-1}.
\end{align*}
\end{proof}

\begin{proof}[Proof of Theorem~\ref{autonomous}]

Let $\tilde G(\phi,J)$ be the function that solves the cohomological equation
$$ \{H_0,\tilde G\}+H_1=R_1 +R_+,$$
where $R_+=\Pi^+_K H_1$.
Observing that $\rho_k(J)=1$
when $k\cdot \partial_JH_0=0$,
we have the following explicit formula for $G$:
{$$ 
\tilde G(\varphi,J)=
(2 \pi i)^{-1}\sum_{|k|\le K}
\frac{(1-\rho_k(J))h_k(J)}{ k\cdot \partial_JH_0(J)}e^{2\pi i (k\cdot \phi)} 
$$}
where each of the functions
$(1-\rho_k(J))h_k(J)/(k\cdot \partial_J H_0)$
is extended by continuity at the points where
the denominator vanishes. This function hence
takes the value zero at these points.
$G$ is well defined thanks to the smoothing terms $1-\rho_k$
we introduced, as whenever $k \cdot \partial_J H_0=0$ we also
have $1-\rho_k=0$ and that term is considered non-present.
Since $\tilde G$ as defined above is only $C^3$,
we will consider a smooth approximation
$$
G(\varphi,J)=\sum_{|k|\le K}g_k(J)e^{2\pi i (k\cdot \phi)} 
$$
where $g_k(J)$ are smooth functions which are sufficiently close to
{$\frac{(1-\rho_k(J))h_k(J)}{(2\pi i)k\cdot \partial_JH_0(J)}$}
in the $C^3$ norm.

Let $\Phi^t$ be the Hamiltonian flow generated by  $\epsilon G$.
Setting $F_t =  R_1 + R_+ + t(H_1-R_1-R_+)$,
we have the standard computation
\begin{align*}
\partial_t\big((H_0+\epsilon F_t)\circ \Phi^t)\big)
&=
\epsilon\partial_t F_t \circ \Phi^t+\epsilon\{H_0+\epsilon F_t,G\}\circ \Phi^t\\
&=\epsilon\big(\partial_tF_t+\{H_0,G\}\big)\circ \Phi^t+
\epsilon^2\{F_t,G\}\circ \Phi^t\\
&=\epsilon^2\{F_t,G\}\circ \Phi^t,
\end{align*}
from which follows that
$$ H_\epsilon \circ \Phi^1= H_0 + \epsilon R_1 + \epsilon R_+
+ \epsilon^2 \int_0^1 \{F_t, G\}\circ \Phi^t dt. $$
Let us estimate the $C^2$ norm
of the function
 $R_2 := R_+ + \epsilon \int_0^1 \{F_t, G\}\circ \Phi^t dt$.
It follows from Lemma~\ref{fourier-est} that
$$
\|R_+\|_{C^2}\le  \kappa_m K^{-r+m+2} \|H_1\|_{C^r}\le \frac12 \delta.
$$
We now focus on the term
  $ \int_0^1 \{F_t, G\}\circ \Phi^t dt$.
  To estimate the norm of $F_t$, it is convenient to write
  $F_t =\tilde F_t + (1-t)R_1$, where $\tilde F_t=(1-t)R_++tH_1$.
  Notice that the coefficients of the Fourier expansion of $\tilde F_t$
  is simply a constant times that of $H_1$, Lemma~\ref{fourier-est}
  then implies that
$$
\|\tilde F_t\|_{C^3}\le
\sum_{k\in{\Zm^m}} [k]^{3-r}\|H_1\|_{C^r}
=
 \kappa_m \|H_1\|_{C^r}
$$
provided that  $r\ge m+4$, where we set $\kappa_m=\sum_{\Zm^m}[k]^{-m-1}$.

We now have to estimate the norm of $R_1$ and $G$. This
requires additional estimates of the smoothing terms $\rho_k$ as well
as the small denominators $k\cdot \partial_JH_0$. We always assume
that $l\in \{0,1,2,3\}$ in the following estimates:
\begin{itemize}
\item[-]  $\rho_k(J)\ne 1\quad \Rightarrow \quad
|(k\cdot \partial_J H_0)^{-1}|\le \beta^{-1}\epsilon^{-1/4} |k|^{-1}$.
\item[-]
$\|(k\cdot \partial_JH_0)^{-1}\|_{C^l}\le c_m \beta^{-l-1}
\epsilon^{-(l+1)/4}\|H_0\|_{C^4}^{l+1}$ on $\{\rho_k\neq 1\}$.
\item[-]
$ \|\rho_k(J)\|_{C^l} \le c_m\beta^{-l} \epsilon^{-l/4} \|H_0\|_{C^4}^l$ and
$ \|1-\rho_k(J)\|_{C^l} \le c_m \beta^{-l} \epsilon^{-l/4} \|H_0\|_{C^4}^l.$
\end{itemize}
We have been using the following estimates on the derivative of
composition of functions: For $f:\R^m\to \R$ and $g: \R^m\to \R^m$ we have
$\|f\circ g\|_{C^l}\le c_{m,l}\|f\|_{C^l}(1+\|g\|_{C^l}^l)$.
\begin{itemize}
\item[-] For each multi-index $|\alpha|\le 3$, we have that
\[
\begin{aligned}&\|\partial_{J^\alpha} \left((1-\rho_k(J))h_k(J)(k\cdot \partial_JH_0)^{-1}\right) \|_{C^0} \\
\le & \sum_{\alpha_1+\alpha_2+\alpha_3=\alpha} \|1-\rho_k(J)\|_{C^{|\alpha_1|}}\|h_k\|_{C^{|\alpha_2|}}\|(k\cdot \partial_JH_0)^{-1}\|_{C^{|\alpha_3|}
(\{\rho_k\neq 1\})} \\
\le&  c_m \sum_{\alpha_1+\alpha_2+\alpha_3= \alpha}  \Big(
\beta^{-|\alpha_1|}\epsilon^{-|\alpha_1|/4} \|H_0\|_{C^4}^{|\alpha_1|}
\cdot [k]^{-r+|\alpha_2|} \|H_1\|_{C^r} \\
&\cdot \beta ^{-|\alpha_3|-1}\epsilon^{-(|\alpha_3|+1)/4}
\|H_0\|_{C^4}^{|\alpha_3|+1}\Big) \\
\le & c_m \beta^{-|\alpha|-1}\epsilon^{-(|\alpha|+1)/4}  [k]^{|\alpha|-r}
 \|H_0\|_{C^4}^{|\alpha|+1} \|H_1\|_{C^r} .
\end{aligned}
\]
\end{itemize}
In these computations, we have used the hypothesis
$\beta \epsilon^{1/4}\leq \|H_0\|_{C^4}$.
Since $G(\varphi, J)=\sum_{k\in \Z^m}(1-\rho_k(J))h_k(J)(k\cdot \partial_JH_0)^{-1}e^{2\pi i (k\cdot \varphi)}$,   Lemma~\ref{fourier-est} implies (since $r\geq m+1$) :
\begin{itemize}
\item[-]
 $\|G\|_{C^l}\le  c_m \beta^{-l-1}\epsilon^{-(l+1)/4}
\|H_0\|_{C^4}^{l+1} \|H_1\|_{C^r}\leq \epsilon^{-1} . $
 \end{itemize}
We now turn our attention to $R_1=\sum_{|k|\le K}\rho_k(J)h_k(J)e^{2i\pi (k\cdot \phi)}$:
 \begin{itemize}
 \item[-] $\|h_k\|_{C^l}\leq [k]^{l-r}\|H_1\|_{C^r}$.
\item[-]$
\|\rho_kh_k\|_{C^l}\le c_m \beta^{-l}  \epsilon^{-l/4}[k]^{-r+l}\|H_0\|_{C^4}^l\|H_1\|_{C^r}.
$
\item[-]
 $\|R_1\|_{C^l}\le c_m \beta^{-l}  \epsilon^{-l/4}\|H_0\|_{C^4}^l\|H_1\|_{C^r}$, 
 provided  $r\ge m+4$.
\end{itemize}
We obtain
 $$\|F_t\|_{C^l}\le \|R_1\|_{C^l}+\|\tilde F_t\|_{C^l}
\le c_m \beta^{-l}  \epsilon^{-l/4}\|H_0\|_{C^4}^l \|H_1\|_{C^r},
 $$
and
$$ \|\{F_t, G\}\|_{C^2}\le
\sum_{|\alpha_1+\alpha_2|\le 3}
 \|F_t\|_{C^{|\alpha_1|}}\|G\|_{C^{|\alpha_2|}} \le
 c_m \beta^{-4} \epsilon^{-1}\|H_0\|_{C^4}^4 \|H_1\|_{C^r}^2 .
$$
Concerning the flow $\Phi^t$,
we observe that $\|\epsilon G\|_{C^3}\leq 1$, and get
the following estimate
(see \textit{e. g.} \cite{DH}, Lemma 3.15):
\begin{itemize}
\item[-]$\|\Phi^t-id\|_{C^2}\le
c_m\epsilon\|G\|_{C^3}\le c_m \beta^{-4}
\|H_0\|_{C^4}^{4} \|H_1\|_{C^r}\leq \delta ,$
\item[-]
$\|\Phi^t-id\|_{C^0}\leq c_m \epsilon\|G\|_{C^1}\leq c_m \beta^{-2} \sqrt{\epsilon}
\|H_0\|_{C^4}^2\|H_1\|_{C^2}\leq \delta  \sqrt\epsilon .$
\end{itemize}
Finally,
we obtain
\begin{align*}
 \epsilon\|\{F_t, G\}\circ \Phi^t\|_{C^2}
&\le c_m \epsilon \|\{F_t, G\}\|_{C^2}\|\Phi^t\|_{C^2}^2
\\
&\le c_m \beta^{-4} \|H_0\|^4_{C^4} \|H_1\|_{C^r}^2 \le \delta/2.
 \end{align*}
\end{proof}

\subsection{Normal form away from additional resonances}

We now return to our non-autonomous system and apply
Theorem \ref{autonomous} around the resonance under study.
To the non-autonomous Hamiltonian
$$
H_{\epsilon}(\theta, p,t)=H_0(p)+\epsilon H_1(\theta,p,t):\Tm^n\times \Rm^n\times\Tm\lto \Rm
$$
we associate the autonomous Hamiltonian
$$
\tilde H_e(\varphi,J)=H_0(I)+e+\epsilon H_1(\theta, I,t):\Tm^{n+1}\times \Rm^{n+1} \lto \Rm,
$$
where $\varphi=(\theta, t)$ and $J=(I,e)$.
We denote the frequencies $\omega\in \Rm^{n+1}$  by
$\omega =(\omega^f,\omega^s,\omega^t)\in \Rm^{n-1}\times \Rm \times \Rm$, and define the set
$$
\Omega(K,s):=\{ \omega\in \Rm^{n+1}:\,
\|\omega^s\|> s, \,
|k^f\omega^f+k^t\omega^t|\geq 3sK \quad\forall (k^s,k^t)\in \Zm^2_K
\},
$$
where we have denoted by $\Zm^2_K$ the set of pairs $(k^f,k^t)$ of integers
such that $0<\max(k^f,k^t)\leq K$.
Note that
$$
\mD (K,s)=\{ p\in \Rm^n : (\partial_p H_0(p),1)\in \Omega(K,s)
\}.
$$

\begin{cor}\label{norm}
There exists a constant $c_n>0$, which depends only on $n$, such that the
following holds.
Given :
\begin{itemize}
 \item A $C^4$ Hamiltonian $H_0(p)$,
\item A $C^r$ Hamiltonian
$H_1(\theta,p,t)$ with $\|H_1\|_{C^r}=1$,
\item Parameters $r\geq n+5$, $\delta \in ]0,1[$, $\epsilon\in ]0,1[$,
$\beta>0$, $K>0$,
\end{itemize}
satisfying
\begin{itemize}
 \item $K \geq c_n\delta^{\frac{-1}{r-n-4}}$,
\item $\beta \geq c_n(1+\|H_0\|_{C^4})\delta^{-1/2},$
\item $\beta\epsilon^{1/4}\leq \|H_0\|_{C^4}$,
\end{itemize}
there exists a $C^2$ symplectic  diffeomorphism $\tilde \Phi$ of $\Tm^{n+1}\times \Rm^{n+1}$
such that,
in the new coordinates, the Hamiltonian $H_{\epsilon}=H_0+\epsilon H_1$ takes the form
$$
N_{\epsilon}=H_0+\epsilon Z+\epsilon R,
$$
with
\begin{itemize}
\item $\|R\|_{C^2}\le \delta$ on $\Tm^n\times \mD(K,\beta\epsilon^{1/4})\times \Tm$,
\item  $\|\tilde \Phi-id\|_{C^0}\le \delta\sqrt{\epsilon}$ and 
$\|\tilde \Phi-id\|_{C^2}\le \delta.$
\end{itemize}
The symplectic diffeomorphism $\tilde \Phi$ is of the form
$$
\tilde \Phi(\theta,p,t,e)=(\Phi(\theta,p,t),e+f(\theta,p,t))
$$
where $\Phi$ is a diffeomorphism of $\Tm^n\times \Rm^n \times \Tm$
fixing the last variable $t$.
The maps $\tilde \Phi$ and $\Phi$ are smooth if $H_0$ and $H_1$ are.
\end{cor}

\begin{proof}
We apply Theorem \ref{autonomous} with
$\tilde H_{\epsilon}$, $m=n+1$ and $\tilde \delta=\delta/2$.
We get a diffeomorphism $\tilde \Phi$ of $\Tm^{n+1}\times \Rm^{n+1}$
as time-one flow of the Hamiltonian $G$.
By inspection in the proof of Theorem \ref{autonomous}, we observe that
$G$ does not depend on $e$, which implies that $\tilde \Phi$ has the desired form.
We have
$$\tilde H_{\epsilon}\circ \tilde \Phi=\tilde H_0(J)+\epsilon \tilde R_1+\epsilon \tilde R_2
$$
where $\|\tilde R_2\|_{C^2}\leq \delta/2$ and
$$
\tilde R_1( \theta,p,t)=
\sum_{[k]\leq K}
\rho\left(\frac{k^f\cdot\partial_{p^f}H_0+k^s\partial_{p^s}H_0+k^t}{\beta\epsilon^{1/4}[k]}\right)
g_k(p)
e^{2i\pi k\cdot(\theta,t)}.
$$
Let us  compute this sum under the assumption that $p\in \mD(K,\beta\epsilon^{1/4})$
(or equivalently, that $(\partial_pH_0,1)\in \Omega(K,\beta\epsilon^{1/4})$).
We have
$$
\left|
\frac{k^f\cdot \partial_{p^f}H_0}{\beta\epsilon^{1/4}[k]}
\right|\leq 1
$$
hence
$$
\rho\left(\frac{k^f\cdot\partial_{p^f}H_0+k^s\partial_{p^s}H_0+k^t}{\beta\epsilon^{1/4}[k]}\right)
=1
$$
for  $k$ such that $k^s=0=k^t$.
For the other terms, we have, by definition of $\Omega(K,s)$,
$$
\left|
\frac{k^s\partial_{p^s}H_0+k^t}{\beta\epsilon^{1/4}[k]}
\right|\geq
\left|
\frac{k^s\partial_{p^s}H_0+k^t}{\beta\epsilon^{1/4}K}
\right|\geq 3,
$$
hence
$$
\left|
\frac{k^f\cdot\partial_{p^f}H_0+k^s\partial_{p^s}H_0+k^t}{\beta\epsilon^{1/4}[k]}
\right|\geq 2
$$
and  these terms vanish in the expansion of $\tilde R_1$.
We conclude that
$$
\tilde R_1(\theta,p,t)=\sum _{k^f\in \Zm^{n-1},[k^f]\leq K} g_{(k_f,0,0)}(p)
e^{2i\pi k^f\cdot \theta^f}
$$
hence
$\tilde R_1=Z-\Pi_K^+(Z)$, with the notation of Lemma \ref{fourier-est}.
Finally
$
\tilde H_{\epsilon}\circ \tilde \Phi=
\tilde H_0+\epsilon Z+\epsilon R_2
$
with $R_2=\tilde R_2-\Pi^+_KZ$.
From Lemma \ref{fourier-est}, we see that
$$
\|\Pi^+_KZ\|_{C^2}\leq c_nK^{m+3-r}\|Z\|_{C^r}\leq c_nK^{m+3-r}\|H_1\|_{C^r}
\leq c_nK^{m+3-r}\leq \delta/2.
$$
On the other hand, $\|\tilde R_2\|_{C^2}\leq \delta/2$, hence
$\|R_2\|_{C^2}\leq \delta$.
\end{proof}

\subsection{Smooth approximation}
We finally remove the restriction on $r$ and obtain a smooth change of coordinates.
If $r<n+5$, we use Lemma~\ref{appoximation} below to approximate $H_1$ by an
analytic  function $H_1^*:=S_{\tau}H_1$ (with a parameter $\tau$ that will be specified later).

\begin{lem}\label{appoximation} \cite{SZ}
  Let $f:\R^n\to \R$ be a $C^{r}$ function, with $r\geq 4$.
Then for each $\tau>0$ there exists an analytic function  $S_\tau f$  such that
$$ \|S_\tau f -f\|_{C^3}\le  c(n,r) \|f\|_{C^3}\tau^{r-3},$$
$$ \|S_\tau f\|_{C^{s}}\le 
c(n,r) \|f\|_{C^s} \tau^{-(s-r)},$$
for  each $s>r$, where $c(n,r)$ is a constant which depends only on $n$ and $r$.
\end{lem}

In order to obtain a smooth change of variables, it is also convenient to approximate
$H_0(p)$
in $C^4(B)$ by a smooth $H_0^*(p)$ (using a standard mollification).
We then apply Corollary~\ref{norm} to the Hamiltonian
$$
H^*_{\epsilon}:=H^*_0+\epsilon H^*_1=H^*_0+\epsilon_2 H_2
$$
with $H_2=H_1^*/\|H_1^*\|_{C^{r_2}}$, with $\epsilon_2=\epsilon\|H_1^*\|_{C^{r_2}}$, and
with some  parameters $r_2\geq r$ and $\delta_2\leq \delta$
to be specified later.
We find a smooth change of coordinates $\tilde \Phi$ such that
$$
\tilde H^*_{\epsilon}\circ \tilde  \Phi=\tilde H^*_0+\epsilon_2 Z_2+\epsilon_2 R_2
=\tilde H^*_0+\epsilon Z^*+\epsilon\|H_1^*\|_{C^{r_2}} R_2
$$
and $\|R_2\|_{C^2}\leq \delta_2$, where
$Z_2(\theta^s,p)=\int H_2d\theta^f dt$ and $Z^*(\theta^s,p)=\int H_1^*d\theta^f dt$.
As usual, we have denoted by $\tilde H^*_{\epsilon}$ and $\tilde H^*_0$ the automomized Hamiltonians
$\tilde H^*_{\epsilon}=H^*_{\epsilon}+e$ and $\tilde H^*_0=H^*_0+e$.
With the same map $\tilde \Phi$, we obtain
$$
\tilde H_{\epsilon}\circ \tilde \Phi=\tilde H_0+\epsilon Z + \epsilon R
$$
with
$$
R=\|H_1^*\|_{C^{r_2}}R_2+(Z-Z^*)+(H_1^*-H_1)\circ  \Phi
+\big((\tilde H^*_0-\tilde H_0)+(\tilde H_0-\tilde H^*_0)\circ \tilde \Phi \big)/\epsilon
$$
In the expression above, the map $\Phi$ is the trace on the $(\theta,p,t)$ variables
of the map $\tilde \Phi$.
Choosing $\tau=\delta_2^{1/(r_2-3)}$, and assuming that $\|H^*_0-H_0\|_{C^4}\leq \epsilon \delta/
c(n,4)$ we get
\begin{itemize}
 \item[-]$\|H_1^*-H_1\|_{C^3}\leq c(n,r_2)\delta_2^{\frac{r-3}{r_2-3}}$
\item[-] $\|H_1^*\|_{C^{r_2}}\leq c(n,r_2)\delta_2^{-\frac{r_2-r}{r_2-3}}$
\item[-] $\|Z^*-Z\|_{C^2}\leq \|H_1^*-H_1\|_{C^2}
\leq c(n,r_2)\delta_2^{\frac{r-3}{r_2-3}}$
\item[-] $\| \tilde \Phi - id \|_{C^2}\leq \delta_2\leq \delta \leq 1,$
\item[-]
$
\|(H_1^*-H_1)\circ  \Phi\|_{C^2}\le 
c(n,r_2) \|H_1^*-H_1\|_{C^2 }
(\|\Phi\|_{C^2}+\| \Phi\|_{C^2}^2) \le 
c(n,5) \|H_1^*-H_1\|_{C^2 }.
$
\item[-]$\|(\tilde H_0-\tilde H^*_0)\circ \tilde \Phi\|_{C^2}\leq \delta/
c(n,r_2)$.
\end{itemize}
and finally
$$
\|R\|_{C^2}\leq c(n,r_2)\delta_2^{\frac{r-3}{r_2-3}}+\delta/
c(n,r_2).
$$
We now set
$$\delta_2=\delta^{\frac{r_2-3}{r-3}}/c(n,r_2)\leq \delta/2
$$
and get $\|R\|_{C^2}\leq \delta$.
To apply Corollary~\ref{norm} as we just did, we need the following conditions
to hold on the parameters:
\begin{itemize}
\item[-] $K\geq  c(n,r_2)\delta^{\frac{r_2-3}{(r-3)(r_2-n-4)}}$, which implies
$K\geq c_n\delta_2^{\frac{-1}{r-n-4}}$,
\item[-] $\beta\geq
 c(n,r_2)(2+\|H_0\|_{C^4})\delta^{-\frac{r_2-3}{2(r-3)}}$
which implies $\beta\geq c_n (1+\|H^*_0\|_{C^4})\delta_2^{-1/2}$,
\item[-] $\beta \epsilon^{1/4}\leq(1+ \|H_0\|_{C^4})\delta ^{\frac{r_2-r}{4(r-3)}}$
which implies $\beta \epsilon_2^{1/4}\leq \|H^*_0\|_{C^4}$.
\end{itemize}
We apply the above discussion with $r_2=2n+5$ and get
Theorem \ref{normal-form}.
Note the estimate
$$
\|id-\tilde \Phi\|_{C^0} \leq \delta_2\sqrt{\epsilon_2}\leq
\delta_2^{1-\frac{r_2-r}{2(r_2-3)}}\sqrt{\epsilon}\leq \sqrt{\epsilon}.
$$
\qed

\section{Normally hyperbolic cylinders}\label{sec:NHIC}
In this section, we study the $C^2$  Hamiltonian
$$
N_\epsilon(\theta,p,t)=H_0(p) + \epsilon Z(\theta^s,p) +
\epsilon R(\theta,p,t).
$$
In the above notations we denote by
$p^s_*(p^f)\in  \Rm^{n-1}$ the solution of the equation
$\partial_{p^s}H_0(p^s_*(p^f),p^f)=0$. We recall also the notation
$p_*(p^f):=(p^s_*(p^f),p^f)$
from the introduction. 
We  assume that $\|Z\|_{C^3}\leq 1$,
and that 
$D^{-1}I\leq \partial^2_{pp}H_0\leq D\,I$ for some $D\geq 1$.
To simplify notations, we will be using the $O(\cdot)$ notation, 
where $f=O(g)$ means $|f|\le Cg$ for a constant $C$ independent of
$\epsilon$, $\lambda$, $\delta$, $r$, $a^-$, $a^+$.
 We will not be keeping track of the parameter $D$, 
which is considered fixed throughout the paper.

Given  parameters
$$
\lambda \in ]0,1], \quad a^-<a^+,
$$
we assume that for each $p^f\in [a^-,a^+]$ there exists  
a local maximum $\theta^s_*(p^f)$ of the map
$\theta^s \lmto Z(\theta^s,p_*(p^f))$, and that $\theta^s_*$
is a $C^2$ function of $p^f$. We assume in addition that
\begin{equation}\label{lambdadef}
 -I \le \partial^2_{\theta^s \theta^s}
Z(\theta^s_*(p^f), p_*(p^f))\le -\lambda I
\end{equation}
for each $p^f\in [a^-,a^+]$,
where as before $I$ is the identity matrix. We shall at some 
occasions lift the map $\theta^s_*$ to a $C^2$ map
taking values in $\Rm^{n-1}$ without changing its name.

%


\begin{thm}\label{nhic-mult}
The following conclusion holds if 
  $b\in ]0,1[$ is a sufficiently small constant
  (how small does not depend on the parameters $\epsilon, \lambda, \delta, a^-, a^+$):
If the parameters $\lambda \in ]0,1]$, $a^-<a^+$, $\epsilon$, $\delta$ satisfy 
$$
0<\epsilon<b\lambda^{9/2}\quad ,
\quad 0\leq \delta< b\lambda^{5/2},
$$
if  $\|R\|_{C^2}\leq \delta$, on the open set
\begin{equation}\label{dom}
\big\{(\theta,p,t):\quad  p^f\in ]a^-,a^+[, \quad
\|p^s-p^s_*(p^f)\|< \epsilon^{1/2}
\big\},
\end{equation}
and if (\ref{lambdadef}) holds for each $p^f\in  [a^-,a^+]$,  then 
there exists a $C^2$ map
 $$
(\Theta^s, P^s)(\theta^f, p^f, t):\T\times
[a^- +\sqrt{\delta\epsilon}, a^+ - \sqrt{\delta\epsilon}]\times \Tm \lto \Tm^{n-1}\times \Rm^{n-1}
$$ 
such that the cylinder
$$\mC=\{ (\theta^s, p^s)=(\Theta^s,P^s)(\theta^f, p^f, t);
\quad p^f\in[a^- +\sqrt{\delta\epsilon}, a^+ - \sqrt{\delta\epsilon}], \quad  
(\theta^f,t)\in \T\times \T\}
$$
is weakly invariant with respect to  $N_\epsilon$ in the sense 
that the Hamiltonian vector field is tangent to $\mC$.
The cylinder $\mC$ is contained in  the set
\begin{align*}
V:=\big\{&(\theta,p,t); p^f\in
 [a^- +\sqrt{\delta\epsilon}, a^+ - \sqrt{\delta\epsilon}];\\
&\|(\theta^s-\theta^s_*(p^f)\|\leq 
b^{1/5}\lambda^{3/2},
\quad
\|p^s-p^s_*(p^f)\|\le 
b^{1/5} \lambda ^{3/2}\epsilon^{1/2}
\big\},
\end{align*}
and it contains all the full orbits  of $N_{\epsilon}$ contained in $V$.
We have the estimates
$$
\|\Theta^s(\theta^f,p^f,t)-\theta^s_*(p^f)\|\leq
O\big(\lambda^{-1}\delta+\lambda^{-3/4}\sqrt{\epsilon}\big),
$$
$$
\|P^s(\theta^f,p^f,t)-p^s_*(p^f)\|\leq
\sqrt{\epsilon}\,O\big(\lambda^{-3/4}\delta+\lambda^{-1/2}\sqrt{\epsilon}\big),
$$
\begin{align*}
 \left\|\frac{\partial\Theta^s}{\partial p^f}\right\|=
O\left(\frac{\lambda^{-2}\sqrt{\epsilon}+\lambda^{-5/4}
\sqrt \delta}{\sqrt{\epsilon}}
\right)
\quad &,\quad
\quad \left\|\frac{\partial \Theta^s}{\partial(\theta^f, t)}\right\|=
O\left(\lambda^{-2}\sqrt{\epsilon}+\lambda^{-5/4}\sqrt\delta
\right),\\
 \left\|\frac{\partial P^s}{\partial p^f}\right\|=
O\left(1
\right)
\quad &,\quad
\quad \left\|\frac{\partial P^s}{\partial(\theta^f, t)}\right\|=
O\left(\sqrt \epsilon
\right).
 \end{align*}
\end{thm}

Notice that the domain $V$ is contained in the domain (\ref{dom}) 
where the assumption on $R$ is made.

\noindent
\textit{Proof of Theorem~\ref{intro-nhic-mult}. }
We derive Theorem~\ref{intro-nhic-mult} from  Theorem~\ref{nhic-mult} 
as follows.
We assume 
that  Hypothesis (\ref{HZl}) holds on 
$$
\Gamma_1:=\{(p_*(p^f)), p^f\in[a_-,a_+]\}.
$$
Then the inequality 
$$
 -I \le \partial^2_{\theta^s \theta^s}
Z(\theta^s_*(p^f), p_*(p^f))\le -2\lambda I
$$
holds for $p^f\in [a_-,a_+]$. Since $\|Z\|_{C^3}\leq 1$, the inequality 
$$
 -I \le \partial^2_{\theta^s \theta^s}
Z(\theta^s, p)\le -\lambda I
$$
holds for each $(\theta^s,p)$ in the $\lambda$-neighborhood of $(\theta^s_*(a_-), p_*(a_-))$.
The inequality 
$$
Z(\theta^s,p_*(a_-))\leq Z(\theta^s_*(a_-),p_*(a_-))-\lambda d^2 (\theta^s, \theta^s_*(a_-))
$$
implies that the function $Z(.,p_*(p^f))$ has a  global maximum $\theta^s_*(p^f)$, which is contained  in the ball
$B(\theta^s_*(a_-), \lambda )$, provided $|p^f-a_-|\leq b\lambda^3$ and $b$ is small enough.
By a similar reasoning at $a_+$, we extend the map $p^f\lmto \theta^s_*(p^f)$ to the interval 
$[a_--b\lambda^3,a_++b\lambda ^3]$ in such a way that, for each $p^f$ in this interval, the point $\theta^s_*(p^f)$
is a local (and even global) maximum of the function $Z(.,p_*(p^f))$ which satisfies the inequalities
$$
 -I \le \partial^2_{\theta^s \theta^s}
Z(\theta^s_*(p^f), p_*(p^f))\le -\lambda I.
$$
Taking a small $b>0$, 
we set $\kappa =b^{1/5}\lambda ^{3/2}$ and $\delta=b^3 \lambda ^{9}$.
Assuming as in the statement of Theorem~\ref{intro-nhic-mult}
that the estimate $\|R\|_{C^2}<\delta$ holds on $\Tm^n\times U_{\epsilon^{1/3}}\times \Tm$, hence on 
$$
\big\{(\theta,p,t):\quad  p^f\in ]a_- - \epsilon^{1/3}/2,a_+ +\epsilon^{1/3}/2[, \quad
\|p^s-p^s_*(p^f)\|< \epsilon^{1/3}/2
\big\}.
$$
and that  $\epsilon\in ]0,\delta[$, we
 apply Theorem \ref{nhic-mult} on the interval
$$
[a^-,a^+]:=
[a_- -\epsilon^{1/3}/2,a_+ +\epsilon^{1/3}/2]\subset [a_--b\lambda^3,a_++b\lambda ^3].
$$
If $b$ (hence $\kappa$) is  small enough, then we have the inclusion
$$
[a^-+\sqrt {\epsilon\delta}, a^+-\sqrt {\epsilon\delta}]\supset [a_- -\kappa \epsilon^{1/3},a_+ +\kappa \epsilon^{1/3}].
$$
\qed

The proof of Theorem \ref{nhic-mult} occupies the rest of the section.

The Hamiltonian flow admits the following equation of motion :
\begin{equation}\label{eq:perturbed}
\begin{cases}
  \dot{\theta}^s = \partial_{p^s}H_0 + \epsilon\partial_{p^s}Z + 
\epsilon\partial_{p^s} R \\
  \dot{p}^s = -\epsilon \partial_{\theta^s}Z - \epsilon\partial_{\theta^s} R \\
  \dot{\theta}^f = \partial_{p^f}H_0 + \epsilon\partial_{p^f}Z +  
\epsilon\partial_{p^f} R \\
  \dot{p}^f = -\epsilon \partial_{\theta^f} R \\
  \dot{t}=1
\end{cases}.
\end{equation}
The Hamiltonian structure of the flow 
is not used in the following proof.

It is convenient in the sequel to lift the angular variables
to real variables and to consider the above system as defined on
$\Rm^{n-1}\times \Rm^{n-1}\times \Rm\times \Rm\times \Rm.
$
We will see this system as a perturbation of the model system
\begin{equation}\label{eq:model}
 \dot{\theta}^s = \partial_{p^s}H_0\quad,\quad
 \dot{p}^s = -\epsilon \partial_{\theta^s}Z \quad, \quad
\dot{\theta}^f = \partial_{p^f}H_0\quad,\quad
\dot{p}^f = 0\quad, \quad
\dot{t}=1.
\end{equation}
The graph
of the map
$$
(\theta^f,p^f,t)\lmto (\theta^s_*(p^f),p^s_*(p^f))
$$
on $\Rm\times ]a^-,a^+[\times \Rm$ is obviously invariant
for the model flow.
For each fixed $p^f$, the point $(\theta^s_*(p^f),p^s_*(p^f))$
is a hyperbolic fixed point of the partial system
$$ \dot{\theta}^s = \partial_{p^s}H_0(p^s,p^f)\quad,\quad
 \dot{p}^s = -\epsilon \partial_{\theta^s}Z(\theta^s,p^s,p^f)
$$
where $p^f$ is seen as a parameter.
This hyperbolicity is the key property we will
use, through the theory of normally hyperbolic invariant manifolds.
It is  not obvious to  apply this theory here  because the model
system itself depends on $\epsilon$, and because we have to deal
with the problem of
non-invariant boundaries.
We will however manage to apply the quantitative version
exposed  in Appendix \ref{sec:abstract-nhic}.

We perform some  changes of coordinates in order
to put the system in the
framework of Appendix
\ref{sec:abstract-nhic}. These coordinates appear naturally from
the study of the model system as follows.
We set
$$
B(p^f):=\partial^2_{p^sp^s}H_0(p_*(p^f))\quad,\quad
A(p^f):=- \partial^2_{\theta^s\theta^s}Z(\theta^s_*(p^f),p_*(p^f)).
$$
If we fix the variable $p^f$ and consider
the model system in $(\theta^s,p^s)$, we observed that this system
has a  hyperbolic fixed point at $(\theta^s_*(p^f),p^s_*(p^f))$.
The linearized system at this point is
$$
\dot \theta^s=B(p^f)\,p^s
\quad,\quad
\dot p^s=\epsilon A(p^f)\,\theta^s.
$$
To put this system under a simpler form, it is useful to consider the matrix
$$
T(p^f):=\big(B^{1/2}(p^f)(B^{1/2}(p^f)A(p^f)B^{1/2}(p^f))^{-1/2}
B^{1/2}(p^f)
\big)^{1/2}
$$
which is symmetric, positive definite, and satisfies
$T^2(p^f)A(p^f)T^2(p^f)=B(p^f)$, as can be checked by a direct computation.
We finally introduce the symmetric positive definite matrix
$$
\Lambda(p^f):= T(p^f)A(p^f)T(p^f)=T^{-1}(p^f)B(p^f)T^{-1}(p^f).
$$
In  the new variables
$$
\xi=T^{-1}(p^f) \theta^s+\epsilon^{-1/2}T(p^f)p^s
\quad,\quad
\eta=T^{-1}(p^f) \theta^s-\epsilon^{-1/2}T(p^f)p^s,
$$
the linearized system is reduced to the following block-diagonal form:
$$
\dot \xi=\epsilon^{1/2}\Lambda(p^f)\xi
\quad,\quad
\dot \eta=-\epsilon^{1/2}\Lambda(p^f)\eta,
$$
see
\cite{Be3} for more details.
This leads us to introduce the following set of new coordinates for 
the full system:

$$
x=T^{-1}(p^f) (\theta^s-\theta^s_*(p^f))+
\epsilon^{-1/2}T(p^f)(p^s-p^s_*(p^f))
$$
$$
y=T^{-1}(p^f) (\theta^s-\theta^s_*(p^f))-
\epsilon^{-1/2}T(p^f)(p^s-p^s_*(p^f)),
$$
$$
I=\epsilon^{-1/2}p^f\quad,\quad \Theta=\gamma \theta^f,
$$
where $\gamma$ is a parameter which will be taken  later equal to 
$\delta^{1/2}$. Note  that
$$
\theta^s=\theta^s_*(\epsilon^{1/2}I)+\frac{1}{2} T(\epsilon^{1/2}I)(x+y),\quad
p^s=p^s_*(\epsilon^{1/2}I)+\frac{\epsilon^{1/2}}{2} T^{-1}(\epsilon^{1/2}I)(x-y).
$$
\begin{lem}\label{lambda}
We have $\Lambda (p^f)\geq \sqrt{\lambda/D}\ I$
for each $p^f\in [a^-,a^+]$.
\end{lem}
\proof
The matrix $\Lambda$ is symmetric, hence it satisfies
$\Lambda\geq \lambda_*I$, where $\lambda_*>0$ is its smallest eigenvalue.
The real number $\lambda_*$ is then an eigenvalue of the matrix
$
\begin{bmatrix}\Lambda& 0\\0&-\Lambda\end{bmatrix}
$
which is similar  to
$
\begin{bmatrix}0&B\\A&0\end{bmatrix}.
$
Since both $A$ and $B$ are square matrices of equal size,
we conclude that $\lambda_*^{-2}$ is an eigenvalue of $A^{-1}B^{-1}$.
Since $\|A^{-1}\|\leq \lambda^{-1}$ and $\|B^{-1}\|\leq D$, we
have $\lambda_*^{-2}\leq \|A^{-1}B^{-1}\|\leq D\/\lambda^{-1}$.
We conclude that $\lambda_*\geq \sqrt{\lambda/D}$.
\qed

The links between the various parameters
$\epsilon$, $\delta$, $\gamma$,  $\lambda$, $\rho$ which appear in
the computations below will be specified later.
We will however assume from the beginning that
$$
\delta\leq \rho\leq  \lambda<1
\quad , \quad
 \sqrt{\epsilon}\leq \rho<1
\quad,\quad
0< \gamma\leq \lambda<1.
$$
Let us first collect some estimates that will be
useful to see that the  system (\ref{eq:perturbed})
is indeed a perturbation of the model system.

\begin{lem}\label{all-estimates}
We have the estimates
\begin{align*}
&\|T\|=O(\lambda^{-1/4}),\  \|T^{-1}\|=O(1),\
\|\partial_{p^f} T\|\leq O(\lambda ^{-5/4}),
\|\partial_{p^f} T^{-1}\|\leq O(\lambda ^{-3/4}),\\ 
&\|\partial_{p^f}\theta^s_*\|\leq O( \lambda^{-1}),\ 
\|p^s_*\|_{C^2}=O(1),
\|\theta^s-\theta^s_*\|\leq O(\lambda^{-1/4}\rho), \
\|p^s-p^s_*\|\leq O(\epsilon^{1/2}\rho),
\end{align*}
where $\rho=\max(\|x\|,\|y\|)$.
\end{lem}

\proof
We recall that
$T=\big(B^{1/2}(B^{1/2}AB^{1/2})^{-1/2}B^{1/2}\big)^{1/2}$ and
$T^{-1}=\big(B^{-1/2}(B^{1/2}AB^{1/2})^{1/2}B^{-1/2}\big)^{1/2}$.
Since $D^{-1}I\leq B\leq D\,I$ and $\lambda I\leq A\leq I$,
we obtain that $\|T\|\leq O(\lambda^{-1/4})$ and that 
$\|T^{-1}\|\leq O(1)$.
To estimate the derivative of $T$, we consider the map
$F:M\lmto M^{1/2}$ defined on positive symmetric matrices.
It is known
that
$$
dF_M\cdot N=\int_0^{\infty}e^{-tM^{1/2}}Ne^{-tM^{1/2}}dt.
$$

To verify this one can diagonalize $M$, perform integration,
and match terms in $(M^{1/2}+\epsilon dF_M\cdot N)
(M^{1/2}+\epsilon dF_M\cdot N)=M+\epsilon N + O(\epsilon^2)$.
This implies that
$$\|dF_M\|\leq \|M^{1/2}\|^{-1}/2\leq \|M^{-1/2}\|/2.
$$
As a consequence, if $M(p_f)$ is a positive symetric matrix depending on $p_f$,
we have
$$
\|\partial_{p^f} M\|\leq \|M^{-1/2}\|\|\partial_{p^f} M\|/2.
$$
We apply this bound several times
to estimate $\partial_{p^f}T$ and $\partial_{p^f}T^{-1}$.
In our situation, we have $\partial_{p^f} A=O(1)$, $\partial_{p^f} B =O(1)$.
Using $M= A$ and $B$, we get $\partial_{p^f} (A^{1/2})=O(\lambda^{-1/2})$
and $\partial_{p^f} (B^{1/2})=O(1)$ resp.
Using $M= B^{1/2}A B^{1/2} $
we get
$\partial_{p^f}[(B^{1/2}AB^{1/2})^{1/2}]=O(\lambda^{-1/2})$,
and then
\begin{align*}
\|\partial_{p^f}[T^{-1}]\|
&\leq \big \|\big(B^{-1/2}(B^{1/2}AB^{1/2})^{1/2}B^{-1/2}\big)^{-1/2}\big \|
\big\|\partial_{p^f} [B^{-1/2}(B^{1/2}AB^{1/2})^{1/2}B^{-1/2}]\big\|\\
&=O(\lambda ^{-1/4}) O(\lambda ^{-1/2})=O(\lambda ^{-3/4}).
\end{align*}
Recalling that 
$$
\|\partial_{p^f} (M^{-1})\|\leq \|M^{-1}\|^2\|\partial_{p^f}M\|,
$$
we obtain (with $M=T^{-1}$)
$$
\|\partial_{p^f} T\|\leq \|T\|^2\|\partial_{p^f}[T^{-1}]\|\leq\partial_{p^f}( M^{1/2})=O(\lambda^{-5/4}).
$$
The other estimates are straightforward.
\qed

\begin{cor}\label{est-V}
Let $\tilde V$ be the image in the $(x,y, I, \Theta,t )$ coordinates  
of the domain called $V$ in the statement.
We have 
\begin{align*}
\tilde V \subset \{x:\|x\|\leq b^{1/6}\lambda^{5/4}\}\times
\{y:\|y\|\leq b^{1/6}\lambda^{5/4}\}\times \Rm\times
\left[\frac{a^-}{\sqrt{\epsilon}}+\sqrt \dt ,\frac{a^+}{\sqrt{\epsilon}}-\sqrt \dt \right]
\times \Rm,\\
\tilde V \supset \{x:\|x\|\leq 2b^{1/4}\lambda^{7/4}\}\times
\{y:\|y\|\leq 2b^{1/4}\lambda^{7/4}\}\times \Rm\times
\left[\frac{a^-}{\sqrt{\epsilon}}+\sqrt \dt ,\frac{a^+}{\sqrt{\epsilon}}-\sqrt \dt \right]
\times \Rm
\end{align*}
provided $b$ is small enough.
\end{cor}

From now   on, we work  on the region
$$
p^f\in [a^-,a^+], \quad \|x\|\leq \rho, \quad \|y\|\leq \rho.
$$
In view of Lemma \ref{all-estimates}, this region is contained in the
(image in the new coordinates of the)  
domain where the inequality $\|R\|_{C^2}\leq \delta$ was assumed.

\begin{lem}\label{uniform-estimates}
The equations of motion in the new coordinates
take the form
\begin{align*}
\dot x &=-\sqrt{\epsilon}\Lambda(\sqrt{\epsilon}I)x
+\epsilon^{1/2}O(\lambda^{-1/4}\delta+\lambda^{-3/4}\rho^2)
+O(\epsilon)
\\
\dot y  &=\sqrt{\epsilon}\Lambda(\sqrt{\epsilon}I)y
+\epsilon^{1/2}O(\lambda^{-1/4}\delta+\lambda^{-3/4}\rho^2)
+O(\epsilon)
\\
\dot I &= O(\sqrt{\epsilon}\delta),
\end{align*}
where $\rho=\max(\|x\|,\|y\|)$ is assumed to satisfy $\rho\leq \lambda $.
The expression for $\dot \Theta$ is not useful here.
\end{lem}
\proof
The last part of the statement is obvious.
We prove the part concerning $\dot x$, the calculations
for $\dot y$ are exactly the same.
In the original coordinates
the vector field (\ref{eq:perturbed}) can be written
$$
\dot \theta^s=B(p^f)(p^s-p^s_*(p^f))+O(\|p^s-p^s_*(p^f)\|^2)+
O(\epsilon),
$$
$$
\dot p^s=\epsilon A(p^f)(\theta^s-\theta^s_*(p^f))+
O(\epsilon\|\theta^s-\theta^s_*(p^f)\|^2)+O(\epsilon \delta).
$$
As a consequence, we have
\begin{align*}
\dot x&=
T^{-1}B(p^s-p^s_*)+\epsilon^{1/2}TA(\theta^s-\theta^s_*)\\
&+T^{-1} \cdot O(\|p^s-p^s_*\|^2+\epsilon)
+\epsilon^{1/2}T \cdot O(\|\theta^s-\theta^s_*\|^2+\delta)\\
&+(\partial_{p^f} T^{-1}) \,\dot p^f(\theta^s-\theta^s_*)
+\epsilon^{-1/2}(\partial_{p^f} T)\, \dot p^f (p^s-p^s_*)\\
&-T^{-1}(\partial_{p^f} \theta^s_* )\,\dot p^f
-\epsilon^{-1/2}T(\partial_{p^f}p^s_*)\,\dot p^f.
\end{align*}
We use the estimates of Lemma \ref{all-estimates}
to simplify (recall also that $\dot p^f=O(\epsilon \delta)$):
\begin{align*}
\dot x&=
T^{-1}B(p^s-p^s_*)+\epsilon^{1/2}TA(\theta^s-\theta^s_*)\\
&+O(\epsilon\rho^2 +\epsilon) +
O(\epsilon^{1/2} \lambda^{-3/4}\rho^2+\epsilon^{1/2} \lambda^{-1/4}\delta)\\
&+O(\lambda^{-1}\epsilon \delta \rho)+O(\lambda^{-5/4}\epsilon \delta \rho)
+O(\lambda^{-1}\epsilon \delta+\lambda^{-1/4}\epsilon^{1/2}\delta).
\end{align*}
\qed

\begin{lem}\label{linearized-estimates}
In the new coordinate system $(x,y,\Theta,I,t)$, the linearized
system is given  by the matrix
\begin{align*}
L=\begin{bmatrix}
\sqrt{\epsilon}\Lambda &
0&
0 & 0  &0\\
0&-\sqrt{\epsilon}\Lambda &
0 &
  0&0\\
0 &0 & 0 &
 0  & 0\\
0 & 0 & 0 & 0 & 0\\
0 & 0 & 0 & 0 & 0
\end{bmatrix}+O(\sqrt{\epsilon}\delta \lambda^{-1/4}\gamma^{-1}+\sqrt{\epsilon}\lambda^{-3/4}\rho+
\epsilon \lambda^{-5/4}+\sqrt{\epsilon}\gamma),
\end{align*}
where $\rho=\max (\|x\|,\|y\|)$.

\end{lem}

\proof
Most of the estimates below are based on Lemma \ref{all-estimates}.
In the original coordinates, the matrix of the linearized system
is:
$$
\tilde L=\begin{bmatrix}
  O(\epsilon) & \partial_{p^sp^s}^2H_0+O(\epsilon) & 0 & \partial_{p^fp^s}^2H_0   +O(\epsilon)& 0 \\
-\epsilon \partial^2_{\theta^s\theta^s}Z &O(\epsilon) & 0 & O(\epsilon)&0\\
O(\epsilon) &O(1) & 0 & O(1) & 0\\
0 & 0 & 0 & 0 & 0\\
0 & 0 & 0 & 0 & 0
\end{bmatrix}+O(\delta\epsilon),
$$
%
In our notations we have
$$
\tilde L=\begin{bmatrix}
  O(\epsilon) & B+O(\epsilon+\sqrt \epsilon \rho) & 0 & \partial_{p^fp^s}^2H_0      +O(\epsilon)& 0 \\
\epsilon A+O(\epsilon  \lb^{-1/4} \rho) &O(\epsilon) & 0 & O(\epsilon)&0\\
O(\epsilon) &O(1) & 0 & O(1) & 0\\
0 & 0 & 0 & 0 & 0\\
0 & 0 & 0 & 0 & 0
\end{bmatrix}+O(\delta\epsilon),
$$
In the new coordinates, the matrix is the product
$$
L=\left[\frac{\partial (x,y,\Theta, I, t)}{\partial (\theta^s,p^s,\theta^f,p^f,t)}
\right]\cdot
\tilde L\cdot
\left[\frac{\partial (\theta^s,p^s,\theta^f,p^f,t)}{\partial (x,y,\Theta, I, t)}\right].
$$
We have
$$
\left[\frac{\partial (\theta^s,p^s,\theta^f,p^f,t)}
{\partial (x,y,\Theta, I, t)}\right]=
\begin{bmatrix}
T/2 &T/2   & 0 &O(\sqrt{\epsilon}\lambda^{-1})& 0 \\
\sqrt{\epsilon}T^{-1}/2 &-\sqrt{\epsilon}T^{-1}/2 & 0 & \sqrt{\epsilon}\partial_{p^f}p^s_*+O(\epsilon\lambda^{-3/4}\rho)&0\\
0 &0  & \gamma^{-1} & 0 & 0\\
0 & 0 & 0 & \sqrt{\epsilon} & 0\\
0 & 0 & 0 & 0 & 1
\end{bmatrix}
$$

hence
\begin{align*}
&\tilde L\left[\frac{\partial (\theta^s,p^s,\theta^f,p^f,t)}
{\partial (x,y,\Theta, I, t)}\right]=O(\gamma^{-1}\delta\epsilon)+\\
&\begin{bmatrix}
\sqrt{\epsilon}BT^{-1}/2+O(\epsilon\lambda^{-1/4}) & -\sqrt{\epsilon}BT^{-1}/2+O(\epsilon\lambda^{-1/4})& 0&
O(\epsilon\lambda^{-3/4}\rho+\epsilon^{3/2}\lambda^{-1})  &0\\
\epsilon AT/2+O(\epsilon\lambda^{-1/2}\rho) &\epsilon AT/2+O(\epsilon\lambda^{-1/2}\rho) & 0 &
 \epsilon^{3/2}O( \lambda^{-5/4}\rho+\lambda^{-1})&0\\
O(\sqrt{\epsilon}) &O(\sqrt{\epsilon})  & 0 &
 O(\sqrt{\epsilon})  & 0\\
0 & 0 & 0 & 0 & 0\\
0 & 0 & 0 & 0 & 0
\end{bmatrix}.
\end{align*}
%
%
%
%
This expression is the result of a tedious, but straightforward, computation.
Let us just detail the computation of  the coefficient
on the first line, fourth row, which contains
an important cancellation:
\begin{align*}
&\sqrt{\epsilon}\partial_{p^sp^s}^2H_0\partial_{p^f}p^s_*+
\sqrt{\epsilon}\partial_{p^fp^s}^2H_0
+O(\epsilon\lambda^{-3/4}\rho+\epsilon^{3/2}\lambda^{-1})\\
=&
\sqrt{\epsilon}\partial_{p^f}\big(\partial_{p^s}H_0(p_*(p^f)\big)
+O(\epsilon\lambda^{-3/4}\rho+\epsilon^{3/2}\lambda^{-1})
=O(\epsilon\lambda^{-3/4}\rho+\epsilon^{3/2}\lambda^{-1}).
\end{align*}
We now write
$$
\left[\frac{\partial (x,y,\Theta, I, t)}{\partial (\theta^s,p^s,\theta^f,p^f,t)}
\right]
=
\begin{bmatrix}
T^{-1} & \epsilon^{-1/2}T   & 0 &O(\epsilon^{-1/2}\lambda^{-1/4})& 0 \\
T^{-1} & -\epsilon^{-1/2}T  &0& O(\epsilon^{-1/2}\lambda^{-1/4})&0\\
0 &0  & \gamma & 0 & 0\\
0 & 0 & 0 & \epsilon^{-1/2}& 0\\
0 & 0 & 0 & 0 & 1
\end{bmatrix},
$$
and compute  that
\begin{align*}
L=&\begin{bmatrix}
\sqrt{\epsilon}\Lambda+O(\sqrt{\epsilon}\lambda^{-3/4}\rho) &
O(\sqrt{\epsilon}\lambda^{-3/4}\rho)&
0 &  O(\epsilon \lambda^{-5/4})  &0\\
O(\sqrt{\epsilon}\lambda^{-3/4}\rho)&-\sqrt{\epsilon}\Lambda+ O(\sqrt{\epsilon}\lambda^{-3/4}\rho) &
0 & O( \epsilon\lambda^{-5/4})
  &0\\
 O(\sqrt{\epsilon}\gamma) &O(\sqrt{\epsilon}\gamma)  & 0 &
  O(\sqrt{\epsilon}\gamma)  & 0\\
0 & 0 & 0 & 0 & 0\\
0 & 0 & 0 & 0 & 0
\end{bmatrix}\\
+&O(\sqrt{\epsilon}\delta \lambda^{-1/4}\gamma^{-1}).
\end{align*}
\qed

In order to prove the existence of a normally hyperbolic
invariant strip (for the lifted system), we apply  Proposition \ref{realNHI}
to the system in  coordinates $(x,y,\Theta,I,t)$.
More precisely, with the notations of appendix \ref{sec:abstract-nhic},
we set:
$u=x, s=y, c_1=(\Theta,t), c_2=I$, and consider the domain
$$
 \Omega=\Rm^2\times \Omega^{c_2}=\Rm^2\times 
\left[\frac{a^-}{\sqrt{\epsilon}}+\sqrt \dt ,\frac{a^+}{\sqrt{\epsilon}}-\sqrt \dt \right].
$$
We fix 
\begin{equation}
  \label{eq:alpha}
  \gamma=\sqrt{\delta},\quad
\alpha =\sqrt{\epsilon\lambda/4D},\quad 
\sigma=\sqrt{\delta},
\end{equation}
observe that $\sqrt{\epsilon}\Lambda\geq 2\alpha I$,
by Lemma \ref{lambda}.
 We assume, as in the statement of the Theorem, 
that $0<\epsilon<b\lambda^{9/2}$ and that
$0\leq \delta< b \lambda ^{5/2}$.
We  apply Proposition \ref{realNHI}
with $B^u=\{u: \|u\|\leq \rho\}$ and $B^s=\{s: \|s\|\leq \rho\}$
under the constraint
\begin{equation}\label{rho}
b^{-1/4}(\lambda^{-3/4}\delta+\lambda^{-1/2}\sqrt{\epsilon})\leq \rho\leq 
b^{1/6}\lambda^{5/4},
\end{equation}
provided $b\in]0,1[$ is small enough.
Observe that, if $b$ is small enough, the inequalities
$$
b^{-1/4}(\lambda^{-3/4}\delta+\lambda^{-1/2}\sqrt{\epsilon})
\leq 2b^{1/4}\lambda^{7/4}\leq b^{1/6}\lambda^{5/4}
$$
holds
under our assumptions on the parameters, hence values of $\rho$ satisfying (\ref{rho}) do exist.
It is easy to check under our assumptions
on the parameters that such values of $\rho$ exist.
Let us check the isolating block condition under the condition (\ref{rho}).
By Lemma \ref{uniform-estimates}, we have
$$
\dot x\cdot x\geq 2\alpha\|x\|^2-
\|x\|\ O( \epsilon^{1/2} \lambda^{-1/4}\delta+\epsilon^{1/2} \lambda^{-3/4}\rho^2
+\epsilon)
$$
if $x\in B^u,y\in B^s$. If in addition $\|x\|=\rho$, then
$$
\lambda^{-3/4}\delta\leq b^{1/4}\|x\|\quad,\quad
\lambda^{-5/4}\rho^2\leq b^{1/6}\|x\|\quad,\quad
\sqrt{\epsilon/\lambda}\leq b^{1/4}\|x\|,
 $$
hence
$$
\dot x\cdot x\geq 2\alpha\|x\|^2-
\|x\|^2b^{1/6}O(\sqrt{\epsilon\lambda})\geq \alpha \|x\|^2
$$
provided  $b$ is small enough. Similarly,
$\dot y \cdot y\leq -\alpha \|y\|^2$
on $B^u\times \partial B^s$ provided $b$ is small enough.
Concerning the linearized system, we have
\begin{align*}
L_{uu}&=\sqrt{\epsilon}\Lambda+
O(\sqrt{\epsilon}\delta \lambda^{-1/4}\gamma^{-1}+\sqrt{\epsilon}\lambda^{-3/4}\rho+
\epsilon \lambda^{-5/4}+\sqrt{\epsilon}\gamma)\\
&=\sqrt{\epsilon}\Lambda+ O(b^{1/6}\sqrt{\epsilon\lambda})\geq \alpha I,\\
L_{ss}&=-\sqrt{\epsilon}\Lambda+ O(b^{1/6}\sqrt{\epsilon\lambda})\leq -\alpha I
\end{align*}
on $B^u\times B^s\times \Omega_r$.
These  inequalities holds when $b$ is small enough
because $\sqrt{\epsilon} \Lambda\geq 2\alpha I$ and
 $\sqrt{\epsilon\lambda}\leq O(\alpha)$.
Finally, still with the notations of Proposition
\ref{realNHI},  we  can take
\begin{equation}
  \label{eq:m}
  \begin{aligned}
  m&=O(\sqrt{\epsilon}\delta \lambda^{-1/4}\gamma^{-1}+\sqrt{\epsilon}\lambda^{-3/4}\rho+
\epsilon \lambda^{-5/4}+\sqrt{\epsilon}\gamma
+\sqrt \epsilon \delta/\sigma
)\\
&=\sqrt{\epsilon\lambda}\,O(\sqrt\delta \lambda^{-3/4}+\rho \lambda^{-5/4}+
\sqrt{\epsilon}\lambda^{-7/4})
=\sqrt{\epsilon\lambda}\,O(b^{1/6}).
  \end{aligned}
\end{equation}
If $b$ is small enough, we have $16m<\alpha$ hence
$$
K\leq 2m/\alpha<1/8,
$$
 and Proposition \ref{realNHI} can be applied.
The invariant strip obtained from the proof of Proposition  \ref{realNHI}
does not depend on the choice of $\rho$, as long as (\ref{rho}) holds.
It contains all the full orbits contained in
$$
\{x:\|x\|\leq b^{1/6}\lambda^{5/4}\}\times
\{y:\|y\|\leq b^{1/6}\lambda^{5/4}\}\times \Rm\times
\left[\frac{a^-}{\sqrt{\epsilon}}+\sqrt \dt ,\frac{a^+}{\sqrt{\epsilon}}-\sqrt \dt \right]
\times \Rm\supset \tilde V,
$$
where $\tilde V$ is the image in the new coordinates of the domain
$V$  defined 
 in the statement of Theorem \ref{nhic-mult}
 and where the last inclusion holds provided $b$ is small enough,
as follows from Corollary \ref{est-V}.
 So our invariant strip contains all the full orbits contained in 
$\tilde V$.
 On the other hand, we can take $\rho=2b^{1/4}\lambda^{7/4}$, and since 
 $$
 \{x:\|x\|\leq 2b^{1/4}\lambda^{7/4}\}\times
\{y:\|y\|\leq 2b^{1/4}\lambda^{7/4}\}\times \Rm\times
\left[\frac{a^-}{\sqrt{\epsilon}}+\sqrt \dt ,\frac{a^+}{\sqrt{\epsilon}}-\sqrt \dt \right]
\times \Rm\subset \tilde V
 $$
(still for $b$ small enough, by Corollary \ref{est-V}), our invariant strip is contained in $\tilde V$.

  The possibility of taking
$\rho=b_1^{-1/4}(\lambda^{-3/4}\delta+\lambda^{-1/2}\sqrt{\epsilon})$
now implies that the cylinder is actually contained in the domain
where
$$
\|x\|, \|y\|\leq b_1^{-1/4}(\lambda^{-3/4}\delta+\lambda^{-1/2}\sqrt{\epsilon}).
$$
Moreover, with this choice of $\rho$
and using that  $K=O( m/\sqrt{\epsilon\lambda})$, we can obtain
an improved estimate of the Lipschitz constant $K$ (notation from the appendix):
\begin{align*}
K  &=O\big(\sqrt\delta \lambda^{-3/4}+\rho \lambda^{-5/4}+
\sqrt{\epsilon}\lambda^{-7/4}\big)\\
&  =O\big(
\sqrt\delta \lambda^{-3/4}+
b_1^{-1/4}\delta \lambda^{-2}+
b_1^{-1/4}\sqrt{\epsilon}\lambda^{-7/4}+
\sqrt{\epsilon}\lambda^{-7/4}
\big)\\
&  =O\big(
\sqrt\delta \lambda^{-3/4}+
\sqrt\delta \lambda^{-1}+
b_1^{-1/4}\sqrt{\epsilon}\lambda^{-7/4}
\big)\\
&  =O\big(
\sqrt\delta \lambda^{-1}+
b_1^{-1/4}\sqrt{\epsilon}\lambda^{-7/4}
\big).
\end{align*}
%
%
Observe finally that, since the system is $1/\gamma$-periodic in
$\Theta$ and $1$-periodic in $t$, so is the invariant strip 
given by Proposition \ref{realNHI}.
We have obtained the existence of a $C^1$ map

$$
w^c=(w^c_u,w^c_s):(\Theta,I,t)\in\Rm\times
\left[\frac{a^-}{\sqrt{\epsilon}}+\sqrt \dt ,\frac{a^+}{\sqrt{\epsilon}}-\sqrt \dt \right]
\times \Rm
\lto \Rm^{n-1}\times \Rm^{n-1}
$$
which is $2K$-Lipschitz, $1/\gamma$-periodic in $\Theta$ and
$1$-periodic in $t$, and the graph of which is tangent to the vector field.
Our last task is to return to the original coordinates
by setting
\begin{equation}
  \label{eq:graph-explicit}
  \begin{aligned}
  \Theta^s(\theta^f,p^f,t)&=\theta^s_*(p^f)+\frac{1}{2}T(p^f)
\cdot(w^c_u+w^c_s)(\gamma \theta^f,\epsilon^{-1/2}p^f,t)\\
P^s(\theta^f,p^f,t)&=p^s_*(p^f)+\frac{\sqrt{\epsilon}}{2}T^{-1}(p^f)\cdot
(w^c_u-w^c_s)(\gamma \theta^f,\epsilon^{-1/2}p^f,t).
  \end{aligned}
\end{equation}
All the estimates stated in Theorem \ref{nhic-mult} follow directly
from these expressions, and from the fact that $\|dw^c\|\leq 2K$.
This concludes the proof of Theorem~\ref{nhic-mult}. \qed

\section{Localization and Mather's projected graph theorem}\label{sec:localization}

We  study the system in  normal form $N_\epsilon =H_0 +\epsilon Z + \epsilon R$
of Theorem \ref{mainnormal} from the point of view of Mather theory
at a fixed cohomology   $c\in \Rm^n$ such  that $\partial_{p^s}H_0(c)=0$
(or in other words such that $c\in \Gamma$).
We assume that $\|Z\|_{C^2}\leq 1$, and that 
$\|R\|_{C^2}\leq \delta$ on $\{\|p-c\|<\epsilon^{1/3}\}$.
We continue to assume (\ref{Dd}), and, for simplicity, we  assume that 
$D$ is large enough and $\epsilon$ small enough for the following inequality to also hold:
$$
(1/D)I\leq \partial^2_p N_\epsilon 
\leq D I.
$$
Most of our statement depend on the shape of the function $Z_c:\theta^s\lmto Z(\theta,c)$.
We will most of the time assume that 
 (\ref{HZl}) holds at $c$ : 
 There exists $\theta^s_*$ such that
 $Z(\theta^s,c)\leq Z(\theta^s_*,c)-\lambda d^2(\theta^s,\theta^s_*).$
 We will rewrite this inequality as
 $$\hat Z_c(\theta^s) \leq -\lambda d^2(\theta^s,\theta^s_*)$$
 with the  notation $\hat Z_c=Z_c-\max Z_c$.
Later in section \ref{sec-double}, we also consider the double peak case, which is not necessary for the 
proof of Theorem~\ref{mainnormal}, but is very natural.
Our first statement  localizes the Ma\~né set.

\begin{thm}\label{thm-loc}
In the single peak case (when (\ref{HZl}) holds at $c$),
if $\delta>0$ is small enough with respect to $n,D, \lambda$
and $\epsilon$ is small enough with respect to $n,D, \lambda, \delta$,
then the Ma\~né set at cohomology $c$ of the Hamiltonian $N_{\epsilon}$ satisfies
$$
 s\tilde \mN(c)\subset  B(\theta^s_*, \delta^{1/5}) \times \Tm \times B(c, \sqrt{\epsilon}\delta^{1/16})\times \Tm
 \subset 
 \Tm^{n-1}\times \Tm \times \Rm^n \times \Tm.$$
\end{thm}

This statement is proved  in Section \ref{sec-loc}.
Our second statement is a quantitative version of the celebrated Mather Lipschitz graph Theorem, it does
not rely on any particular assumption on $Z$, besides $\|Z\|_{C^2}\leq 1$:

\begin{thm}\label{thm-lip}
For   each Weak KAM solution $u$ of $N_{\epsilon}$
at cohomology $c$,
the set $\tilde \mI(u,c )\subset \Tm^n \times \Rm^n$ is contained in a $9\sqrt{D\epsilon}$-Lipshitz graph
above $\Tm^n $.
\end{thm}

This theorem is proved in Section \ref{sec-lip}. 
We will always assume in this section that $\delta$ is sufficiently small with respect to $n, H_0$ and $\lambda$, 
and that $\epsilon$ is sufficiently small with respect to $n, H_0 , \lambda$ and $\delta$.

\subsection{Some inequalities}

We will denote by $N$ the Hamiltonian $N_{\epsilon}$ and by $L$ the associated Lagrangian function,
which is defined by
$$
L(\theta, v,t)= \max_{p\in \Rm^n} \big( p\cdot v-N(\theta,p,t)\big).
$$
The function $L$ is then $C^2$, and the maps 
$$
(\theta, p,t)\lmto \partial_p N(\theta,p,t), \quad (\theta, v,t)\lmto \partial_v L(\theta,v,t), 
$$
are diffeomorphisms of $\Tm^n \times \Rm^n\times \Tm$, which are inverse of each other.
The maximum in the definition of $L$ is reached at $p=\partial_v L(\theta,v,t)$.
Since $I/D\leq \partial_{pp} N\leq DI$, we have 
$$
I/D\leq \partial_{vv} L\leq DI.
$$
We will also denote by $L_0(v)$ the Lagrangian associated to $H_0$, or more explicitly
$L_0(v):= \sup_p (p\cdot v - H_0(p))$. It satisfies 
$$
I/D\leq \partial_{vv} L_0\leq DI.
$$

\begin{lem}\label{domain}
For each $\rho\in[4D\epsilon, \epsilon^{1/4}]$,
the image of the open set $\Tm^n \times B(c,\rho)\times \Tm$ under the difféomorphism $\partial_p N$ 
contains the set 
$$
\Tm^n \times B(\partial_p H_0 (c),\rho/2D-2\epsilon)\times \Tm.
$$
In particular, if $\epsilon$ is small enough, the image of $\Tm^n \times B(c,\epsilon^{1/4})\times \Tm$
contains $\Tm^n \times B(c,\epsilon^{1/4}/4D)\times \Tm$.
\end{lem}

\proof
In view of the estimate $\partial_{p}^2 H\geq I/D$, 
each of the applications $p\lmto \partial_p N(\theta, p,t)$ sends the ball $B(c,r)$
to a set which contains the ball $B(\partial_pN(\theta,c,t), r/2D$).
Since $|\partial_pN(\theta,c,t)-\partial_pH_0(c)|\leq \epsilon +\epsilon \delta\leq 2\epsilon$,
we conclude that the image contains 
$B(\partial_p H_0 (c),\rho/2D-2\epsilon)$.
\qed

\begin{lem}\label{lem-L}
The estimates
$$
\|\partial_{\theta v}L\|_{C^0}\leq 2 D\epsilon, \quad
\|\partial_{\theta \theta}L\|_{C^0}\leq 3\epsilon.
$$
hold on $\Tm^n \times B(c,\epsilon^{1/3}/4D)\times \Tm$.
\end{lem}

\proof
Note first that the estimates 
$$
\|\partial_{\theta p}H\|\leq 2\epsilon, \quad
\|\partial_{\theta \theta}H\|\leq 2\epsilon
$$
hold on the domain $\Tm^n \times B(c,\epsilon^{1/3})\times \Tm$, which contains 
the image of $\Tm^n \times B(c,\epsilon^{1/3}/4D)\times \Tm$ under $\partial_vL$.
Observing that 
$
\partial_{\theta}L=-\partial_{\theta}N(\theta, \partial_vL_{\epsilon}(\theta, v)),
$
which implies
$$
\partial_{v\theta}L(\theta,v,t)=-
\partial_{p\theta}
N_{\epsilon}\big(\theta,\partial_vL(\theta,v,t),t\big) \partial_{vv}L(\theta,v,t)
$$
we deduce that
$\|\partial_{\theta v}L\|\leq 2D\epsilon
$
on $\Tm^n \times B(c,\epsilon^{1/3}/4D)\times \Tm$.
The equality
$$
\partial_{\theta \theta}L(\theta,v,t)=
-\partial_{\theta\theta}N(\theta,\partial_vL(\theta,v,t),t)
-\partial_{p\theta}
N\big(\theta,\partial_vL(\theta,v,t),t\big) \partial_{\theta v}L(\theta,v,t),
$$
implies that 
$\|\partial_{\theta \theta}L\|\leq 2\epsilon +(2\epsilon)(2D\epsilon)
$
on $\Tm^n \times B(c,\epsilon^{1/3}/4D)\times \Tm$.
\qed

\begin{lem}
We have the estimate
$$
|L(\theta,v,t)-
(L_0(v)-\epsilon Z(\theta^s, c))|\leq 2\epsilon \delta
$$
if $|v-\partial_pH_0(c)|< \epsilon ^{1/3}/4D$.
\end{lem}
\proof
On the domain $\{|p-c|<\epsilon^{1/3}\}$, we have 
$$
|N(\theta,p,t)-
(H_0(p)+\epsilon Z(\theta^s,c)|\leq \epsilon ^{5/4}+\epsilon \delta \leq 2\epsilon \delta.
$$
If $|v-\partial_pH_0(c)|< \epsilon ^{1/3}/4D$, then by Lemma \ref{domain},
$$
L(\theta,v,p)=\sup_{|p-c|<\epsilon^{1/3}}
[p\cdot v-N(\theta,p,t)]
$$
and, by Lemma \ref{domain} applied with $R\equiv 0$ and $Z(\theta^s,p)\equiv Z(\theta^s,c)$
$$
L_0(v)-\epsilon Z(\theta^s,c)=
\sup_{p}
[p\cdot v-H_0(p)-\epsilon Z(\theta^s,c)]
=
\sup_{|p-c|<\epsilon^{1/3}}
[p\cdot v-H_0(p)-\epsilon Z(\theta^s,c)].
$$
\qed

Let us now estimate the value
$\alpha(c)$
of the Mather function of $N$.
We use the notation $Z_c(\theta^s):= Z(\theta^s,c)$.

\begin{lem}
The value $\alpha(c)$ of the Mather function of $N$ satisfies 
$$
|\alpha(c)-(H_0(c)+\epsilon \max Z_c)|\leq 2\epsilon \delta.
$$
\end{lem}

The reason behind this inequality is that the value $\alpha(c)$ of the Hamiltonian $H_0+\epsilon Z_c$
is $H_0(c)+\epsilon \max Z_c$.

\proof
On one hand, we have
$$\alpha(c)\leq \max_{(t,\theta)} N_\epsilon (t,\theta,c)\leq H_0(c)
+\epsilon \max Z_c
+\epsilon \max_{(t,\theta)\in \Tm^{n+1}}R(\theta,c,t)
\leq H_0(c)
+\epsilon \max Z_c  +\epsilon \delta.
$$
For the other inequality, we use that $\partial_{p^s}H_0=0$.
We   consider the Haar measure $\mu$  of the torus
$\Tm \times \{\theta^s_*(c)\}\times \{\partial H_0(c)\}\times \Tm$,
where $\theta^s_*(c)$ is any point maximizing $Z_c$.
This measure is not necessarily invariant under the Lagrangian flow of $L$,
but it is invariant under the Lagrangian flow of $L_0-Z_c$ (because $\partial_{p^s}H_0=0$)
hence  it is closed,
which means that $\int \partial_t f+\partial_{\theta}f\cdot v \,d\mu(\theta,v,t)=0$
for each smooth function $f(t,\theta)$. See \cite{Ba, FS} 
(both inspired from \cite{Mn}) for the notion of closed measures.
Each closed measure $\mu$ has a rotation vector
 $\rho(\mu):=\int v\,d\mu(\theta,v,t)\in \Rm^n$, and its action is not less than
$c\cdot \rho(\mu)-\alpha(c)$. Here, $\rho(\mu)=\partial_pH_0(c)$ hence
\begin{align*}\alpha(c)\geq c\cdot \partial_pH_0(c) -\int L d\mu
&=c\cdot \partial_pH_0(c)-L_0(\omega)+\epsilon Z_c(\theta^s_*(c)) -2\epsilon \delta \\
&= H_0(c)+\epsilon \max Z_c -2\epsilon \delta.
\end{align*}
\qed

\begin{lem}\label{lagrangiencorrige}
If $\epsilon$ is small enough (with respect to $D$ and $\delta$), we have the estimates
\begin{align}
L(\theta,v,t)-c\cdot v +\alpha(c)
&\geq \|v-\partial H_0(c)\|^2/(4D)-\epsilon \hat Z_c(\theta^s)-4\epsilon \delta\\
L(\theta,v,t)-c\cdot v +\alpha(c)
&\leq
D\|v-\partial H_0(c)\|^2 -\epsilon \hat Z_c(\theta^s)+4\epsilon \delta
 \end{align}
for each $(\theta, v,t)\in \Tm^n\times \Rm^n \times \Rm$,
where
$\hat Z_c(\theta^s):=Z(\theta^s,c)-\max_{\theta^s}Z(\theta^s,c)$.
\end{lem}

\proof
It is a direct computation :
\begin{align*}
L(\theta,v,t)
& \geq 
c\cdot v-N(\theta,c,t)+\|v-\partial_pN(\theta,c,t)\|^2/2D\\
& \geq c\cdot v-H_0(c)-\epsilon Z_c(\theta^s)-\epsilon\delta  +\big(\|v-\partial_pH_0(c))\|-2\epsilon\big)^2/2D\\
&\geq c\cdot v-\alpha(c)+\epsilon (\max Z_c-Z_c(\theta^s))-3\epsilon \delta
+\|v-\partial_pH_0(c))\|^2/4D-16\epsilon ^2,
\end{align*}
\begin{align*}
L(\theta,v,t)
& \leq 
c\cdot v-N(\theta,c,t)+D\|v-\partial_pN(\theta,c,t)\|^2/2\\
& \leq c\cdot v-H_0(c)-\epsilon Z_c(\theta^s)+\epsilon\delta  +D\big(\|v-\partial_pH_0(c))\|+2\epsilon\big)^2/2\\
&\leq c\cdot v-\alpha(c)+\epsilon (\max Z_c-Z_c(\theta^s))+3\epsilon \delta
+D\|v-\partial_pH_0(c))\|^2+8D\epsilon ^2.
\end{align*}
\qed

It is useful to consider suspended weak KAM solutions.
Recall that we defined Weak KAM solutions associated to a Lagrangian $L$ at cohomolgy $c$
as functions $u$ on $\Tm^n$ such that, for each $t\in \Nm$, 
$$
u(\theta)=\inf_{\gamma} 
\left(u(\gamma(0))+
\int_{0}^t L(\gamma(s),\dot \gamma(s), s)-c\cdot \dot \gamma(s)+\alpha(c) ds
\right),
$$
where the infimum is taken on the set of $C^1$ curves $\gamma:\Rm\lto \Tm^n$ such that $\gamma(t)=\theta$.
We can similarly define suspended weak KAM solutions as functions $u:\Tm^n\times \Tm\lto \Rm$
such that 
$$u(\theta,T \text{ mod } 1)=\inf_{\gamma} \left(u(\gamma(S),S\text{ mod } 1)+
\int_S^T L(\gamma(t),\dot \gamma(t),t) +
c\cdot \dot \gamma(t)\ dt\right),
$$
for each real times $S\leq T$ , where the infimum
is taken on the space of $C^1$ curves $\gamma:[S,T]\lto \Tm^n$
such that $\gamma(T)=\theta$.
There is a bijection between suspended weak KAM solution
$u(\theta,t)$ and genuine weak KAM solutions:
Each suspended weak KAM solution $u(\theta,t)$ restricts to a genuine weak
KAM solution $u(\theta)=u(\theta,0)$, and each genuine weak KAM
solution $u(\theta)$ is the restriction of a unique suspended
weak KAM solution $u(\theta,t)$ which can be defined by
$$
u(\theta, t\text{ mod } 1)=\inf_{\gamma} \left( u(\gamma(0)+
\int_0^t L(\gamma(s),\dot \gamma(s),s) +
c\cdot \dot \gamma(s) +\alpha(c) \ ds\right),
$$
for each $t>0$, 
where the infimum is taken on $C^1$ curves $\gamma:\Rm\lto \Tm^n$ such that $\gamma(t)=\theta$.
We shall use the same notation for a weak KAM solution $u$ and the associated suspended weak KAM solution.
Curves  $\gamma$ calibrated by the weak KAM solutions $u(\theta)$ are also calibrated by the corresponding
suspended weak KAM solution  in the sense that 
$$
u(\gamma(t_2),t_2\text{ mod } 1)-u(\gamma(t_1),t_1\text{ mod } 1)
=\int_{t_1}^{t_2}L(\gamma(s),\dot \gamma(s),s) +
c\cdot \dot \gamma(s) +\alpha(c) \ ds
$$
for each  time interval  $[t_1,t_2]$. 
Let us now estimate the oscillation $\text{osc } u:= \max u- \min u$
of suspended weak KAM solutions. We consider 
a convex subset $\Omega \in \Tm^{n-1}$, meaning that it is the projection of a convex subset $\tilde \Omega$
of $\Rm^{n-1}$, of diameter less than $2\sqrt{n}$.

\begin{lem}\label{oscillation}
Let $u(\theta,t)$ be a suspended weak KAM solution of $N$ at cohomology $c$.\\
Given   two points
$(\theta_1,t_1), (\theta_2,t_2)\in \T\times \Omega\times \T$,
we have
$$ u(\theta_2,t_2)-u(\theta_1,t_1)\le 10\sqrt{nD\epsilon(m+4\delta)},$$
where $m:=-\inf_{\Omega}\hat Z_c$.
We can take in particular $\Omega =\Tm^{n-1}$, then $m\leq 1$ and we conclude that 
$\text{\emph{osc} } u\leq 10 \sqrt{2nD\epsilon}$.
\end{lem}

\begin{proof}
We have $0\geq \hat Z_c \geq -m$ on $\Omega$.
We take two points
$(\theta_i, t_i)$, $i=1$ or $2$ in the domain $\T\times \Omega \times \T$,
and consider the curve
$$
\theta(t)=\theta_1+(t- \tilde t_1)
\frac{\tilde  \theta_2- \tilde\theta_1+ [(T+\tilde t_2-\tilde t_1)\partial H_0(c)]}
{T+\tilde t_2-\tilde t_1}
$$
where $T\in \Nm$ is a parameter to be fixed later,
where $\tilde t_i\in [0,1[$ and $\tilde \theta _i\in [0,1[\times \tilde \Omega$
are representatives of the angular variables $t_i, \theta_i$,
and where $[\omega]\in \Zm^{n}$ is the component-wise integral part of $\omega$.
Note that $\theta(\tilde t_1)=\theta_1$ and $\theta(\tilde t_2+T)=\theta_2$,
hence
\begin{align*}
u(\theta_2,t_2)-u(\theta_1,t_1)
&\leq \int_{\tilde t_1}^{\tilde t_2+T}
L(\theta(t),\dot \theta (t),t)-c\cdot \dot \theta(t)+\alpha(c)\ dt\\
&\leq \int_{\tilde t_1}^{\tilde t_2+T}
D\|\dot \theta -\partial H_0(c)\|^2-\epsilon \hat Z_c(\theta^s(t))+4\epsilon \delta \ dt\\
&\leq \int_{\tilde t_1}^{\tilde t_2+T}
\frac{9Dn}{(T+\tilde t_2-\tilde t_1)^2}+\epsilon m+4\epsilon \delta \ dt\\
&\leq \frac{9Dn}{(T+\tilde t_2-\tilde t_1)}+(T+\tilde t_2-\tilde t_1)\epsilon (m+4\delta).
\end{align*}
This inequality holds for all $T\in \Nm$,
in particular, we can choose $T\in \Nm$
so that
$$
2\sqrt{\frac{nD}{\epsilon (m+4\delta)}}\leq T+\tilde t_2-\tilde t_1\leq 3\sqrt{\frac{nD}{\epsilon(m+4\delta)}}
$$
and obtain
$
u(\theta_2,t_2)-u(\theta_1,t_1)
\leq
10\sqrt{nD\epsilon(m+4\delta)}.
$
\end{proof}

\subsection{Localization of the invariant sets}\label{sec-loc}

We prove Theorem \ref{thm-loc}.
It is enough to prove that the inclusion 
$$
s\tilde \mI(u,c) \subset  B(\theta^s_*(c), \delta^{1/5}) \times \Tm \times B(c, \sqrt{\epsilon}\delta^{1/16})\times \Tm
$$
holds for each (suspended) weak KAM solution $u$. We fix such a solution 
$u(\theta,t)$ and prove the inclusion.
The following preliminary localization,
which does not use any assumption on the shape of $Z$, implies that  the set
$s\tilde \mI(u,c)$ is contained (when $\epsilon$ is small enough)
in the domain $\{\|p-c\|<\epsilon^{1/3}\}$
where the assumption $\|R\|_{C^2}\leq \delta$ is made.

\begin{lem}\label{init-vert}
Let 
$(\theta(t),p(t)):[t_1,t_2]\lto \T^n\times \R^n$ be an orbit calibrated by $u$.
If $t_2-t_1\geq \epsilon^{-1/2}$, then 
$$
\|p(t)-c\|\leq C \sqrt{\epsilon}
$$
for each $t\in[t_1,t_2]$, where $C$ is a constant which 
depends on $n$ and $D$.  In particular, 
$$
s\tilde \mI(u,c)\subset \Tm^n \times B(c,C \sqrt{\epsilon})\times \Tm\subset \Tm^n \times B(c,\epsilon^{1/3})\times \Tm.
$$
\end{lem}

\proof
We denote by $C_i$ various positive constants which depend 
on $n$ and $D$. Since $\hat Z_c \leq 0$, we have 
$L(\theta,v,t)\geq \|v-\partial_pH_0(c)\|^2/4D-4\epsilon \delta$.
As a consequence, $L(\theta,v,t)\geq 20 \epsilon \sqrt{nD}$ if 
$\|v-\partial_pH_0(c)\|\geq C_1\sqrt{\epsilon }$.
In view of Lemma \ref{domain}, we thus have 
$$
L(\theta(t), \dot \theta(t),t)\geq 20 \sqrt{nD}\epsilon 
$$ 
for each $t$ such that $\|p(t)-c\|\geq C_2\sqrt {\epsilon}$. 
Since $\theta$ is a calibrated curve, we have 
$$
\int_{t'_1}^{t'_2} L(\theta(t), \dot \theta(t), t) dt 
\leq \text{osc } u
$$
for each $[t'_1,t'_2]\subset [t_1,t_2]$. In particular, 
by Lemma \ref{oscillation} we have 
$20\epsilon  \sqrt{nD} (t_2-t_1) > \text{osc }u$. Therefore,
there exists a time $t_0\in [t_1,t_2]$ such that 
{$\|p(t_0)-c\|= C_2 \sqrt{\epsilon}$}. Let $t_3\in [t_1,t_2]$ be 
the time maximizing $\|p(t)-c\|$. We assume for definiteness 
that $t_3\geq t_0$, and that $\|p(t)-c\|\geq C_2 \sqrt{\epsilon}
$
for each $t\in [t_0,t_3]$ (otherwise we reduce the interval).
The equations of motion imply that $\|\dot p\|\leq 2\epsilon$ 
on $[t_0,t_3]$, hence 
$t_3\geq t_0+ (\|p(t_3)-c\|-C_2\sqrt{\epsilon})/2\epsilon$, and 
{using the above lower bound on $L(\theta(t),\dot \theta(t),t)$}
$$
20\sqrt{nD\epsilon}
\geq \text{osc } u \geq \int_{t_0}^{t_3} L(\theta(t), \dot \theta(t), t) dt
\geq 10\sqrt{nD}(\|p(t_3)-c\|-C_2\sqrt{\epsilon})
$$
which implies that $\|p(t_3)-c\|\leq (2+C_2)\sqrt{\epsilon}$.
\qed

We now assume that
$Z(\theta^s,c)\leq Z(\theta^s_*,c)-\lambda d^2(\theta^s,\theta^s_*),$
or {, equivalently,} that 
$\hat Z_c(\theta^s) \le -\lambda d^2(\theta^s,\theta^s_*),$ and prove 
the horizontal part of Theorem \ref{thm-loc}, or more precisely that 
\begin{equation}\label{eq-hor}
s\mI(u,c)\subset \Tm \times B(\theta^s_*(c), \delta^{1/5})\times \Tm.
\end{equation}

We  consider the domain $\Omega=B(\theta^s_*,4\sqrt{\delta/\lambda})$.
On this domain, we have $-8\delta/\lambda\leq \hat Z_c$, hence, by  
Lemma \ref{oscillation}, the oscillation of $u$ on 
$\Tm\times \Omega\times \Tm$ satisfies 
$$
\text{osc}_{\Tm\times \Omega\times \Tm}\ \  u\leq 40\sqrt{nD\epsilon \delta/\lambda}.
$$ 
For $\theta^s\not \in \Omega$, we have 
$$
L(\theta,v,t)
-c\cdot v-\alpha(c)
\geq
\|v-\partial_pH_0(c)\|^2/4D+\lambda \epsilon d^2(\theta^s, \theta^s_*)/2
\geq 
\|v^s\|^2/4D+\lambda \epsilon d^2(\theta^s, \theta^s_*)/2,
$$
by Lemma \ref{lagrangiencorrige}.
Let $\theta(t):\Rm \lto \Tm^n$ be a curve calibrated by $u$,
and let $[t_1,t_2]$ be an excursion of $\theta^s$  outside of $\Omega$, meaning that 
$d(\theta^s(t), \theta^s_*)>4\sqrt{\delta/\lambda}$ for each 
$t\in ]t_1,t_2[$, and that 
$d(\theta^s(t_1), \theta^s_*)=4\sqrt{\delta/\lambda}
=d(\theta^s(t_2), \theta^s_*)$. We have the inequalities
\begin{align*}
40\sqrt{nD\epsilon \delta/\lambda}
&\geq  
\int_{t_1}^{t_2}L(\theta(t),\dot \theta(t),t)-c\cdot \dot \theta(t)+\alpha(c)\ dt
\geq 
\int_{t_1}^{t_2}
\frac{\|\dot \theta^s(t)\|^2}{4D}+ \lambda \epsilon
 \frac{d^2(\theta^s(t), \theta^s_*(c))}{2}\ dt.
\end{align*}
If the curve $\theta^s(t)$ is not contained in $B(\theta^s_*,\delta^{1/5})$ 
on $[t_1,t_2]$, then there exists a time interval
$[t_3,t_4]\subset [t_1,t_2]$  such that 
$d(\theta(t), \theta^s_*)>\delta^{1/5}/2$ on $[t_3,t_4]$,
$d(\theta(t_3), \theta^s_*)=\delta^{1/5}/2=d(\theta(t_4), \theta^s_*)$,
and $\max_{t\in [t_3,t_4]}d(\theta(t), \theta^s_*)>\delta^{1/5}$.
We then have 
$\int _{t_3}^{t_4} \|\dot \theta ^s(t)\|dt\geq \delta^{1/5}$
hence
\begin{align*}
40\sqrt{nD\epsilon \delta/\lambda}
&
\geq 
\int_{t_1}^{t_2} \frac{\|\dot \theta^s(t)\|^2}{4D}+ 
 \lambda \epsilon \frac{d^2(\theta^s(t), \theta^s_*(c))}{2}\ dt 
 \\ \geq  \int_{t_3}^{t_4}
 \|\dot \theta^s(t)\|^2/(4D)+ \lambda \epsilon
 d^2(\theta^s(t), \theta^s_*(c))/2\ dt 
&\geq \frac{1}{4D(t_4-t_3)}\left(\int _{t_3}^{t_4} 
\|\dot \theta ^s(t)\|dt\right)^2 + \\
\lambda \epsilon (t_4-t_3)\delta^{2/5}/8
& \geq 
\frac{1}{4D(t_4-t_3)}\delta^{2/5}+
\lambda \epsilon \delta ^{2/5}(t_4-t_3)/8\\
&
\geq 
\frac{\sqrt{\lambda \epsilon }}{8\sqrt D}\delta^{2/5}
\end{align*}
which is a contradiction when $\delta$ is small enough with respect to $n, D$ and $\lambda$.
We have proved (\ref{eq-hor}).
\qed

We can now prove a better vertical localization of the set $s\tilde \mI(u,c)$ than was obtained in Lemma
\ref{init-vert}.
On the domain
$\Tm \times B(\theta^s_*, \delta^{1/5})\times \Tm$, we have 
$\hat Z_c \geq -\delta^{2/5}/2$.
We deduce from Lemma \ref{oscillation} that 
$$
10 \delta^{1/5}\sqrt{nD\epsilon}\geq u(\theta(t_2),t_2)-u(\theta(t_1),t_1)=
 \int_{t_1}^{t_2}L(\theta(t),\dot \theta(t),t)-c\cdot \dot \theta(t)+\alpha(c)\ dt
$$
for each curve $\theta:\Rm\lto \Tm^n$ calibrated by $u$ and each time interval $[t_1,t_2]$.
We can chose the time interval $[t_1,t_2]$ as a maximal excursion outside
of 
$\{\|p-c\|<\sqrt\epsilon \delta^{1/16}/2\}$.
On $[t_1,t_2]$, we have 
$\|\dot \theta-\partial_pH_0(c)\|\geq \sqrt{\epsilon}\delta^{1/16}/5D$
(by Lemma \ref{domain})
hence
$$
L(\theta,\dot \theta, t)-c\cdot \dot \theta+\alpha(c)\geq \epsilon \delta^{1/8}/100D^2-4\epsilon \delta
\geq \epsilon \delta^{1/8}/200D.
$$
We thus have 
$$(t_2-t_1)\epsilon \delta^{1/8}/200D\leq 10 \delta^{1/5}\sqrt{nD\epsilon}
$$
hence $2\epsilon (t_2-t_1)\leq \sqrt \eps \delta^{1/16}/2$ 
(if $\delta$ is small enough). Since $\|\dot p\|\leq 2\epsilon$ and 
$\|p(t_1)-c\|=\sqrt \eps \delta^{1/16}/2$, we conclude
that $\|p(t)-c\|\leq \sqrt \eps \delta^{1/16}$ on $[t_1,t_2]$.
This ends the proof of Theorem \ref{thm-loc}.
\qed

\subsection{The Lipschitz constant}\label{sec-lip}

We prove Theorem \ref{thm-lip}. We will work here with weak KAM solutions
rather than suspended weak KAM solutions.
We  recall the concept of semi-concave function on $\Tm^n$.
A function $u:\Tm^n\lto \Rm$ is called $K$-semi-concave if the function
$$
x\lmto
u(x)-K\|x\|^2/2
$$
is concave on $\Rm^n$, where $u$ is seen as a periodic function on $\Rm^n$.
It is equivalent to require that, for each $\theta\in \Tm^n$,
there exists a linear form $l$ on $\Rm^n$
such that the inequality
$$
u(\theta+y)\leq u(\theta) + l\cdot y +K\|y\|^2/2
$$
holds for each $y\in \Rm^n$. 
It is sufficient to check that, for each $\theta$, there exists $l$ such that  this inequality holds for $\|y\|\leq 1$.
We will   need the following regularity result of Fathi, see \cite{Fa}:

\begin{lem}\label{freg}
Let $u_1$ and $u_2$ be  $K$-semiconcave functions, and let $\mI\subset \Tm^n$
be the set of points where the sum $u_1+u_2$ is minimal.
Then the functions $u_1$ and $u_2$ are differentiable at each point of $\mI$,
and the differential $x\lmto du_1(x)$ is $6K$-Lipshitz on $\mI$.
\end{lem}

The Weak KAM solutions of cohomology $c$ are the functions $u:\Tm^n \lto \Rm$
such that 
$$
u  (\theta):=
\min_{\gamma} \left( u(\gamma(0))+\int_0^T L(\gamma(t),\dot \gamma(t),t) -
c\cdot \dot \gamma(t) +\alpha(c)\, dt\right),
$$
for each $T\in \Nm$, 
where the minimum is taken on the set of $C^1$ curves $\gamma:[0,T]\lto \Tm^n$
satisfying the final condition $\gamma(T)=\theta$.

\begin{prop}\label{scc}
For each $c\in \Rm^n$, each Weak KAM solution
$u$ at cohomology $c$  is $3\sqrt{D\epsilon}/2$-semi-concave.
\end{prop}

\proof
Given $T\in \Nm$ and $\theta\in \Tm^n$, there exists 
a curve $\Theta:[0,T]\lto \Tm^n$ 
such that $\Theta(T)=\theta$ and which is 
calibrated by $u$, which means that
$$u(\theta)=
u(\Theta(0))+\int_{0}^T L(t,\Theta(t),\dot \Theta(t) )-c\cdot \dot \Theta(t) +\alpha(c)dt.
$$
We assume that $T\geq \epsilon^{-1/2}$, which implies by Lemma \ref{init-vert} that 
$\|p(t)-c\|\leq C\sqrt{\epsilon}$, for a contant $C$ independant of $\epsilon$ and $\delta$.
We deduce that $\|\dot \Theta-\partial_pH_0(c)\|\leq C\sqrt{\epsilon}$ (with a higher constant $C$) for each $t\in [0,T]$.
We lift $\Theta$  (and the point $\theta=\Theta(T)$) to a curve in $\Rm^n$ without changing its
name, and consider, for each $x\in \Rm^n$,
the curve
$$\Theta_x(t):=\Theta(t)+tx/T,$$
so that $\Theta_x(T)=\theta+x$.
Each of the curves $\Theta_x, \|x\|\leq 1$, satisfy   $\|\dot \Theta_x -\partial_pH_0(c)\|\leq C\sqrt{\epsilon}\leq \epsilon^{1/3}$ 
(provided $\epsilon$ is small enough).
We have the inequality
$$
u(\theta+x)-u(\theta)\leq \int_0^T L(\Theta_x(t),\dot \Theta_x(t),t)-
L(\Theta(t), \dot \Theta(t),t) -c\cdot x/T \,dt.
$$
Use Lemma \ref{lem-L}, we get  
\begin{equation}\label{action-cp}
\begin{aligned}
L(\Theta_x(t),\dot \Theta_x(t),t)&\leq
L(\Theta(t), \dot \Theta(t),t)\\
&+
\partial_{\theta}L (\Theta(t), \dot \Theta(t),t)\cdot tx/T
+\partial_vL (\Theta(t), \dot \Theta(t),t)\cdot x/T\\
&+3\epsilon |tx/T|^2/2+2D \epsilon t|x/T|^2+D|x/T|^2/2.
\end{aligned}
\end{equation}
Using the Euler-Lagrange equation and integrating by parts, we conclude that
$$
u(\theta+x)-u(\theta)\leq \big(c+\partial_vL(T,\Theta(T), \dot \Theta(T))\big)\cdot x
+(\epsilon T/2+ D\epsilon +D/2T)|x|^2
$$
for each $T\in \Nm$, $T\geq \epsilon^{-1/2}$.
Taking $T\in [\sqrt{D/\epsilon}, \sqrt{2D/\epsilon}]$, we obtain
$$
u(\theta+x)-u(\theta)\leq (c+\partial_vL(T,\Theta(T), \dot \Theta(T))\cdot x
+3\sqrt{D\epsilon}|x|^2/2
$$
for each $x\in \Rm^n$, $\|x\|\leq 1$.
This ends the proof of the
 semi-concavity.
\qed

\textit{Proof of Theorem  \ref{thm-lip}. }
Let $u$ be a weak KAM solution, and let $\check u$ be the conjugated dual weak KAM
solution.
Then the set $\tilde \mI(u,c)$ can be characterized as follows:
Its projection $\mI(u,c)$ on $\Tm^n$ is the set where $u=\check u$,
and
$$\tilde \mI(u,c)=\{(x,c+du(x)), x\in \mI(u,c)\}.
$$
Since $-\check u$ is semi-concave, it is a consequence of Lemma \ref{freg}
that the differential $du(x)$ exists for $x\in \mI(u,c)$.
Moreover, we can prove exactly as in Proposition  \ref{scc} that
$-\breve u$ is $3{D\epsilon}/2$-semi-concave.
Lemma \ref{freg} then implies that the map $x\lmto du(x)$ is
$9\sqrt{D\epsilon}$-Lipschitz on $\mI(u,c)$.
\qed

\subsection{Double peak case}
\label{sec-double}
We now localize the Aubry and Ma\~né sets in the more general case where (\ref{HZl}) is replaced by:
$$ \hat Z_c(\theta^s) \le - \lambda
\big(\min\{d(\theta^s-\theta^s_1), d(\theta^s-\theta^s_2)\}\big)^2.
$$
It is natural to relax (\ref{HZl}) in this  way because, for a generic family of functions $\hat Z_c, c\in \Gamma$,
there exist values of $c$ for which $\hat Z_c$ has two degenerate maxima.
Note that Theorem \ref{thm-lip} is still valid in this case,  its proof does not use (\ref{HZl}).
On the other hand, Theorem \ref{thm-loc} is replaced by:

\begin{thm}\label{thm-double}
 If $\delta>0$ is small enough with respect to $n,D, \lambda$
and if $\epsilon$ is small enough with respect to $n,D, \lambda, \delta$,
then the Aubry set at cohomology $c$ of the Hamiltonian $N_{\epsilon}$ satisfies
$$
 s\tilde \mA(c)\subset \big( B(\theta^s_1, \delta^{1/5})\cup B(\theta^s_2, \delta^{1/5}) \big)
 \times \Tm \times B(c, \sqrt{\epsilon}\delta^{1/16})\times \Tm
 \subset 
 \Tm^{n-1}\times \Tm \times \Rm^n \times \Tm.
 $$
If, moreover, the projection $\theta^s(s\mA(c)) \subset \Tm^{n-1}$ 
is contained in one of the (disjoint) balls $B(\theta^s_i,\delta^{1/5})$, then 
the projection  $\theta^s(s\mN(c)) \subset \Tm^{n-1}$  of the Ma\~né set is contained in the same ball
$B(\theta^s_i,\delta^{1/5})$.
 \end{thm}

\proof
We assume that $\theta^s_1\neq \theta^s_2$, and that $\delta$ is small enough for the balls
$B(\theta^s_i, 2\delta^{1/5})$ to be disjoint.
We first show that 
$$
\theta^s(s\mA(c))\subset B(\theta^s_1,\delta^{1/5})\cup B(\theta^s_2,\delta^{1/5}).
$$
As in the single peak case, we set $r_1=4\sqrt{\delta/\lambda}$, 
and observe that
$$L(\theta,v,t)
-c\cdot v-\alpha(c)
\geq 
\|v^s\|^2/4D+\lambda \epsilon \big(\min\{d(\theta^s-\theta^s_1), d(\theta^s-\theta^s_2)\}\big)^2/2
$$
 for $\theta^s\notin B(\theta_1^s, r_1)\cup B(\theta_{2}^s, r_1)$.
The $\theta^s$ component of
each orbit of the Aubry set spends a finite amount of time outside of
$B(\theta^s_1,r_1)\cup B(\theta^s_2,r_1) $.
There are four type of excursions that the orbits of $\mA(c)$ can perform
outside of this union : From $B(\theta^s_i,r_1)$ to $ B(\theta^s_j,r_1) $
for $i\in \{1,2\}$ and $j\in \{1,2\}$.
Exactly as in the single pick case, the orbits segments connecting
$B(\theta^s_i,r_1)$ to itself are contained in $B(\theta^s_i,\delta^{1/5})$.
So the claim holds, provided there exists no orbit segment 
in $s\mA(c)$ connecting 
$B(\theta^s_i,r_1)$ to $B(\theta^s_j,r_1)$ with $i\neq j$.

Assume for example that there exists  an orbit segment $\theta(t):[t_1,t_2]\lto \Tm^n$
connecting $B(\theta^s_1,r_1)$ to $B(\theta^s_2,r_1)$.
Then, given any suspended weak KAM solution $u$, the same action estimates as in the single peak case imply that
$$
u(\theta(t_2),t_2)-u(\theta(t_1),t_1)
\geq  \frac{\sqrt{\lambda\epsilon}}{8\sqrt{D}}\delta^{2/5}.
$$
Since the Aubry set is chain recurrent, there must exist an orbit segment 
$\check\theta(t):[\check t_1,\check t_2]\lto \Tm^n$ 
connecting $B(\theta^s_2,r_1)$ to $B(\theta^s_1,r_1)$, and we have
$$
u(\check\theta(\check t_2), \check t_2)-u(\check \theta(\check t_1),\check t_1)
\geq  \frac{\sqrt{\lambda\epsilon}}{8\sqrt{D}}\delta^{2/5}.
$$
By using Lemma \ref{oscillation} with  $\Omega=B(\theta^s_1,r_1)$ and  $\Omega=B(\theta^s_2,r_1)$,
we get that 
$$
u(\check\theta(\check t_2), \check t_2)-u(\theta(t_1),t_1)\leq 40\sqrt{nD\epsilon \delta/\lambda}
\quad \text{and}\quad
u(\theta(t_2),t_2)-u(\check \theta(\check t_1),\check t_1)\leq 40\sqrt{nD\epsilon \delta/\lambda}.
$$
All these inequalities together imply that 
$$
40\sqrt{nD\epsilon \delta/\lambda}\geq \frac{\sqrt{\lambda\epsilon}}{8\sqrt{D}}\delta^{2/5},
$$
which does not hold if $\delta$ is small enough.
This contradiction proves that no excursion connecting $B(\theta^s_1,r_1)$ to $B(\theta^s_2,r_1)$
can exist in the Aubry set. Note that we have used the chain recurrence of the Aubry set,
and that the conclusion does not in general apply to the Ma\~né set.
 We have proved that 
$$
s\mA(c)\subset \Tm\times \big(B(\theta^s_1,\delta^{1/5})\cup B(\theta^s_2,\delta^{1/5})\big) \times \Tm.
$$
The vertical part of the localisation follows exactly as in the single peak case.

In general, such a  localization does not hold for the Ma\~n\'e set, which may contain 
connections from one of the regions $\Tm\times B(\theta^s_i,\delta^{1/5})\times \Tm$ to the other
(but, in view of the calculations above, not in both direction).
If such a connection exists, then its $\alpha$-limit is contained in one of the domains
$\Tm\times B(\theta^s_i,\delta^{1/5})\times \Tm$, say $\Tm\times B(\theta^s_1,\delta^{1/5})\times \Tm$,
and its $\omega$-limit is containedin the other domain 
$\Tm\times B(\theta^s_2,\delta^{1/5})\times \Tm$.
Recalling that the $\alpha$ and $\omega$ limits of the Ma\~né set are contained in the Aubry set,
we conclude that each of the intersections 
$$s\mA(c)\cap \big( \Tm\times B(\theta^s_i,\delta^{1/5})\times \Tm\big)
$$
is non empty.
This proves the last part of the statement
\qed



\section{Nondegeneracy of the barrier functions}\label{sec:variational}

In this section we prove:
\begin{thm}\label{thm:totally-disc}
In the context of Theorem~\ref{intro-gen}, by possibly taking a smaller $\delta_0$, for  a residue set of 
$R\in \mR =  \mathcal{R}(r, \epsilon, \delta_0)$ the following hold: for 
any $c\in \Gamma_1$ such that $\rho(c)$ is irrational and ${\theta^f}(\mN_N(c))=\T$, 
the set $\tilde{\mN}_{N\circ \Xi}(\xi^*c)- \Xi^{-1}(\tilde{\mN}_N(c))$ is 
totally disconnected. 
\end{thm}

This is a delicate perturbation problem, and a version of it for \emph{a priori unstable} systems appeared in \cite{CY2} and was discussed in \cite{Mag}. In this 
section we give a self-contained proof with many new ingredients.

\subsection{Outline of the proof}

In this section we prove Theorem~\ref{thm:totally-disc} assuming some statements to be proven in later subsections. 
Let $L$ denote the Lagrangian associated to $N$. 
\begin{itemize}
	\item We define $\mR_1 \subset \mR(r, \epsilon, \delta)$ to be the set of $R$ such that $\theta^f(\mN_N(c)) \ne \T$ whenever $\rho^f(c)$ is rational. The set $\mR_1$ is a residue subset of $\mR$. {We also abuse notations and denote by $\mR_1$ the set of Hamiltonians of the form $N=H_0+\epsilon Z+\epsilon R, R\in \mR_1$.}
	\item We define 
	\[
		\Gamma_*(N) = \left\{ c \in \Gamma_1: \quad \theta^f(\mN_N(c)) = \T \right\},
	\]
	according to the previous item, for $N \in \mR_1$ and  $c \in \Gamma_*(N)$, we necessarily have $\rho^f(c)$ irrational. In particular, $\mA_N(c) = \mN_N(c)$ contains a unique static class. 
{In view of the upper semi-continuity of the Ma\~né set, $\Gamma_*(N)$ is a compact subset of $\Gamma_1$.}
	\item  If $N \in \mR_1$ and $c \in \Gamma_*(N)$, then the Aubry set $\tmA_{N \circ \Xi}(\xi^*c) = \Xi^{-1} \tmA_N(c)$ contains exactly two static classes denoted $\tmS_1, \tmS_2$ (with projections $\mS_1, \mS_2$). Then the Ma\~ne set is the disjoint union
	\begin{equation}
		\label{eq:disjoint-union}
		\tmN_{N\circ \Xi}(\xi^*c) = \tmS_1 \cup \tmS_2 \cup \tmH_{12} \cup \tmH_{21}, 
	\end{equation}
	where $\tmH_{12}$ (and $\tmH_{21}$) is the set of heteroclinic orbits from $\tmS_1$ to $\tmS_2$ (and vice versa). Projections are denoted $\mH_{12}, \mH_{21}$.  Note that $\tilde{\mN}_{N\circ \Xi}(\xi^*c)- \Xi^{-1}\tilde{\mN}_N(c) = \tmH_{12} \cup \tmH_{21}$.  We will also use the notations $\tmS_i(N,c)$ and $\tmH_{ij}(N,c)$ when discussing the dependence on $N, c$. 
	\item For $N \in \mR_1$ and $c \in \Gamma_*(N)$, the static classes $\tmS_1, \tmS_2$ determine two elementary forward and two backward weak KAM solutions
	\[
		h(\zeta_1, \cdot), \, h(\zeta_2, \cdot), \, \quad
		h(\cdot, \zeta_1), \, h(\cdot, \zeta_2), \, \quad \zeta_i \in \mS_i,\, i = 1,2,
	\]
	where the barrier functions are evaluated for $N\circ \Xi$ and $\xi^*c$. The associated 
	pseudographs are denoted $\mE_i(N,c)$ and $\cmE_i(N,c)$, $i =1,2$ respectively, 
	they do not depend of the choices of points  $\zeta_1\in \mS_1, \zeta_2\in \mS_2$.
	Define 
	\[
	b^-_{N,c}(\theta) = h(\zeta_1, \theta) + h(\theta, \zeta_2) - h(\zeta_1, \zeta_2)
	\]
	and $b^+_{N,c}$ similarly defined with $\zeta_1$, $\zeta_2$ switched. 
	The functions $b^{\pm}_{N,c}$ do not depend on the choice of points $\zeta_1\in \mS_1, \zeta_2\in \mS_2$,
	they are non-negative, and vanish, respectively, on $\mH_{12} \cup \mS_1 \cup \mS_2 $ and 
	$\mH_{21} \cup \mS_1 \cup \mS_2 .
	$
\end{itemize}

Given $\underline c\in \Gamma_1$, we consider the compact subset $\underline \mK\subset\Tm^n$ formed by points $\theta$
such that $d(\theta^s(\underline c),\theta^s)\geq 1/10$.
There exists $\sigma>0$ such  that the Ma\~né set $\mN(N,c)$ is disjoint from
$\underline \mK$ for each $c\in \Gamma_1\cap \B_{\sigma}(\underline c)$ and  $N\in \mR(r, \epsilon, \delta_0)$.
The compact set   $\mK=\xi^{-1}(\underline \mK)$ ($\xi$ is the double covering) is then disjoint from
$\mA_{N\circ \Xi}(\xi^*c)$.
Moreover, for these $N$ and $c$, the set $\pi^{-1}(\mK)$ intersects each orbit of 
$\tilde \mN_{N\circ \Xi} (\xi^*c)-\tilde \mA_{N\circ \Xi}(\xi^*c)$.

Since the compact interval $\Gamma_1$ is the union of finitely compact segments, each contained in a ball of the form $B_{\sigma}(\underline c)$, it suffices to prove Theorem~\ref{thm:totally-disc} for each segment. Therefore, we can assume without loss of generality that $\Gamma_1$ is actually contained in one of these balls.
Then, there exists a compact set $\mK$ such that 
\begin{itemize}
	\item  For each $c\in \Gamma_1$ and $N\in \mR(r, \epsilon, \delta_0)$,
	 $\mK$ is disjoint from $\mA_{N\circ \Xi}(\xi^*c)$ and $\pi^{-1}(\mK)$  intersects each orbit of 
$\tilde \mN_{N\circ \Xi} (\xi^*c)-\tilde \mA_{N\circ \Xi}(\xi^*c)$.
\end{itemize}

We make this additional assumption for the sequel of the section.

\begin{lem}
For each $(N,c)\in \mR_1\times \Gamma_1$, the set 
$\tilde{\mN}_{N\circ \Xi}(\xi^*c)- \Xi^{-1}\tilde{\mN}_N(c) $
is totally disconnected if and only if the set 
 $${\mN}_{N \circ \Xi}(\xi^* c) \cap  \mK
 =
 (\mH_{12}(N,c)\cup \mH_{21}(N,c))\cap \mK
 $$ 
 is totally disconnected.
\end{lem}

\begin{proof}\label{lem-pro}
The set  ${\mN}_{N \circ \Xi}(\xi^* c) \cap  \mK$ is a compact metric space, so it is totally disconnected if and only if it has  topological dimension zero, see \cite{HW}.
Assuming that this property holds, 
The set
 $\tilde \mN_{N\circ \Xi}(\xi^*c)\cap \pi^{-1}(K)$ is the disjoint union of two  homeomorphic 
 copies of  ${\mN}_{N \circ \Xi}(\xi^* c) \cap  \mK$, hence it is compact and of zero topological  dimension.
 As a consequence, each of the sets
 $\phi^k({\mN}_{N \circ \Xi}(\xi^* c) \cap  \mK), k\in \Zm$ is compact and of zero topological dimension, where $\phi^k$ is the time $k$ Hamiltonian flow of $N$.
 The countable union 
 $$\tilde \mN_{N\circ \Xi} (\xi^*c)-\tilde \mA_{N\circ \Xi}(\xi^*c)=
 \bigcup_{k\in \Zm}\phi^k({\mN}_{N \circ \Xi}(\xi^* c) \cap  \mK)
 $$
 is then also of zero dimension.
 As a consequence the projection $\mN_{N\circ \Xi} (\xi^*c)- \mA_{N\circ \Xi}(\xi^*c)$
 is of zero topological dimension, hence it is totally disconnected.
\end{proof}

We want to prove that a dense $G_{\delta}$ of Hamiltonians $N\in \mR_1$ have the property 
that ${\mN}_{N \circ \Xi}(\xi^* c) \cap  \mK$ is totally disconnected for each $c\in \Gamma_*(N)$.
The $G_{\delta}$ part follows from the next Lemma.

\begin{lem}\label{lem:G-delta}
Let {$J\subset \Gamma_1$ and $K \subset \T^{n}$ be  compact subsets}, 
then the set of $R \in \mR$ such that all $c \in (J \cap \Gamma_*(N))$ satisfies 
\begin{equation}
	\label{eq:Q}
	Q(N, c, K): = {\mN}_{N \circ \Xi}(\xi^* c) \cap  K 
\end{equation}
is totally disconnected is a $G_\delta$ set. 
\end{lem}
\begin{proof}
	Consider $N$ satisfying the conditions of the lemma, then for each $c \in (J \cap \Gamma_*(N))$,  $Q(N, c, K)$ is  compact and totally disconnected, and hence has zero topological dimension.

	Let's call a compact subset $1/k$ disconnected if it admits a finite disjoint covering by compact subsets of diameter at most $1/k$. 
	If $N$ satisfies the conditions of the Lemma, then 
	${\mN}_{N \circ \Xi}(\xi^* c) \cap  K $ is $1/k$ disconnected for each $k\in \NN$ and each $c \in (J \cap \Gamma_*(N))$. Since the Ma\~ne set is upper semi-continuous in the Hamiltonian (in the $C^2$ topology), so is ${\mN}_{N \circ \Xi}(\xi^* c) \cap  K$ and  we have, for each fixed $k$ :

	There exists an open set $\Gamma'$ containing $\Gamma_*(N)\cap J$ and  a neighborhood $\mU$ of $N$ in $C^2$ such that  the set 
	${\mN}_{N' \circ \Xi}(\xi^* c') \cap  K$
	 is $1/k$ disconnected
for all $c' \in \Gamma'$ and $N' \in \mU$.

	We now use the observation that $\Gamma_*(N)$ is upper semi-continuous in $N$, hence so is $J\cap \Gamma^*(N)$ since $J$ {is compact}.
	We deduce  the existence of
	 a smaller neighborhood $\mU' \subset \mU$ of $N$, such that $J\cap \Gamma_*(N') \subset \Gamma'$ for each $N' \in \mU'$. We have proved: the property that ${\mN}_{N \circ \Xi}(\xi^* c) \cap  K $ is $1/k$ disconnected for each $c \in \Gamma_*(N)\cap J$ is $C^2$ open (and hence $C^r$ open).
	 The Lemma follows by taking the intersection on $k$.
\end{proof}

We now adress the density part.
Let us consider the product space $C^r(\T^n \times \R^n \times \T) \times \R^n$ with the standard norms on both spaces. Define the following subset 
\[
	\mQ = \{(N, c): \, N \in \mR_1, \, c \in \Gamma_*(N)\} \subset \mR \times \Gamma_1 \subset C^r(\T^n \times \R^n \times \T) \times \R^n. 
\]

The following proposition allows us to perturb the function $b^\pm_{N,c}$ locally simultaneously 
for an open set of $c$. The proof is given in section~\ref{subsec:pert}. 
\begin{prop}\label{prop:pert-barrier}
{Let $(N_0, c_0) \in \mQ$ and $K \subset \T^n$ be a compact set disjoint from
$\mA_{N_0\circ \Xi}(\xi^*c_0)$.
Then  there exists $\sigma>0$ such that for all $N \in \mR_1 \cap B_\sigma(N_0)$,  
$\theta_0 \in K \cap \mH_{12}(N_0, c_0)$, and $\varphi \in C_c^r(B_{\sigma}(\theta_0))$ with 
$\|\varphi\|_{C^r}< \sigma$, there exists a Hamiltonian $N_\varphi$ such that:} 
\begin{enumerate}
	\item For all $c \in B_\sigma(c_0)$, the Aubry set $\tmA_{N_\varphi\circ \Xi}(\xi^*c)$ coincides with $\tmA_{N\circ \Xi}(\xi^*c)$, with the same static classes. In particular, $B_\sigma(c_0) \cap \Gamma_*(N) = B_\sigma(c_0) \cap \Gamma_*(N_\varphi)$. 
	\item For all $c \in B_\sigma(c_0)\cap \Gamma_*(N)$, there exists a constant $e\in \Rm$ such that 
	\begin{equation}
		\label{eq:b-plus-pert}
		b_{N_\varphi, c}^+ (\theta) = b_{N,c}^+(\theta) + \varphi(\theta)+e, \quad \theta \in B_{\sigma}(\theta_0).
	\end{equation}	
\end{enumerate}
The same holds for $\theta_0 \in K \cap \mH_{21}(N_0, c_0)$, with $b^+$ replaced 
with $b^-$ in \eqref{eq:b-plus-pert}. 
{Moreover, for each $N\in \mR_1\cap B_{\sigma}(N_0)$, $\|N_{\varphi}-N\|_{C^r}\lto 0$ 
when $\|\varphi\|_{C^r} \lto 0$.}
\end{prop}

We will use Proposition~\ref{prop:pert-barrier} to perturb all barrier functions near a given $c_0$ simultaneously. Because we are perturbing an uncountable family of functions, we need an additional information on how the functions $b^\pm_{N,c}$ depends on $c$. The proof is given in Section \ref{subse:haus}. 

\begin{prop}\label{prop:hd}
For each $N\in \mR_1$, 
the maps $c\lmto b^+_{N,c}, b^-_{N,c}$ are $1/2$-Hölder from $\Gamma_*(N)$ to $C^0(\T^n,\Rm)$.
\end{prop} 

This regularity implies that the set $\{b^{\pm}_{N,c}, c\in \Gamma^*(N)\}$ is compact and has 
Hausdorff dimension at most 2 in $C^0(\T^n,\Rm)$. The following Lemma will allow to take 
advantage of this fact:

\begin{lem}\label{lem:disc}
Let $\mF\subset C^0([-1,1]^n,\Rm)$ be a compact set of finite Hausdorff dimension.
The following property is satisfied on a residue set of functions $\varphi\in C^r(\Rm^n,\Rm)$ (with the uniform $C^r$ norm):

For each $f\in \mF$, the set of minima of the function $f+\varphi$ on $[-1,1]^n$ is totally disconnected.

As a consequence, for each open neighborhood $\Omega$ of $[-1,1]^n$ in $\Rm^n$, there exists arbitrarily $C^r$-small 
compactly supported functions $\varphi : \Omega\lto \Rm$ satisfying this property.
\end{lem}

\begin{proof}
We first consider the case $n=1$
The set $\tilde \mF=\{c-f, f\in \mF, c\in \Rm\}$ is compact and of finite Hausdorff dimension (one more than the dimension of $\mF$).
For each compact subinterval $J\subset [-1,1]$, the set $\tilde \mF_J\subset C(J,\Rm)$ is also compact and finite dimensional,
since the restriction map is Lipschitz.
If $J$ is non trivial, the complement
$$\Phi(J):=C^r(\Rm,\Rm) -(\tilde \mF_J\cap C^r(\Rm,\Rm)) 
$$ 
is open and dense in $C^r(\Rm,\Rm)$.
To prove density, we consider a subspace $H\subset C^r(\Rm,\Rm)$ of finite dimension larger that the Hausdorff dimension 
of $\tilde \mF$.
We moreover assume that all functions of $H$ are compactly supported inside the interior of $J$.
 Given $\varphi \in C^r(\Rm,\Rm)$, we consider the affine space $\varphi+H$.
 Considering the $C^0([-1,1],\Rm)$ distance, the Hausdorff dimension of $\tilde \mF_J\cap (\varphi+H)$ is not greater than the Hausdorff dimension of $\tilde \mF$,
hence it is less than the dimension of $H$. This implies that the complement  $(\varphi+H)-\tilde \mF$ is dense in $\varphi+H$
endowed with the $C^0$ distance. Since the $C^0$ and $C^r$ norms are equivalent on the finite dimensional space $\varphi+H$,
we conclude 
  that $\varphi$ belongs to the closure of $\Phi(J)$ in $C^r(\Rm, \Rm)$.

Let $J_k$ be a sequence of compact subintervals of $[-1,1]$ such that each open interval contains one of the $J_k$.
Then if $\varphi \in \cap_k \Phi(J_k)$ (this intersection is a dense $G_{\delta}$),
 each of the functions $f+\varphi, f\in \mF$ has the property that it is not constant on any open interval, 
hence its set of minima in $[-1,1]$ is totally disconnected. 

Let us now turn to the general case.
We denote by $\pi_i:[-1,1]^n\lto[-1,1]$ the projections on the factors.
We associate to each function $f\in C^0([-1,1]^n,\Rm)$ the functions 
$$f_i:[-1,1]\ni x_i \lmto f_i(x_i)=\min_{\pi_i(x)=x_i}f(x) .
$$
For each $k$ and $i$, the following property holds on an open and dense subset of functions $\varphi\in C^r(\Rm^n,\Rm)$:
None of the functions  $(f+\varphi)_i, f\in \mF$ is constant on $J_k$.

 To prove density, we consider a function $\varphi \in C^r(\Rm^n,\Rm)$.
The map $f\lmto f_i$ is Lipschitz hence the set
$\mF_i(\varphi)=\{(f+\varphi)_i,f\in \mF\}\subset C^0([-1,1], \Rm)$ is compact and has finite Hausdorff dimension.
We can aplpy the result for $n=1$ to this family and obtain that for generic  $\varphi_1\in C^r(\Rm,\Rm)$, none of the functions 
$$(f+\varphi)_i+\varphi_i=(f+\varphi+\varphi_i)_i
$$
for $f\in \mF$ is  constant on the interval $J_k$.

By taking the intersection on $n$ and $k$,
we obtain that, for generic $\varphi\in C^r(\Rm^n,\Rm)$,
each of the functions $(f+\varphi)_i$ has a totally disconnected set of minima in $[-1,1]$.

Since $\pi_i(\argmin (f+\varphi))\subset \argmin (f+\varphi)_i$, this implies that $\argmin (f+\varphi)$ 
is totally disconnected.
\end{proof}

\begin{proof}[Proof of Theorem~\ref{thm:totally-disc}]
	
	Let $\mR_2\subset \mR_1$ be the set of Hamiltonians $N$ which have the property that 
	${\mN}_{N \circ \Xi}(\xi^* c) \cap  \mK $ is totally disconnected for each $c\in \Gamma_*(N)$.
	
	By Lemma \ref{lem-pro}, it is enough to prove that $\mR_2$ is a dense $G_{\delta}$. By Lemma 
	\ref{lem:G-delta}, $\mR_2$ is a $G_{\delta}$, we have to prove density.
	
Let us fix $N_0\in \mR_1$.
 For each $\theta_0\in{\mN}_{N \circ \Xi}(\xi^* c) \cap  \mK $, we consider  $\sigma>0$
 small enough so that 
Proposition \ref{prop:pert-barrier} applies.
We define  the cube 
$$D_{\sigma}(\theta_0) = \{\theta : \max_i|\theta^i - \theta_0^i| \le \sigma/{(2\sqrt{n})}\} \subset B_{\sigma}(\theta_0).
$$
In view of Proposition \ref{prop:hd}, we can apply Lemma \ref{lem:disc} to the family of functions 
$b^{\pm}_{N,c}, c\in \Gamma_1\cap \Gamma_*(N)$ on the cube $D_{\sigma}(\theta_0)$ for each $N\in \mR_1$.
We find arbitrarily small functions $\varphi$ compactly supported in $B_\sigma(\theta_0)$ and such that each of  the functions 
$b^{\pm}_{N,c}+\varphi, c\in \Gamma^*(N)\cap \Gamma_1$
have a totally disconnected set of minima in $D_{\sigma}(\theta_0)$. 
If $N\in \mR_1\cap B_{\sigma}(N_0)$, we can apply 
 Proposition \ref{prop:pert-barrier} to get Hamiltonians $N_{\varphi}$ approximating $N$.
We obtain:
\begin{itemize}
\item
The set of Hamiltonians $N$ such that $ {\mN}_{N \circ \Xi}(\xi^* c) \cap D_{\sigma}(\theta_0)$ is totally disconnected for each $c\in \Gamma_*(N)$
is dense in $\mR_1\cap B_{\sigma}(N_0)$. By Lemma \ref{lem:G-delta}, it is a $G_{\delta}$.
\end{itemize}

	Since $\mK$ is compact, there is a finite cover $\mK \subset \bigcup_{i=1}^k D_{\sigma_i}(\theta_i)$, 
	such that the above can be applied on each $D_{\sigma_i}(\theta_i)$
 some  constant $\sigma_i>0$. For $\sigma_0 = \min \sigma_i>0$, we obtain:
	\begin{itemize}
		\item For a residue set of $N \in B_{\sigma_0}(N_0)$, the set    $ {\mN}_{N \circ \Xi}(\xi^* c) \cap D_{\sigma_i}(\theta_i)$  is totally disconnected for all $i=1, \ldots, k$ and $c \in \Gamma_*(N)$.
	\end{itemize}
	Taking the intersection over $i$,  we obtain :  
\begin{itemize}
		\item For a residue set of $N \in B_{\sigma_0}(N_0)$, the set     $ {\mN}_{N \circ \Xi}(\xi^* c) \cap \mK$ is totally disconnected for all $c \in \Gamma_*(N)$.
	\end{itemize}
	In particular, $N_0$ is in the closure of $\mR_2$.
\end{proof}

\subsection{Perturbing the Peierls' barrier functions}
\label{subsec:pert}

Let $L$ be the Lagrangian for $N = H_0 + \epsilon Z + \epsilon R$. We define the generating function $\R^n \times \R^n \to \R$  by 
\[
	G_N(x, x') = \min\int_0^1 L(\gamma, \dot{\gamma}, t) dt, \quad \gamma(0) = x, \quad \gamma(1) = x'.
\]
Note that $G_N(x + k, x' +k) = G_L(x, x')$ for all $x, x' \in \R^n$ and $k \in \Z^n$. 
If $\epsilon$ is sufficiently small, there is a one-to-one correspondence between the time-1 map of the Euler-Lagrange flow of $L$, and the generating function $G$.  We will also consider the generating function of the Hamiltonian $N\circ \Xi$ (pull back of the double covering), which satisfies 
\begin{equation}
	\label{eq:lift-G}
	G_{N \circ \Xi}(x, x')  = G_N(\xi x, \xi x'),  
\end{equation}
where we have lifted $\xi$ to a map $\R^n \to \R^n$. It is important to keep in mind that $G_{N\circ \Xi}$ has an additional symmetry $G_{N\circ \Xi}(x, x') = G_{N\circ \Xi}(x + \frac12 e_1, x' + \frac12 e_1)$ where $e_1 = (1, 0, \cdots, 0)$, corresponding to the deck transformation of $\xi$. We also denote 
\[
	A_{N, c}^M(\theta_1, \theta_2) = 
	\min \int_0^M L(\gamma, \dot{\gamma}, t) - c \cdot \dot{\gamma} + \alpha_N(c)\, dt, \quad \gamma(0) = \theta_1, \, \gamma(M) =\theta_2 \in \T^n,
\]
and note that $A_{N,c}^M$ and therefore $h_{N, c}$ is completely determined by $G_N$. We will perturb the barrier functions by perturbing $G_N$. 

Let $U, V \subset \R^n$  be open sets which projects injectively to $\T^n$, namely $U \cap (U + k) = \emptyset$ for all $k \in \Z^d$. We define a \emph{perturbation block} to be the set 
\[
	\mB_N(U, V) : = \phi_N(U \times \R^n) \cap (V \times \R^n)  \subset \R^n \times \R^n, 
\]
in other words, the set of $(\theta, p)$ such that $\theta \in V$ and $\pi_\theta \Phi_N^{-1}(\theta, p) \in U$, where $\phi_N$ is the time-$1$-map of the Hamiltonian $N$. We can also consider $\mB_N$ as a subset of $\T^n \times \R^n$ since $V$ projects injectively to $\T^n$. 

Given $U_1 \subset U_2 \subset \R^n$ and $V \subset \R^n$ as before, for $\varphi \in C_c^r(V)$, we define a perturbation of the generating function (depending on $\varphi$, $U_1, U_2, V$) as follows: 
\begin{equation}
	\label{eq:G-varphi}
	G_\varphi(x, x') = G_N(x, x') + \rho(x)\varphi(x'),
\end{equation}
and extends it by periodicity $G_\varphi(x + k, x' + k) = G_\varphi(x, x')$ for all $k \in \Z^n$. Here $\rho: \R^n \to \R^+ \cup \{0\}$ is a standard mollifier function such that 
\[
	\rho|_{U_1} = 1, \quad \rho|_{(U_2)^c} = 0. 
\]
\begin{lem}\label{lem:susp}
When $\|\varphi\|_{C^r}$ is small enough, there exists a Tonelli Hamiltonian $N_\varphi$ whose generating function is equal to $G_\varphi$. Moreover, $\|N_\varphi - N\|_{C^r} \to 0$ as $\|\varphi\|_{C^r} \to 0$. 
\end{lem}
\begin{proof}
Let $g(x,x') = \rho(x)\varphi(x')$, extended by periodicity, then $\|g\|_{C^r} \le C \|\varphi\|_{C^r}$ for some $C>0$ depending on $\rho$. Let $G_t(x,x')$ be the generating function of the time-$t$ map of the Hamiltonian $N$, we consider the following functions 
\[
	G_t'(x, x') = G_t(x, x') + s(t)g(x, x'),
\]
where $s:[0,1] \to [0,1]$ is a $C^\infty$ mollifier function with $s(t) = 0$ on $[0, \frac13]$ and $s(t) = 1$ on $[\frac23, 1]$. When $\|g\|_{C^2}$ is small enough, the functions $G_t'$ uniquely determines exact symplectic maps $\psi_t: \T^n \times \R^n \to \T^n \times \R^n$. 

It's easy to see that there exists an exact symplectic isotopy between  $\psi_t$ and $\phi_t$, then there is an exact symplectic isotopy between $(\phi_t)^{-1}\psi_t$ and $id$. In view of Proposition 9.19 and Corollary 9.20 of \cite{MDS}, we get $\{\psi_t\}_{0 \le t \le 1}$ is a Hamiltonian isotopy. Moreover, since $\frac{d}{dt}\psi_t$ is periodic in $t$, it must be generated by a time periodic Hamiltonian $N'(\theta, p, t)$. The maps are $C^{r-1}$ in $(\theta, p)$ and $C^\infty$ in $t$, the vector fields are $C^{r-1}$ and the Hamiltonians are $C^r$. 

Moreover, it's easy to see that $\psi_t (\phi_t)^{-1}$ converges in $C^{r-1}$  to identity uniformly over $t$ as $\|g\|_{C^r} \to 0$. Since $\psi_t (\phi_t)^{-1}$ has the Hamiltonian function $-N_t \circ \phi_t + N_t' \circ \phi_t$ (see \cite{MDS} Proposition 10.2) we conclude  that $\|N_t - N_t'\|_{C^r} \to 0$ as $\|g\|_{C^r} \to 0$. 
\end{proof}

The following lemma prepares us for the perturbation. For an orbit contained in the psudograph $\overline{\mE_1(N, c)}$, there exists a perturbation block that the orbit of $(\theta, p)$ never returns to in backward time. Moreover, the orbit also does not return to the ``copy'' of the perturbation block under the deck transformation of $\Xi$. This is important because we would like to perturb the generating function $G_{N\circ \Xi}$ by perturbing only $N$. 
\begin{lem}\label{lem:local-back}
Consider $(N_0, c_0) \in \mQ$, and $(\theta_0, p_0) \in \tmH_{12}(N_0, c_0)$. Then there exists $\sigma >0$, and  open sets $V \ni \theta_0$ and $U_1 \subset U_2 \subset \R^n$, such that 
\begin{itemize}
	\item The covering map $\xi: \T^n \to \T^n$ is injective on $\overline{U_2}, \overline{V}$. 	
	\item $\overline{U_2} \cup (\overline{U_2} + \frac12 e_1)$, $\overline{V} \cup (\overline{V} + \frac12 e_1)$ are disjoint from $\mA_{N_0 \circ \Xi}(\xi^*c)$. 
\end{itemize}
The following hold for each $(N,c) \in \mQ \cap B_\sigma(N_0, c_0)$.

\begin{enumerate}
	\item For $\theta \in V$,  let $(\theta, p)$ be contained in the closure of the psudograph $\overline{\mE_1(N,c)}$. 
	\begin{enumerate}
		\item  $(\theta, p) \in \mB_{N \circ \Xi}(U_1, V)$. 
		\item The backward orbit $\phi^{-k}_{N\circ \Xi}(\theta, p)$ is asymptotic to $\tmS_1(N,c)$. 
		\item For $k \ge 1$,  $\phi_{N\circ\Xi}^{-k}(\theta, p)$ is not contained in $\mB_{N\circ \Xi}(U_2, V)$ or $\mB_{N\circ \Xi}(U_2 + \frac12 e_1, V + \frac12 e_1)$. 
	
	\end{enumerate}
	\item  For $\theta \in V$, let $(\theta, p)$ be contained in the closure of the psudograph $\overline{\cmE_2(N,c)}$. 
	\begin{enumerate}
		\item The forward orbit $\phi^k_N(\theta,p)$ is asymptotic to $\tmS_2$. 		
		\item For $k \ge 1$, $\phi_N^k(\Xi(\theta, p))$ is not contained in  $\mB_{N\circ \Xi}(U_2, V)$ or $\mB_{N\circ \Xi}(U_2 + \frac12 e_1, V + \frac12 e_1)$
	\end{enumerate}
\end{enumerate}
Moreover, an analogous statement holds for $\mH_{21}$, where the roles of $\mE_1$, $\cmE_2$ are replaced by $\mE_2$ and $\cmE_1$. 
\end{lem}

\begin{proof}
	First we claim: for any $\iota>0$, there is $\sigma>0$ such that: if $\|\theta - \theta_0\|< \sigma$, $(N, c) \in B_\sigma(N_0, c_0) \cap \mQ$, then $(\theta, p) \in \tmE_1(N,c)$ implies: 
	\begin{enumerate}[(c1)]
		\item $\|p - p_0\|< \iota$. 
		\item The backward orbit $\phi^{-k}_{N\circ \Xi}(\theta, p)$ is asymptotic to $\tmS_1(N,c)$.
		\item There exists $M >0$ such that $k > M$ implies $\dist(\phi^{-k}_{N\circ \Xi}(\theta, p), \tmS_1(N, c))< \iota$. 
	\end{enumerate}
	We note that $\theta_0 \in \mH_{12}(N_0, c_0)$ implies the weak KAM solution $h(\zeta_1, \cdot)$ is differentiable at $\theta_0$, and therefore $p_0$ is the unique super-differential. Item (c1) then follows from semi-continuity of super-differentials, see Proposition~\ref{prop:cont}. 

	Since $\theta_0 \in \mH_{12}(N_0, c_0)$, we have  for $h = h_{N_0 \circ \Xi, \, \xi^*c}$
	\begin{equation}
		\label{eq:theta0-act}
		h(\zeta_1, \theta_0) + h(\theta_0, \zeta_2) = \min_{\theta}\left( h(\zeta_1, \cdot) + h(\cdot, \zeta_2) \right) = h(\zeta_1, \zeta_2).
	\end{equation}
	Assume by contradiction that for $(N_k, c_k) \to (N_0, c_0)$ in $\mQ$,  and $(\theta_k, p_k) \in \mE_1(N_k, c_k)$ with $\theta_k \to \theta_0$, the backward orbit of $(\theta_k, p_k)$ accumulates to $\mS_2(N_k, c_k)$. This implies 
	\[
		h_{N_k \circ \Xi, \, \xi^* c_k}(\zeta^k_1, \theta_k) = h_{N_k \circ \Xi, \, \xi^* c_k}(\zeta^k_1, \zeta^k_2)  + h_{N_k \circ \Xi, \, \xi^* c_k}(\zeta^k_2, \theta_k), \quad \zeta^k_1 \in \mS_1, \zeta^k_2 \in \mS_2. 
	\]
	Taking limit as $k \to \infty$ (by Proposition~\ref{prop:cont}), we obtain 
	\[
		h_{N_0 \circ \Xi, \, \xi^* c_0}(\zeta_1, \theta_0) = h_{N_0 \circ \Xi, \, \xi^* c_0}(\zeta_1, \zeta_2)  + h_{N_0 \circ \Xi, \, \xi^* c_0}(\zeta_2, \theta_0), \quad \zeta_1 \in \mS_1, \zeta_2 \in \mS_2. 
	\]
	Combine with \eqref{eq:theta0-act} we get (omitting the subscript of $h$)
	\[
		h(\zeta_1, \zeta_2) = h(\zeta_1, \theta_0) + h(\theta_0, \zeta_2) =h(\zeta_1, \zeta_2)  + h(\zeta_2, \theta_0) + h(\theta_0, \zeta_2),
	\]
	or $h(\zeta_2, \theta_0) + h(\theta_0, \zeta_2) =0$  this is a contradiction with $\theta_0 \notin \mS_2$. 

	To prove (c3) we again argue by contradiction. Let $N_k, c_k, \theta_k, p_k$ be as before, we assume that there exists $M_k \to \infty$  such that $\dist(\phi^{-M_k}_{N\circ \Xi}(\theta_k, p_k), \tmS_1(N, c))\ge  \epsilon$. Denote $m_k = \pi \phi^{-M_k}_{N\circ \Xi}(\theta_k, p_k)$, using the fact that backward orbit of $(\theta_k, p_k)$ is calibrated, we have 
	\[
		h_{N_k \circ \Xi, \, \xi^* c_k}(\zeta_1, \theta_k) = h_{N_k \circ \Xi, \, \xi^* c_k}(\zeta_1, m_k) + A^{M_k}_{N_k \circ \Xi, \, \xi^* c_k}(m_k, \theta_k).
	\]
	Up to taking a subsequence, assume $m_k \to m_0$, take limit as $k \to \infty$, we obtain 
	\[
		h(\zeta_1, \theta_0) \ge h(\zeta_1, m_0) + h(m_0, \theta_0) = h(\zeta_1, m_0) + \min_{i =1,2}\left( h(m_0, \zeta_i) + h(\zeta_i, \theta_0) \right),
	\]
	where $h$ are evaluated at $N_0 \circ \Xi, \, \xi^* c_0$. 
	Since $h(\zeta_1, m_0) + h(m_0, \zeta_1) >0$, the above minimum is not reached at $\zeta_1$. Therefor $h(\zeta_1, \theta_0) \ge h(\zeta_1, m_0) + h(m_0, \zeta_2) + h(\zeta_2, \theta_0) \ge h(\zeta_1,\zeta_2) + h(\zeta_2, \theta_0)$, but we showed (in the proof of (c2)) this is also impossible. 

	We now define the sets $U, V$. 
	Since $\phi_{N_0\circ \Xi}^{-k}(\theta_0, p_0)$ is asymptotic to $\mS_1(N_0, c_0)$, project via $\Xi$ implies $\phi_{N_0}^{-k}(\Xi(\theta_0, p_0))$ is asymptotic to $\Xi (\mS_1) = \mA_{N_0}(c_0)$. There exists $\iota_1 >0$ such that 
	\[
		\phi^{-k}_N (\Xi(\theta_0, p_0)) \cap \Xi(B_\epsilon(\theta_0, p_0)) = \emptyset,
	\]
	and $\xi(B_{\iota_1}(\theta_0)) \cap \mA_{N} = \emptyset$ for all $N \in B_{\iota_1}(N_0) \cap \mR_1$. 

	Apply claim (c1)-(c3) to $\iota = \iota_1/2$, and obtain the parameters $\sigma, M$. Since the orbit of $(\theta_0, p_0)$ is wondering, there exists $0 <\sigma_1 < \sigma$ such that $(\theta,p)\in B_{\sigma_1}(\theta_0, p_0)$, $N \in B_{\sigma_1}(N_0)$ implies 
	\[
		\phi_N^{-k}\left( \Xi(B_{\sigma_1}(\theta_0, p_0)) \right) \cap \Xi(B_{\sigma_1}(\theta_0, p_0)) = \emptyset, \quad 
			1 \le k \le M.    
	\]
	apply the relation $\Xi \circ \phi_{N\circ \Xi}= \phi_N \circ \Xi$ we get 
	\begin{equation}
		\label{eq:disjoint}
		\phi_{N\circ \Xi}^{-k}\left( \Xi(B_{\sigma_1}(\theta_0, p_0)) \right) \cap \Xi^{-1} \Xi(B_{\sigma_1}(\theta_0, p_0)) = \emptyset, \quad 1 \le k \le M. 
	\end{equation}

	For a later determined $\sigma_2 < \sigma_1$, choose  $\sigma_3<\sigma_2$ using claim (c1) again to ensure any $(\theta, p) \in \mE_1(N,c)$ with $\| \theta - \theta_0\| < \sigma_3$ implies $\|p - p_0\|<\sigma_2$. Define $V = B_{\sigma_3}(\theta_0)$, 
	\begin{equation}
		\label{eq:U1}
		U_1 = \bigcup_{N \in B_{\sigma_3}(N_0)} \pi \phi^{-1}_{N \circ \Xi}(B_{\sigma_3}(\theta_0) \times B_{\sigma_2}(p_0)),
	\end{equation}
	$U_2 = B_{\sigma_2}(U_1)$. Since $U_1 \to \pi \phi_{N_0\circ \Xi}^{-1}(\theta_0, p_0)$, as $\sigma_2, \sigma_3 \to 0$, we can choose $\sigma_2, \sigma_3$ small enough such that 
	\[
		\mB_{N \circ \Xi}(\overline{U_2}, \overline{V}) \subset B_{\sigma_1}(\theta_0, p_0), \quad 
		\forall N \in B_{\sigma_3}(N_0).    	
	\]
	We now verify that for $\theta \in V$ and $(\theta, p) \in \overline{\mE_1(N,c)}$, $\phi^{-1}_{N\circ \Xi}(\theta, p) \in U_1$ due to \eqref{eq:U1}. Moreover, since 
	\[
	 \mB_{N\circ \Xi}(\overline{U_2}, \overline{V}) \, \cup \,  \mB_{N\circ \Xi}(\overline{U_2} + \frac12 e_1, \overline{V} + \frac12 e_1) \subset \Xi^{-1} \Xi B_{\sigma_1}(\theta_0, p_0),
	\]
	\eqref{eq:disjoint} implies 1(c) for $1 \le k \le M$. On the other hand, (c3) ensures the same for $k > M$ as well. 

	The proof of 2(a)(b) and the moreover part is analogous and we omit it. 
\end{proof}

\begin{proof}[Proof of Proposition~\ref{prop:pert-barrier}]

	Given $\theta_0 \in K \cap \mH_{12}(N_0, c_0)$, let $(\theta_0, p_0)$ be the corresponding point in $\tmH_{12}(N_0, c_0)$. Choose $\sigma >0$,  $U_1, U_2, V$ as in Lemma~\ref{lem:local-back}. For  $\varphi \in C_c^r(\xi V)$, consider perturbation $N_\varphi$ via \eqref{eq:G-varphi} using the neighborhoods $\xi U_1, \xi U_2, \xi V$. Note that for $W = U_i, V$, we have  $\xi^{-1} \xi W = W \cup (W + \frac12 e_1)$ and we will use this notation throughout the proof. 
{First, notice that according to Lemma~\ref{lem:susp}, $\|N_\varphi - N\|_{C^r} \to 0$ as $\|\varphi\|_{C^r}\to 0$. }

	\emph{Item 1}. 
	We first show that the perturbation $N_\varphi$ does not affect Aubry set and static classes. Lemma~\ref{lem:local-back} asserts $\xi^{-1}\xi\overline{U_2}, \xi^{-1}\xi\overline{V}$ are disjoint from $\mA_{N_0 \circ \Xi}(\xi^* c_0)$. For $(N, c) \in B_\sigma(N_0, c_0) $ and $\sigma$ small enough, using semi-continuity, $\xi^{-1}\xi\overline{U_2}, \xi^{-1}\xi\overline{V}$are disjoint from $\mA_{N\circ \Xi}(\xi^* c)$ and $\mA_{N_\varphi\circ \Xi}(\xi^* c)$. Then \eqref{eq:G-varphi} and \eqref{eq:lift-G} implies the $L_{N\circ \Xi}$ action and $L_{N_\varphi\circ \Xi}$ action coincide on orbits of $\tmA_{N \circ \Xi}(\xi^*c)$ and $\tmA_{N_\varphi \circ \Xi}(\xi^*c)$. As a result $\tmA_{N \circ \Xi}(\xi^*c)$ and $\tmA_{N_\varphi \circ \Xi}(\xi^*c)$ must coincide with the same static classes. 

	\emph{Item 2}. We proceed to prove \eqref{eq:b-plus-pert}. Let $(\theta, p) \in \overline{\mE_1(N,c)}$, then $\gamma(t):=\pi_\theta \circ \phi^t(\theta, p)$ is a calibrated orbit (on $(-\infty, 0]$) for the weak KAM solution $h_{N_\varphi \circ \Xi, \, \xi^*c}(\zeta_1, \cdot)$, with $\zeta_1 \in \mS_1$. Write $\gamma_t = \gamma(t)$.  Since $\gamma(t)$ is backward asymptotic to $\mS_1$, there is $i_k \to \infty$ such that 
	\begin{equation}
		\label{eq:N-varphi-action}
		\begin{aligned}
			&h_{N_\varphi \circ \Xi, \, \xi^*c}(\zeta_1, \theta) = \lim_{k \to \infty} A^{i_k}_{N_\varphi \circ \Xi, \, \xi^*c}(\gamma_{-i_k}, \gamma_0) \\
			&= \lim_{k \to \infty} \sum_{j = - i_k}^{-1} \left( G_{N_\varphi \circ \Xi}(\gamma_j, \gamma_{j+1}) - \xi^*c \cdot (\gamma_{j+1}- \gamma_j) + \alpha_{N_\varphi\circ \Xi}(\xi^*c)  \right),
		\end{aligned}
	\end{equation}
	where in the last line $\gamma$ is lifted to $\R^n$. 
	In view of 1(c) and \eqref{eq:G-varphi}, for any $j \le  -2$, we have 
	\[
			G_{N_\varphi \circ \Xi}(\gamma_j, \gamma_{j+1}) = G_{N_\varphi}(\xi \gamma_j, \xi \gamma_{j+1}) = G_N(\xi\gamma_j, \xi\gamma_{j+1})  
		= G_{N\circ \Xi}(\gamma_j, \gamma_{j+1}). 
	\]
	By the same reasoning, we have 
	\[
		G_{N_\varphi \circ \Xi}(\gamma_{-1}, \gamma_0) = G_{N \circ \Xi}(\gamma_{-1}, \gamma_0) + \rho(\gamma_{-1})\varphi(\gamma_0) 
		= G_{N\circ \Xi}(\gamma_{-1}, \gamma_0).		
	\]
	Using \eqref{eq:N-varphi-action}, we get 
	\[
		h_{N_\varphi \circ \Xi, \, \xi^*c}(\zeta_1, \theta) = \lim_{k \to \infty} A^{i_k}_{N \circ \Xi, \, \xi^*c}(\gamma_{-i_k}, \gamma_0) \le h_{N\circ \Xi, \, \xi^*c}(\zeta_1, \theta). 
	\]
	Observe that the previous arguments holds when $N_\varphi$ and $N$ are switched, the last displayed formula becomes an equality. By the same reasoning, using Lemma~\ref{lem:local-back}, 2(a),(b), we obtain 
	\[
		h_{N_\varphi\circ \Xi,\, \xi^*c}(\theta, \zeta_2) = h_{N\circ \Xi, \, \xi^*c}(\theta, \zeta_2), \quad \zeta_2 \in \mS_2.  
	\]
	These \eqref{eq:b-plus-pert} follows. The proof for $b^-$ is identical with two static classes switched. 
\end{proof}

\subsection{Hölder continuity of the barrier functions}
\label{subse:haus}

We prove Proposition \ref{prop:hd} by relating the barriers  to the stable and unstable manifolds of the Aubry sets. 

Recall that the system $N$ admit a weakly invariant cylinder $\mC$ which contains the Aubry set $\tmA_{N}(c)$ for $c \in \Gamma_1$. Using the covering map $\Xi$, 
we obtain $\Xi^{-1} \mC = \mC_1 \cup \mC_2$ 
and denote $\tmS_i(N, c) = \mC_i\cap \Xi^{-1}(\tilde \mA(c))$, $i =1,2$ for all
  $c \in \Gamma_*(N)$.

Recall that $\Gamma_*(N)$ is the set of $c\in \Gamma_1$ such that $\mA_N(c)$ is an invariant curve contained in $\mC$. Let $c^\pm$ be the $c\in \Gamma_*(N)$ with the smallest and largest $p^f$ component. Then the component of $\mC$ bounded by $\mA_N(c^\pm)$ is an invariant set for $\phi_N$, we denote it $\Lambda_*$. Let $\Lambda_1, \Lambda_2$ be the lifts under $\Xi$, then  $\Lambda_i \subset \mC_i$ are normally hyperbolic
invariant manifolds  for $\phi_{N \circ \Xi}$. 

They admit $C^2$ center stable and center unstable manifolds $W^{cs/cu}$, which are locally graphs above $(\theta,p^f)$.
These manifolds are foliated by the strong stable and unstable manifolds $W^{s,u}(z)$ of the points of $\Lambda_i$, 
see Appendix \ref{sec:abstract-nhic}.
The leaves $W^{s,u}(z)$ of this foliation are $C^2$, they are locally graphs above $\theta^s$. The foliation itself is $C^1$.

Consider $c \in \Gamma_*(N)$, then for $i = 1,2$,  $\tmS_i(N,c)$ is a Lipshitz invariant curve. Define the sets 
\[
	W^{u/s}_i(N,c) = \bigcup_{z \in \tmS_i(N,c)} W^{u/s}(z). 
\]
 Since $\tmS_i(N,c)$ are Lipshitz graphs over $\theta^f$,
 and since $W^{u,s}$ are a $C^1$ foliation whose leaves are graphs over $\theta^s$,
  $W^{u/s}_i(N,c)$ are  Lipshitz graphs over $\theta$  in a neighborhood of $\tmS_i$. We will show that they coincides with the pseudographs $\mE_i(N,c)$ in a neighborhood of $\mS_i(N,c)$. 

\begin{lem}\label{lem:back-min}
	For $i, j = 1,2$, if $(\theta, p) \in \overline{\mE_i(N, c)}$ is backward asymptotic to $\mS_j(N,c)$, then there exists $M>0$ such that 
	 $\phi^{-k}_{N\circ \Xi} \in W^u_j(N,c)$ for each $k>M$.
\end{lem}
Suppose an orbit is backward asymptotic to $\mS_1(N,c)$, then it is asymptotic to the normally hyperbolic set $\Lambda_1$. 
This orbit is contained in the strong manifold of a point $z' \in \Lambda_1$
which is asymptotic to $\mS_1(N,c)$, but which in principle may not belong to $\mS_1(N,c)$.
 To prove that $z'\in \mS_1(N,c)$,
  we need an argument similar to Theorem 1.4.

We need the following version of Proposition 4.3.
\begin{prop}\label{prop:strong-lip}
Suppose $k \ge 1/\sqrt{\epsilon}$, then for each semi-concave function $u_0$, the function $u_k  = T_c^k u_0$ is $6D \sqrt{\epsilon}-$semi-concave and $6D\sqrt{n\epsilon}-$Lipschitz. Similar statement holds for $\check{T}^k_c u$. As a result, for any weak KAM solution $u$ and $k \ge 1/\sqrt{\epsilon}$, the set 
\[
	\phi_N^{-k}(\overline{\mG_{c,u}})
\]
is a $6D\sqrt{\epsilon}-$Lipschitz graph over the $\theta$ component. 
\end{prop}
\begin{proof}
	We observe that the proof of Proposition 4.3 applies as long as we replace $u(\theta)$ by $u_k$ and $u(\Theta(0))$  by $u_0(\Theta(0))$. The assumption $k \ge \frac{1}{\sqrt{\epsilon}}$ ensures we can choose $T \in [1/2\sqrt{\epsilon}, 1/\sqrt{\epsilon}]$ in that proof. 

	For the second part, observe that 
	\[
		\phi_N^{-k}(\overline{\mG_{c,u}}) \subset  \mG_{c,u} \tilde\wedge \check{\mG}_{c, \check{T}_c^k u}	
	\]
	and the proof is similar to Theorem 4.1. 
\end{proof}

For the rest of this section, $\phi$ denotes $\phi_{N\circ \Xi}$.

\begin{proof}[Proof of Lemma~\ref{lem:back-min}]
	We only prove for the case $i=j=1$ as the others are similar.	Since $z := (\theta, p)$ is backward asymptotic to $\mS_1(N,c) \subset \Lambda_1$, then 
there exists $z_1 \in \Lambda_1$ such that $(\theta, p) \in W^u(z_1)$. Necessarily $\phi^{-k}(z_1)$ converges to $\mS_1(N,c)$. We will show $z_1 \in \mS_1(N,c)$.

	Arguing by contradiction, suppose $z_1 \notin \mS_1(N,c)$, then using the fact that $T\mC_1$ is the central direction, $\dist(\phi^{-k}z_1, \mS_1(N,c))$ converges at a  maximal rate of $\rho^n$. 

	Denote $z_1^k = \phi^{-k}(z_1) $, $\mS_1(N,c)$ projects onto $\theta^f$ component, for any $k \in \NN$, there is $z_2^k  \in \mS_1(N,c)$ such that $\theta^f(z_1^k) = \theta^f(z_2^k)$. According to Theorem~\ref{nhic-mult}, there exists $D_1 >1$ such that $\mC$ is an $D_1/\sqrt{\epsilon}$ graph over $(\theta^f, p^f)$, which implies
	\begin{equation}
		\label{eq:slow-conv}
		\|p^f(z_1^k) - p^f(z_2^k)\| \ge \sqrt{\epsilon}/D_1 \|z_1^k-  z_2^k\| \ge  D_2^{-1}\sqrt{\epsilon}\rho^k
	\end{equation}
	for some $D_2 >1$. Let $z^k = \phi^{-k}(z)$, we have $\|z^k - z_1^k\| < C \lambda^k$. Suppose $k$ is large enough such that $C\lambda^k < \frac12 D_2^{-1} \sqrt{\epsilon} \rho^k$, then 
	\begin{equation}
		\label{eq:pf-lower}
		\|p^f(z^k) - p^f(z_2^k)\| \ge \|p^f(z_1^k) - p^f(z_2^k)\| - \|p^f(z^k) - p^f(z_1^k)\| \ge \frac12 \|p^f(z_1^k) - p^f(z_2^k)\|.
	\end{equation}

	Assume $k \ge 1/\sqrt{\epsilon}$. We now use Proposition~\ref{prop:strong-lip} to get for some $D_3>1$, 
	\begin{equation}
		\label{eq:graph-1}
		\begin{aligned}
			&   \|p(z^k) - p(z_2^k)\| \le D_3 \sqrt{\epsilon} \left( \|\theta^s(z^k) - \theta^s(z_3^k)\| + \|\theta^f(z^k) - \theta^f(z_3^k)\| \right) \\
			& \le D_3 \sqrt{\epsilon} \left( \|\theta^s(z^k) - \theta^s(z_3^k)\| + \|\theta^f(z^k) - \theta^f(z_1^k)\| \right)  \\
			& \le D_2 \sqrt{\epsilon} \left( \|\theta^s(z^k) - \theta^s(z_2^k)\| \right) + D_2 D_3 \sqrt{\epsilon} \lambda^k,
		\end{aligned}
	\end{equation}
	keep in mind that $\theta^f(z_1^k) = \theta^f(z_2^k)$. Since $z^k_1, z_2^k \in \mC_1$, using Theorem 1.4, we get for small $\epsilon$, 
	\[
		\begin{aligned}
			& 	\|\theta^s(z^k) - \theta^s(z_2^k)\| \le \|\theta^s(z_1^k) - \theta^s(z_2^k)\| + C\lambda^k \\
			& \le \frac{1 + \sqrt{\delta/\epsilon}}{\kappa}\left( \|\theta^f(z_1^k) - \theta^f(z_2^k)\| + \|p^f(z_1^k) - p^f(z_2^k)\| \right) + C\lambda^k \\
			& \le 4 \kappa^{-1}\delta^{\frac12} \epsilon^{-\frac12}\|p^f(z^k) - p^f(z_2^k)\| + C \lambda^k
		\end{aligned}
	\]
	Combine with \eqref{eq:graph-1}, we get 
	\[
		\|p(z^k) - p(z_2^k)\| \le 4C \kappa^{-1}\delta^{\frac12} \|p(z^k) - p(z_2^k)\| + 2 D_2D_3 \sqrt{\epsilon} \lambda^k. 
	\]
	When $\kappa^{-1}\delta^{\frac12} < \frac12$ we get  $\|p(z^k) - p(z_2^k)\| \le 4 D_2 D_3 \sqrt{\epsilon}\lambda^k$, but this contradicts with \eqref{eq:slow-conv} and \eqref{eq:pf-lower}. 
\end{proof}

\begin{lem}\label{lem:local-weak-kam}
For $(N, c_0) \in \mQ$, there is $\sigma_1, \sigma_2, M>0$ such that for all $c \in B_{\sigma_1}(c_0) \cap \Gamma_*(N)$, we have for $i = 1,2$,
\begin{enumerate}
	\item 
	\[
		\overline{\mE_i(N,c)} \cap \pi^{-1}(B_{\sigma_2}(\mS_i(N,c_0))) \subset W_i^u(N,c). 
	\]
	This also implies $\mE_i(N,c) = \overline{\mE_i(N,c)}$ and is $C^1$ over $B_{\sigma_2}(\mS_i(N,c_0))$. 
	\item For each $(\theta, p)\in \overline{\mE_i(N,c)}$, 
 there exists  $k \le M$ such that  
$$\phi^{-k}(\theta, p) \in 
B_{\sigma_2}(\mS_1(N,c)\cup \mS_2(N,c)).
$$ 
\end{enumerate}
\end{lem}
\begin{proof}
We prove  item 1. for $i=1$, the proof  for $i=2$ is identical. We first prove the statement for $c = c_0$ then extend to a neighborhood by continuity. First of all, we refer to \cite{Be1} Lemma 4.4, to get the existence of  $\sigma_3 >0$ such that every $(\theta,p) \in \mE_1(N,c)$ with $\theta \in B_{\sigma_3}(\mS_1(N, c_0))$ is backward asymptotic to $\mS_1$. By Lemma~\ref{lem:back-min},
  there exists $k$ such that $\phi^{-k}(\theta,p) \in W^u_1(N,c)$. 
We now show that $k$ can be chosen uniformly for all $\theta \in \overline{B_{\sigma_3/2}(\mS_1(N,c_0))}$. Arguing by contradiction, if there is $k_i \to \infty$ and  $\phi^{-j}(\theta_i, p_i) \notin W^u_1(N,c)$ for all $0 \le j \le k_i$, after taking a convergent subsequence, we get $(\theta_i, p_i) \to (\theta_*, p_*) \in \overline{\mE_1(N,c)}$ whose backward orbit does not intersect $W^u_1(N,c)$. This is a contradiction. Using a similar compactness argument over $c$, we obtain: 

There exists $\sigma_4, \sigma_5>0$ and $M>0$, such that for all $c \in B_{\sigma_4}(c_0)\cap \Gamma_*(N)$ and  $(\theta, p) \in B_{\sigma_5}(\mS(N, c_0))$, we have $\phi^{-k}(\theta, p) \in W^u_1(N,c)$ for all $k \ge M$. 

Finally, we choose $\sigma_6$ small enough so that $B_{\sigma_6}(\mS_1(N, c_0)) \subset \phi^{-M}(B_{\sigma_5}(\mS_1(n,c_0)))$. Since $\mS_1(N,c)$ is semi-continuous in $c$, this property extends to a small neighborhood of $c \in \Gamma_*(N)$. 

We now prove item 2, for $i=1$.  
 Assume there exists $\sigma_7>0$, $k_i \to \infty$,  $(\theta_i, p_i)\in\overline{ \mE_1(N,c_i)}$ with $c_i \to c_0$,  such that $\phi^{-j}(\theta_i, p_i)\notin B_{\sigma_7}(\mS_1 \cup \mS_2)$ for all $0 \le j \le k_i$. Taking limit up to a subsequence, we obtain an orbit 
$(\theta_*, p_*) \in \overline{\mE_1(N, c_0)}$
 \emph{not} backward asymptotic to $\mS_1 \cup \mS_2$, a contradiction. 
\end{proof}

For each $c\in \Gamma_*(N)$, the set $\tilde \mS_1(N,c)$ is a graph over $\theta^f$, hence there exists a map $\eta_c:\Tm\lto \Tm^n\times \Rm^n$ such that 
$\mS_1(N,C)$ is the image of $\eta_c$ and $\pi_{\theta^f}\circ \eta_c(s)=s$.

\begin{lem}\label{lem:cyl-holder}
 There exists $C_1>0$ such that 
	\[
	\sup_s\|\eta_c(s) - \eta_{c'}(s)\| \le C_1 \|c - c'\|^{\frac12}
	\]
	for each $c$ and $c'$ in $\Gamma_*(N)$.
\end{lem}
\begin{proof}
We denote by $D_i$ different positive constants that may depend on $\epsilon$ and $\delta$.
Since $\mC_1$ is a Lipschitz graph over $(\theta^f,p^f)$, 
	\begin{equation}
		\label{eq:eta-pf}
		\sup_s\|\eta_c(s) - \eta_{c'}(s)\| \le D_1\sup_s\|\pi_{p^f} \eta_c(s) - \pi_{p^f}\eta_{c'}(s)\|. 
	\end{equation} 
Each Weak KAM solution $u_c$ is differentiable on $\mS_1(N,c)$, and we have $\pi_{p}\circ \eta_c=c+du_c(\pi_{\theta}\circ \eta_c)$.
We have
$$
\int _{\eta} pd\theta= \int _{\eta} cd\theta+\int_\eta du_c (\pi_{\theta}\circ \eta_c) d\theta=\pi_{p^f}(c),
$$
hence the symplectic area  $A(\eta_c, \eta_{c'})$
 of the domain of $\mC_1$ delimited by the curves $\eta_c$ and  $\eta_{c'}$ is 
	$$
	A(\eta_c, \eta_{c'}) = \left( \int_{\eta}- \int_{\eta_{c'}} \right) p d \theta=\pi_{p ^f}(c)-\pi_{p ^f}(c').
	$$
	Recall that the cylinder $\cC_1$ is given by a graph $(\theta^s, p^s) = (\Theta^s, P^s)(\theta^f, p^f)$.
	The estimates (\ref{blow-up}) imply that, if $v, v'$ are two vectors tangent to $\mC_1$, then 
	$ |(d\Theta^s\wedge dP^s)(v,v')|\leq C\sqrt \delta |d\theta^f\wedge dp^f(v,v')|$,
	hence, if $\delta$ is small enough,
	$$
	|(d\Theta\wedge dP)(v,v')|\geq \frac{1}{2} |(d\theta^f \wedge dp^f)(v,v')|.
	$$
 Note that given two  $C$ Lipshitz functions $\gamma_1, \gamma_2: \T \to  \R$ with $\gamma_1(s) > \gamma_2(s)$, 
	\[
		\int (\gamma_1 - \gamma_2)ds \ge \frac{1}{4C} \sup\|\gamma_1(s) - \gamma_2(s)\|^2. 
	\]

	Let $\Omega$ denote the region on $\mC_1$ between $\eta_c$ and $\eta_{c'}$. For $c, c' \in \Gamma_*$, there is $D_3, D_4 >1$ such that  
	\begin{equation}
		\label{eq:area-est}
		\begin{aligned}
			& D_3\|c - c'\| \ge \|\pi_{p^f}(c) - \pi_{p^f}(c')\| = |A(\eta_c, \eta_{c'})| \ge \frac12  \left|\int_\Omega d\theta^f \wedge dp^f \right|\\
			&= \frac12 \left|\int (\pi_{p^f} \circ \eta_c (s) - \pi_{p^f}\circ \eta_{c'}(s)) dt \right|\ge \frac{1}{D_4} \sup\|\pi_{p^f}\circ \eta_c(s) - \pi_{p^f} \circ \eta_{c'}(s)\|^2.   
		\end{aligned}
	\end{equation}
	Combine with \eqref{eq:eta-pf} we get our conclusion. 
\end{proof}

\begin{lem}

	In the context of Lemm~\ref{lem:local-weak-kam}, consider for $c, c' \in B_{\sigma_1}(c_0) \cap \Gamma_*(N)$, and $\zeta_1 \in \mS_1(N, c)$ and $\zeta_1' \in \mS_1(N, c')$, denote 
\[
	u_c(\cdot) = h_{\xi^*c}(\zeta_1, \cdot) = h_{N\circ \Xi, \xi^*c}(\zeta_1, \cdot), \quad
	u_{c'}(\cdot) = h_{\xi^*c'}(\zeta_2, \cdot) = h_{N\circ \Xi, \xi^*c'}(\zeta_1', \cdot) .
\]
 Then for $\theta \in B_{\sigma_2}(\mS_1(N, c_0))$ :
	\begin{enumerate}
	 \item 	$|  \nabla u_c(\theta) - \nabla u_{c'}(\theta) | \le C_2 \|c - c'\|^{\frac12}$;
	\item $|u_c(\theta) - u_{c'}(\theta)- C_3| \le C_2 \|c - c'\|^{\frac12}$. 
	\end{enumerate}
Moreover, the same holds with $\mS_1$ replaced with $\mS_2$. 
\end{lem}
\begin{proof}
For $\theta \in B_{\sigma_2}(\mS_1(N, c_0))$, let $y = (\theta, \nabla u_c(\theta))$, and let $z \in \mS_1(N, c)$ be such that $y \in W^s(z)$. We then define $z' \in \mS_1(N, c')$ be the unique such point with $\theta^f(z') = \theta^f(z)$. Finally, define $y' \in W^u(z')$ such that $\theta^s(y') = \theta^s(y)$, which is possible since $W^u(z')$ is locally a graph over $\theta^s$. 

We note that within the center unstable manifold $W^u(\Lambda)$, the NHIC $\Lambda$ on one hand, and $\theta^s = \theta^s(y)$ on the other hand serves as two transversals to the strong unstable foliation $\{W^u(\cdot)\}$. Since the foliation is $C^1$, there exists $D_1>0$ such that 
\[
	\|y - y'\| \le D_1 \|z - z'\| \le C_1 D_1 \|c - c'\|^{\frac12},
\]
where $C_1$ is from Lemma~\ref{lem:cyl-holder}. Denote $w = (\theta, \nabla u_{c'}(\theta))$, and noting $y' \in W^u_1(N,c') = \{(x, \nabla u_{c'}(x))\}$ which is locally a $C^1$ graph, we get for $D_2 > 0$
\[
	\|w - y'\| \le D_2 \|\pi_\theta(w) - \pi_\theta(y')\| = D_2 \|\pi_\theta(y) - \pi_\theta(y')\| \le  D_2\|y - y'\|,
\]
therefore 
\[
	\|\nabla u_c(\theta) - \nabla u_{c'}(\theta)\| \le \|w - y\| \le \|w -y'\| + \|y - y'\| \le D_3 \|y - y'\| \le D_4 \|c - c'\|^{\frac12}. 
\]
Item 1 follows. For item 2, we consider $\theta, \theta_0 \in B_{\sigma_2}(\mS_1(N, c_0))$, then integrating item 1 leads to 
\begin{equation}
  \label{eq:u-comp}
	|u_c(\theta) - u_{c'}(\theta) - \left( u_c(\theta_0) - u_{c'}(\theta_0) \right)| \le D_5 \|c - c'\|^{\frac12}.   
\end{equation}
Item 2 follows by taking $C_3 = u_c(\theta_0) - u_{c'}(\theta_0)$. 
\end{proof}

\begin{proof}[Proof of Proposition~\ref{prop:hd}]
	Fix $(N,c_0) \in \mQ$, we consider $c \in B_{\sigma_2}(c_0) \cap \Gamma_*(N)$ in the context of Lemma~\ref{lem:local-weak-kam}. From item 2 of that lemma,  for every $\theta \in \T^n$, there exists a calibrated orbit $\gamma: (-\infty, 0] \to \T^n$ with $\gamma(0) = \theta$,  such that $\gamma(t) \in B_{\sigma_2}(\mS_1(N,c) \cup \mS_2(N, c))$ whenever $t < -M$. Then (omitting the subscript $N\circ \Xi$)
	\[
		h_{\xi^*c}(\zeta_1, \theta) = \min_{i = 1,2} \min_{k \le M}
		\min_{\theta' \in B_{\sigma_2}}
		\left\{h_{\xi^*c}(\zeta_1, \zeta_i) +  h_{\xi^*c}(\zeta_i, \theta') + A^k_{\xi^*c}(\theta',\theta) \right\}.
	\]
	Since $h_{\xi^*c}(\zeta_i, \theta')$ are uniformly $\frac12$ Holder in $c$ for  $\theta' \in B_{\sigma_2}(\mS_i(N,c))$ and $c \in B_{\sigma_2}(c_0) \cap \Gamma_*(N)$, each $A^k_{\xi^*c}$ are uniformly Lipshitz in $c$, the family $h_{\xi^*c}(\zeta_1, \theta)$ is $\frac12$ Holder in $c$. 
\end{proof}

\appendix

\section{Normally hyperbolic manifold}
\label{sec:abstract-nhic}

Let $F:\Rm^n\lto \Rm^n$
be a $C^1$ vector field. We give sufficient conditions
for the existence of a Normally hyperbolic invariant graph of $F$.
We split the space $\Rm^n$ as $\Rm^{n_u}\times\Rm^{n_s}\times \Rm^{n_c}$,
and denote by $x=(u,s,c)$ the points of $\Rm^n$. We denote by
$(F_u,F_s,F_c)$ the components of $F$:
$$
F(x)=(F_u(x),F_s(x),F_c(x)).
$$
We study the flow of $F$ in the domain
$$\Omega=B^u\times B^s\times \Omega^c
$$
where $B^u$ and $B^s$ are  the open Euclidean balls of radius $r_u$  
and $r_s$ in $\Rm^{n_u}$ and $\Rm^{n_s}$, and $\Omega^c$ is a convex 
open subset of $\Rm^{n_c}$. We denote by
$$
L(x)=dF(x)=\begin{bmatrix}
L_{uu}(x)&L_{us}(x)&L_{uc}(x)\\L_{su}(x)&L_{ss}(x)&L_{sc}(x)\\L_{cu}(x) &L_{cs}(x)&
L_{cc}(x)
\end{bmatrix}
$$
the linearized vector field at point $x$. We  assume that $\|L(x)\|$ is bounded
on $\Omega$, which implies that each trajectory of $F$ is defined until it leaves $\Omega$.
We denote by $W^c$ the union of full orbits contained in $\Omega$.
In other words, this is the set of initial conditions $\underline x\in \Omega$
such that there exists a solution $x(t):\Rm\lto \Omega$
of the equation $\dot x =F(x)$ satisfying $x(0)=\underline x$.
We denote by $W^{sc}$ the set of points whose positive  orbit remains inside
$\Omega$.
In other words, this is the set of initial conditions $\underline x\in \Omega$
such that there exists a solution $x(t):[0,\infty)\lto \Omega$
of the equation $\dot x =F(x)$ satisfying $x(0)=\underline x$.
Finally, we denote by $W^{uc}$
the set of points whose negative  orbit remains inside
$\Omega$.
In other words, this is the set of initial conditions $\underline x\in \Omega$
such that there exists a solution $x(t):(\infty,0]\lto \Omega$
of the equation $\dot x =F(x)$ satisfying $x(0)=\underline x$.
These sets have specific features under the following assumptions:

\begin{hyp}[Isolating block]\label{block}
We have:
\begin{itemize}
\item
$F_c=0$ on $B^u\times B^s\times \partial\Omega^c$.
\item $F_u(u,s,c)\cdot u> 0$ on  $\partial B^u \times \bar B^s \times \bar \Omega^c$.
\item $F_s(u,s,c)\cdot s< 0$ on  $ \bar B^u \times \partial B^s \times \bar \Omega^c$.
\end{itemize}
\end{hyp}
\begin{hyp}\label{cone}
There exist  positive constants $\alpha$ and  $m$  such that:
\begin{itemize}
\item
$L_{uu}(x)\geq \alpha I , \quad L_{ss}(x)\leq -\alpha I
$ for each $x\in \Omega$
in the sense of quadratic forms.
\item$
\|L_{us}(x)\|+\|L_{uc}(x)\|+\|L_{su}(x)\|+\|L_{sc}(x)\|+
\|L_{cu}(x)\|+\|L_{cs}(x)\|+\|L_{cc}(x)\|\leq m
$ for each $x\in \Omega$.
\end{itemize}
\end{hyp}

\begin{thm}\label{abstractNHI}
Assume that  Hypotheses \ref{block} and \ref{cone} hold, and that
$$
0\leq K:= \frac{m}{\alpha-2m}\leq \frac{1}{\sqrt{2}}.
$$
Then the set $W^{sc}$ is  the graph of a $C^1$ function
$$w^{sc}:B^s \times \Omega^c \lto B^u,
$$
 the set $W^{uc}$ is  the graph of a $C^1$ function
$$w^{uc}:B^u \times \Omega^c \lto B^s,
$$
and  the set $W^c$ is the graph of a $C^1$ function
$$w^c=(w^c_u,w^c_s):\Omega^c\lto B^u\times B^s.
$$
Moreover, we have the estimates
$$\|dw^{sc}\|\leq K,\quad \|dw^{uc}\|\leq K,
\quad \|dw^c\|\leq 2K.
$$
\end{thm}
\proof
This results could be reduced to several already existing ones,
see \cite{Fe, HPS,McG,Ch}
or proved directly by well-known methods. We shall use Theorem 1.1 in \cite{Ya}
which is the closest to our needs because it is expressed  in terms
of vector fields.
We first derive some conclusions from the isolating block conditions.
We denote by $\pi^{sc}$ the projection $(u,s,c)\lmto (s,c)$, and so on.

\begin{lem}\label{h3}
If Hypothesis \ref{block} holds, then
$$
\pi^{sc}(W^{sc})=B^s\times \Omega^c.
\quad
\text{and} \quad
\pi^{uc}(W^{uc})=B^u\times \Omega^c 
$$
Moreover, the closures of $W^{sc}$ and $W^{uc}$
satisfy
$$
\bar W^{sc}\subset B^u\times \bar B^s\times \bar \Omega^c,
\quad
\bar W^{uc}\subset \bar B^u\times  B^s\times \bar \Omega^c.
$$
\end{lem}
\proof
Let us define $T^+(x)\in [0, \infty]$ as the first positive time
where the orbit of $x$ hits the boundary $\partial \Omega$.
Let us denote by $\varphi(t,x)$ the flow of $F$.
If $T^+(x)<\infty$ (which is equivalent to $x\not \in W^{sc}$),  we have
$\varphi(T^+(x),x)\in \partial B^u\times B^s\times \Omega^c$,
as follows from Hypothesis \ref{block}.
Then, it is easy to check that the function $T^+$ is continuous, and even $C^1$,
at $x$.

We  prove the first equality of the Lemma  by contradiction,
and assume that there exists a point $(s,c)\in B^s\times \Omega^c$
such that $W^{sc}$ does not intersect the disc $B^u\times \{s\} \times \{c\}$.
Then, the first exit map
$$B^u\ni u\lmto \pi^u\circ \varphi(T^+(u,s,c),(u,s,c))\in \partial B^u,
$$
extends by continuity to a continuous retraction
from $\bar B^u$ to its boundary $\partial B^u$.
 Such a retraction does not exist.
The proof of the other equality is similar.

Finally, we have
$$
\bar W^{sc} \subset \bar B^u\times \bar B^s \times \bar \Omega ^c
=\big( B^u\times \bar B^s \times \bar \Omega ^c\big)\bigcup
\big( \partial B^u\times \bar B^s \times \bar \Omega ^c\big).
$$
Hypothesis \ref{block} implies that each point of
$\partial B^u\times \bar B^s \times \bar \Omega ^c$ has a neighborhood
formed of points which leave $\Omega$ after a small time.
As a consequence, the set $\partial B^u\times \bar B^s \times \bar \Omega ^c$
can't intersect $\bar W^{uc}$, and we have proved that
$
\bar W^{sc} \subset B^u \times \bar B^s \times \bar \Omega ^c.
$
The other inclusion can be proved in a similar way.
\qed

In order to prove the statement of the Theorem  concerning $W^{sc}$, we apply Theorem 1.1 of \cite{Ya}.
More precisely, using the notation of that paper, we set
$$
a=u/K, \quad
z=(s,c),\quad f(a,z)=F_u(Ka,z)/K,\quad g(a,z)=(F_s(Ka,z),F_c(Ka,z)).
$$
We have the estimates
$$
\partial_af=L_{uu}\geq \alpha, \quad
\partial_z g=
\begin{bmatrix}
L_{ss}&L_{sc}\\L_{cs}&L_{cc}
\end{bmatrix}\leq m
$$
in the sense of quadratic forms.
Moreover, we have the estimates
$$
\|\partial_z f\|\leq \frac{m}{K},\quad \|\partial_ag\|\leq Km.
$$
Since
$$
m+m/K+Km<2m+m/K=\alpha
$$
we conclude that Hypothesis 2 of \cite{Ya} is satisfied.
Hypothesis 1 of \cite{Ya} is verified by the domain $\Omega$,
and Hypothesis 3 is precisely the conclusion of Lemma \ref{h3}.
As a consequence, we can apply Theorem 1.1 of \cite{Ya},
and conclude that the set $W^{sc}$ is the graph of a $C^1$ and $1$-Lipschitz
map above $B^s\times \Omega^c$
in $(a,z)$ coordinates, and therefore the graph of a $K$-Lipschitz
$C^1$ map $w^{sc}:B^s\times \Omega^c\lto B^u$ in  $(u,s,c)$ coordinates.

In order to prove the statement concerning $W^{uc}$,
we apply Theorem 1.1 of \cite{Ya} with
$$
a=s/K, \quad
z=(u,c),\qquad \qquad $$
$$
f(a,z)=-F_s(Ka,z)/K,\quad g(a,z)=-(F_u(Ka,z),F_c(Ka,z)).
$$
It is easy to check as above that all hypotheses are satisfied.

Let us now study the set $W^c=W^{sc}\cap W^{uc}$.
First, let us prove that $W^c$ is a $C^1$ graph above $\Omega^c$.
We know that $W^{sc}$ is the graph of a $K$-Lipshitz $C^1$ function
$w^{sc}(s,c)$ and that  $W^{uc}$
is the graph of a $K$-Lipshitz $C^1$ function
$w^{uc}(u,c)$. The point $(u,s,c)$ belongs to $W^c$
if and only if
$$
u=w^{sc}(s,c) \quad \text{and}\quad s=w^{uc}(u,c),
$$
or in other words if and only if $(u,s)$ is a fixed point
of the $K$-Lipschitz $C^1$ map
$$
(u,s)\lmto (w^{sc}(s,c),w^{uc}(u,c)).
$$
For each $c$, this contracting  map has a unique  fixed point in
$\bar B^u\times \bar B^s$, which corresponds to a point
of $\bar W^{sc}\cap \bar W^{uc}$. It follows from  Lemma \ref{h3}
that this point is contained in $B^u\times B^s$.
Then,  it depends in a $C^1$ way of the
parameter $c$. We have proved that $W^c$ is the graph
of a $C^1$ function $w^c$.
In order to estimate the Lipschitz constant of this graph,
we consider two points $(u_i,s_i,c_i), i=0,1$ in $W^c$.
We have
$$
\|u_1-u_0\|^2\leq K^2 (\|s_1-s_0\|^2+\|c_1-c_0\|^2)
$$
and
$$
\|s_1-s_0\|^2\leq K^2 (\|u_1-u_0\|^2+\|c_1-c_0\|^2).
$$
Taking the sum gives
$$
(1-K^2)(\|u_1-u_0\|^2+\|s_1-s_0\|^2)\leq 2 K^2\|c_1-c_0\|^2
$$
and
$$
\|(u_1,s_1)-(u_0,s_0)\|\leq \sqrt{\frac{2K^2}{1-K^2}}\|c_1-c_0\|
\leq 2K\|c_1-c_0\|,
$$
since $K\leq 1/\sqrt{2}$.
We conclude that $w^c$ is $2K$-Lipschitz.
\qed

It is useful to go a bit further in the study of the invariant manifold $W^c = \{(w_u^c(c), w^s_c(c), c)\}$.
This manifold is a partially hyperbolic invariant set, hence by the usual theory, to each point 
$x\in W^c$ is attached a strong stable manifold $W^s(x)$ and a strong unstable manifold $W^u(x)$,
which are $C^1$ (and even $C^r$ if $F$ is $C^r$).
The manifolds $W^u(x), x\in W^c$ partition $W^{uc}$, although this partition is not usually a $C^1$ foliation.
For each $x\in W^{uc}$, we denote by $E^u(x)$ the strong unstable space, which is
the tangent space at $x$ of the only unstable manifold $W^u(x_0)$ which contains $x_0$.
We define the exponents
\begin{align*}
e_u&:= - \sup_{x\in W^c, v\in E^u(x)} \limsup _{t\lto \infty} \log (\|v(-t)\|)/t, \\
&=-\sup_{x\in W^{uc}, v\in E^u(x)} \limsup _{t\lto \infty} \log (\|v(-t)\|)/t, \\
e_c^+&:= \sup_{x\in W_c, v\in T_xW_c} \limsup _{t\lto \infty} \log (\|v(t)\|)/t, \\
e_c^-&:= \inf_{x\in W_c, v\in T_xW_c} \liminf _{t\lto \infty} \log (\|v(t)\|)/t,
\end{align*}
where $v(t)$ is the solutions of the linearized equation $\dot v(t)=dF_{x(t)}\cdot v(t)$ with initial condition $v(0)=v$,
and $x(t)$ is the solution of $\dot x(t)=F\circ x(t)$ starting from $x(0)=x$.

\begin{lem}
$$
-m-2mK\leq e_c^-\leq e_c^+\leq m+2mK.
$$
\end{lem}

\proof
We consider an orbit $x(t)\in W^c$, and a variational orbit $v(t)=(u'(t),s'(t),c'(t))$ tangent to $W^c$.
Observe that  $\|(u',s')\| \le 2K \|c'\|$ for each $t$, which implies:
\[
\left|	\frac{d}{dt}\|c'\|^2 \right|=2|\langle c', L_{cu} u' + L_{cs} s' + L_{cc} c'\rangle| \le 2(m + 2K m)\|c'\|^2.
\]
\qed

The next Lemma implies that the manifolds $W^{s,c}(x)$ are the graphs of $C^1$ and $K$-Lipschitz maps $w^s_x:B^s\lto B^u\times \Omega^c , 
w^u_x:B^u\lto B^s\times \Omega^c
$.

\begin{lem}
If $x(t):]T^-,T^+[\lto \Omega$ is an orbit of $F$, then the linearized equation 
$\dot v(t)=dF_{x(t)}\cdot v(t)$ preserves the cone
 $C^u = \{\|(s', c')\| \le K \|u'\|\}$ in forwad time, and  the cone $C^s = \{\|(u', c')\| \le K \|s'\|\}$ in backward time.
 
 We have $E^u(x)\subset C^u$ for each $x\in W^{uc}$, $E^s(x)\subset C^s$ for each $x\in W^{sc}$.

Finally we have the estimate
$$
e_u\geq \alpha-2mK>\alpha/2.
$$
\end{lem}

\begin{proof}
Let $v(t)=(u'(t),s'(t),c'(t))$ be a solution of the linearized equation along $x(t)$.
Then
$$
 \frac{d}{dt} \|u'\|^2=\langle
 u', L_{uu} u'+L_{us}s'+L_{uc}c'  
 \rangle\geq \alpha \|u'\|^2- m\|(s',c')\|\|u'\|\geq (\alpha-mK)\|u'\|^2
 $$
 (this estimate will also provide the desired growth rate in the unstable direction) and 
\begin{align*}
 \frac{d}{dt} \|(s',c')\|^2 &=
\langle
s', L_{su}u'+L_{ss}s'+ L_{sc} c'\rangle + 
\langle
c', L_{cu}u'+L_{cs}s'+ L_{cc} c'
\rangle\\
&
\leq m \|(s',c')\|(\|u'\|+\|(s',c')\|)\leq mK(1+K)\|u'\|^2.
\end{align*}
This implies implies 
that 
$$
\frac{d}{dt} \big( K^2 \|u'\|^2-\|(s',c')\|^2\big)\geq K^2(\alpha-mK-m-m/K)\|u'\|^2\geq 0,
$$
recalling that $m+m/K+mK<\alpha$.
The estimates concerning $C^s$ are similar.
\end{proof}

In general, the maps $w^s_x$ and $w^u_x$ are not better than (H\"older)-continuous in $x$, but we can obtain a better regularity under
stronger hypotheses:

\begin{thm}
In the context of Theorem \ref{abstractNHI}, let us assume  the additional assumptions that 
 $F$ is $C^2$ and $K<1/8$ (or equivalently, $m<\alpha/6$).
 Then each of the manifolds $W^c, W^{uc}, W^{sc}$ is $C^2$,
 and the manifolds $W^u(x), x\in W^c$ form a $C^1$ foliation of $W^{uc}$ (similarly for $W^s$ in $W^{us}$).
 The foliations are $C^1$ in the strongest possible sense, namely the map $x\lmto E^u(x)$ is $C^1$ on $E^{cu}$,
 which imply that the foliation admits $C^1$ charts, and that the local holonomies ar $C^1$.
\end{thm}

\proof
An easy computations shows that $m+2mK<\alpha/4$, hence we obtain
$$
e_u>\alpha/2, \quad e_c^+< \alpha/4, \quad e_c^->-\alpha/4.
$$
This implies that $e_u>2e_c^+$, hence $W^c$ is $2$-normally hyperbolic,
hence it is $C^2$, as well as $W^{uc}$ and $W^{sc}$, see \cite{Fe, HPS}.

Moreover, we have the bunching condition $e_u>e_c^+-e_c^-$, which implies the $C^1$ regularity of the unstable foliation,
see \cite{Fe2, PSW, DLS}.
\qed

We need the following easy  addendum:
\begin{prop}\label{translation}
Assume  in addition that there exists a translation  $g$
of $\Rm^{n_c}$
such that
$$
g(\Omega^c)=\Omega^c\quad \text{and} \quad
F\circ (id\otimes id\otimes g)=F.
$$
 Then we have
$$
w^{sc}\circ(id\otimes g)=w^{sc},\quad
w^{uc}\circ(id\otimes g)=w^{uc}, \quad
w^c\circ g=w^c.
$$
\end{prop}
\proof
It follows immediately from the definition of the sets
$W^{sc}$, $W^{uc}$ and $W^c$ that $g(W^{sc})=W^{sc}$,
$g(W^{uc})=W^{uc}$ and $g(W^c)=W^c$.
\qed

In applications the first condition of Hypothesis \ref{block}
is usually not satisfied, except in the case where
$\Omega^c=\Rm^{n_c}$.
In view of the applications we have in mind,
it is useful to split the central variables into two groups
and consider 
$$
\Omega^c=\Rm^{n^1_c}\times \Omega^{c_2},
$$
where $\Omega^{c_2}$ is a convex open set in $\Rm^{n^2_c}$,
$n^1_c+n^2_c=n_c$.
Given a positive parameter ${\sigma}$, let $\Omega^{c_2}_{\sigma}$ be the set of points
$c_2\in \Rm^{n^2_c}$ such that $d(c,\Omega^{c_2})<{\sigma}$.
This is a convex open subset of $\Rm^{n^2_c}$ containing $\Omega^{c_2}$.
We denote by $\Omega^c_{\sigma}$ the product $\Rm^{n^1_c}\times \Omega^{c_2}_{\sigma}$
and by
 $\Omega_{\sigma}$ the product $B^u\times B^s\times \Omega^c_{\sigma}$.
With the notation $F_c=(F_{c_1},F_{c_2})$,
and  denoting by $W^{sc}(F,\Omega), W^{uc}(F,\Omega), W^c(F,\Omega)$ the set of positive half orbits
(resp. negative half orbits, full orbits) of $F$ contained in $\Omega$, we have:

\begin{prop}\label{realNHI}
Let 
$
F: \Rm^{n_u}\times \Rm^{n_s}\times \Omega^c_{\sigma}
\lto
\Rm^{n_u}\times \Rm^{n_s}\times \Rm^{n_c}
$
be a $C^2$ vector field.
Assume that there exists $\lambda,m,\sigma>0$ such that
\begin{itemize}
\item $F_u(u,s,c)\cdot u> 0$ on  $\partial B^u \times \bar B^s \times \bar \Omega^c_{\sigma}$.
\item $F_s(u,s,c)\cdot s< 0$ on  $ \bar B^u \times \partial B^s \times \bar \Omega^c_{\sigma}$.
\item
$L_{uu}(x)\geq \alpha I , \quad L_{ss}(x)\leq -\alpha I$ for each $x\in \Omega_{\sigma}$
in the sense of quadratic forms.
\item
$
\|L_{us}(x)\|+\|L_{uc}(x)\|+\|L_{ss}(x)\|+\|L_{sc}(x)\|+
\|L_{cu}(x)\|+\|L_{cs}(x)\|+\|L_{cc}(x)\| \leq m
$ for each $x\in \Omega_{\sigma}$.
\item
$
\|L_{us}(x)\|+\|L_{uc}(x)\|+\|L_{ss}(x)\|+\|L_{sc}(x)\|+
\|L_{cu}(x)\|+\|L_{cs}(x)\|+\|L_{cc}(x)\|+
2\|F_{c_2}(x)\|/{\sigma}\leq m
$ for each $x\in \Omega_{\sigma}-\Omega$.
\end{itemize}
Assume furthermore that
$$
K:= \frac{m}{\alpha-2m}\leq \frac{1}{8},
$$
then there exist $C^2$ maps 
$$ w^{sc}:B^s \times \Omega^c_{\sigma} \lto B^u,
\quad  w^{uc}:B^u \times \Omega^c_{\sigma} \lto B^s,
\quad
 w^c:\Omega^c_{\sigma}\lto B^u\times B^s
$$
satisfying the estimates
$$\|d w^{sc}\|\leq K,\quad \|d w^{uc}\|\leq K,
\quad \|d w^c\|\leq 2K,
$$
the graphs of which 
respectively contain $W^{sc}(F,\Omega), W^{uc}(F,\Omega), W^{c}(F,\Omega)$.
Moreover, the graphs of the restrictions of $w^{sc},w^{uc}$ and $w^c$ to, respectively,
$B^s\times \Omega^c$, $B^u\times \Omega^c$ and $\Omega^c$, are tangent to the flow.

There exists an invariant  $C^1$ foliation of the graph of $w^{uc}$ whose leaves are graphs of $K$-Lipschitz maps above $B^u$.
The set $W^{uc}(F,\Omega)$ is a union of leaves : it has the structure of an invariant $C^1$ lamination.
Two points $x,x'$ belong to the same leaf of this lamination if and only if $d(x(t),x'(t))e^{t\alpha/4}$ is bounded on $\Rm^-$.

If  in addition  there exists a group $G$ of translations  
of $\Rm^{n_{c_1}}$
such that
$
F\circ (id\otimes id\otimes g\otimes id)=F
$
for each $g\in G$, 
 then the maps $w^*$ can be chosen such that
\begin{equation}\label{eq-w}
w^{sc}\circ(id\otimes g\otimes id)=w^{sc},\quad
w^{uc}\circ(id\otimes g\otimes id)=w^{uc}, \quad
w^c\circ (g\otimes id)=w^c
\end{equation}
for each $g\in G$. The lamination is also translation invariant.
\end{prop}

In contrast to the earlier results of this section, the map $w^{sc}$
is not uniquely defined, and neither is its restriction to $B^s\times \Omega^c$. 
Moreover, the intersection with $\Omega$ of the graph of $w^{sc}$
is not necessarily positively invariant. It can contain strictly the set 
$W^{sc}(F,\Omega)$. Similar remarks apply to $w^{uc}$ and $w^c$.

\proof
We take a function $\rho:\Omega^{c_2}_{\sigma}\lto [0,1]$ such that :
\begin{itemize}
\item $\rho=0$ near the boundary of $\Omega^{c_2}_{\sigma}$,
\item $\rho=1$ on $\Omega^{c_2}$,
\item $\|d\rho\|\leq 2/{\sigma}$ uniformly.
\end{itemize}
We claim that the vectorfield 
$$
\tilde F(u,s,c):=(F_u(u,s,c_1,c_2),F_s(u,s,c_1,c_2),
F_{c_1}(u,s,c_1,c_2),
\rho(c_2)F_{c_2}(u,s,c_1,c_2))
$$
satisfies all the hypotheses of Theorem \ref{abstractNHI} on $\Omega_{\sigma}$.
Note also that $\tilde F=F$ on $\Omega$.
Denoting by $\tilde L_{**}$ the variational matrix associated to $\tilde F$,
we see that
$$
\tilde L_{cu}(u,s,c)=\rho(c_2) L_{cu}(u,s,c),\quad
 \tilde L_{cs}(u,s,c)=\rho(c_2) L_{cs}(u,s,c),
$$
$$
\tilde L_{c_1c_1}(u,s,c)=\rho(c_2) L_{c_1c_1}(u,s,c),\quad
 \tilde L_{c_1c_2}(u,s,c)=\rho(c_2) L_{c_1c_2}(u,s,c),
$$
and
$$
\tilde L_{c_2c_2}(u,s,c)=\rho(c_2) L_{c_2c_2}(u,s,c)+d\rho(c_2)\otimes F_{c_2}(u,s,c).
$$
As a consequence, we have
\begin{align*}
&\|\tilde L_{us}(x)\|+\|\tilde L_{uc}(x)\|+\|\tilde L_{ss}(x)\|+\|\tilde L_{sc}(x)\|+
\|\tilde L_{cu}(x)\|+\|\tilde L_{cs}(x)\|+\|\tilde L_{cc}(x)\|
\\
=&\rho(c_2) \big(\|L_{us}(x)\|+\|L_{uc}(x)\|+\|L_{ss}(x)\|+\|L_{sc}(x)\|+
\|L_{cu}(x)\|+\|L_{cs}(x)\|+\|L_{cc}(x)\|\big)\\
+&\|F_{c_2}(x)\|\|d\rho(c_2)\|
\leq m.
\end{align*}
The claim is proved.
We define $w^{sc},w^{uc}, w^c$ as the maps given by Theorem \ref{abstractNHI}
applied to $\tilde F$ on $\Omega_{\sigma}$.
Since $\tilde F=F$ on $\Omega$, we have $W^*(F,\Omega)\subset W^*(\tilde F, \Omega_{\sigma})$ for
$*=sc, uc$ or $c$.
These maps may depend on the choice of the function $\rho$ but, once the function $\rho$ is chosen, they are uniquely defined.
In the case where a group of translation $G$ exists as in the statement,
then we have 
$
\tilde F\circ (id\otimes id\otimes g\otimes id)=\tilde F
$
for each $g\in G$. The uniqueness then implies
(\ref{eq-w}).
By definition, $W^*(\tilde F, \Omega_{\sigma})$ is the graph of $w^*$, the statement follows from this observation.
\qed

\section{Disconnectedness of Heteroclinics}
\label{sec:disc}

We consider a Tonelli Hamiltonian $H$,
a cohomology $c$, and the associated Aubry and  Mañé sets
$\tilde \mA$ and $\tilde \mN$.
We assume that the Aubry set is the union of two static 
classes $\tilde \mS_i, i=1,2$.
The Mañé set can then be written as the disjoint union
$$
\tilde \mN=\tilde \mS_1\cup \tilde \mS_2 \cup \tilde \mH_{12}\cup \tilde \mH_{21},
$$
where $\tilde \mH_{12}$ is a set of heteroclinic orbits 
from $\tilde \mS_1$ to $\tilde \mS_2$, and 
$\tilde \mH_{21}$ is a set of heteroclinic orbits 
from $\tilde \mS_2$ to $\tilde \mS_1$.
Morever, the sets 
$$
\tilde \mI_{12}:= \tilde \mS_1\cup \tilde \mS_2 \cup \tilde \mH_{12}
\quad
\text{and}
\quad 
\tilde \mI_{21}:= \tilde \mS_1\cup \tilde \mS_2 \cup \tilde \mH_{21}
$$
are invariant compact Lipschitz graphs. 
In the notations of \cite{Be1}, we have 
$\tilde \mI_{12}=\tilde \mI(E_{\mS_1})=E_{\mS_1}\wedge \breve E_{\mS_2}$,
$\tilde \mI_{21}=\tilde \mI(E_{\mS_2})=E_{\mS_2}\wedge \breve E_{\mS_1}$.

In \cite{Be1}, Section 9, it is proved that the cohomology  $c$ is in the interior 
of its forcing class provided each of the sets $\tilde \mH_{12}$
and $\tilde \mH_{21}$ is neat in the following sense:

The set $\tilde \mH_{12}$ is neat if there exists a compact subset 
$\tilde \mK_{12}$ which contains one and only one point in each orbits of 
$\varphi_{|\tilde \mH_{12}}$ and which is acyclic, which means that there exist
an open  neighborhood $U$ of $\mK_{12}$ in $TM$ such that the inclusion 
of $U$ into $TM$ generates the null map in homology.

In Section \ref{sec:proof} of the present paper,
we apply this result under the assumption that the sets 
$\tilde \mH_{12}$ and $\tilde \mH_{21}$ are totally disconnected.
We can do so in view of the following:

\begin{prop}
The set $\tilde \mH_{12}$ (or $\tilde \mH_{21}$ ) is neat if it is totally disconnected. 
\end{prop}

\proof

We first recall that a compact metric space is totally disconnected if  and only
if it has dimension zero, which means that each of its points has a basis of neighborhood 
made of open and closed sets, see \cite{HW}, section II.4.

By removing small open neighborhoods of $\tilde \mS_1$ and $\tilde \mS_2$ in $\tilde \mI_{12}$,
 we form
a compact subset of $\tilde \mH_{12}$ which contains at least one point in each orbit.
This compact subset is totally disconnected (it is a subset of $\tilde \mH_{12}$)
hence each of its points is contained in an open and closed set which is disjoint from
both $\tilde \mS_1$ and $\tilde \mS_2$.
We cover our compact by finitely many of these neighborhood.
Their union  is
a compact and open subset $\tilde \mQ$ of $\tilde \mH_{12}$ which contains at least one point in each orbit.
The set $\tilde \mK_{12}:= \tilde \mQ-\varphi(\tilde \mQ)$ is then compact and open, and it contains exactly one
point of each $\varphi$-orbit.
It is totally disconnected, and therefore acyclic, in view of the following Lemma.
\qed

\begin{lem}
Let $M$ be a manifold and let $K\subset M$ be a totally disconnected 
compact subset of $M$. Then $K$ is acyclic.
\end{lem}

\proof
The subset $K$ has dimension $0$, see \cite{HW}.
As a consequence, each point of $K$ is contained in an open, closed, and acyclic neighborhood
(small open sets are contained in discs hence are acyclic).
We cover $K$ by finitely many of these subsets $U_1,\ldots,U_k$ and set  $V_1=U_1$, $V_2=U_2-V_1$, $V_i=U_i-V_{i-1}$. We obtain $k$ open 
acyclic subsets $V_i$ which are pairwise disjoint and cover $K$.
This implies that $K$ is acyclic.
\qed

\section{Continuity property of the Peierls' barrier function}

We consider here a general Tonelli Lagrangian $L$.
We recall, see \cite{Be2}, section 4, that the difference of two weak KAM solutions is constant on each static class.

\begin{prop}\label{prop:cont}
Let $L_k \to L$ be a sequence of Tonelli Lagrangians $\T^n \times \R^n \times \T$ converging in the $C^2$ compact open topology, and $c_k \to c \in \R^n \simeq H^1(\T^n, \R)$. Assume that $\mA_L(c)$ has finitely many static classes. Let $\zeta_k \in \mA_{L_k}(c_k)$ be such that $\zeta_k \to \zeta_0 \in \mA_L(c)$, then for any $\theta \in \T^n$, 
\[
	\lim_{k\to \infty}h_{L_k, c_k}(\zeta_k, \theta) = h_c(\zeta_0, \theta). 
\]
\end{prop}
\begin{proof}
	First, since each $A^M_{L,c}(\theta_1, \theta_2)$ is continuous in $L$ and $c$, we obtain 
	\[
		\lim_{k \to \infty}\Phi_{L_k,c_k}(\theta_1, \theta_2) \le \lim_{k \to \infty} \left( A^M_{L_k, c_k}(\theta_1, \theta_2)\right)  = A^M_{L,c}(\theta_1, \theta_2)
	\]
	taking infimum over $N$, we get $\lim_{k \to \infty}\Phi_{L_k,c_k}(\theta_1, \theta_2) \le \Phi_{L,c}(\theta_1, \theta_2)$. Since $h_{L,c}(\theta_1, \theta_2) = \Phi_{L,c}(\theta_1, \theta_2)$ if either $\theta_1$ or $\theta_2$ is in $\mA_L(c)$, we obtain 
	\[
		\lim_{k\to \infty}h_{c_k}(\zeta_k, \theta) \le h_c(\zeta_0, \theta). 
	\]

	Given $\epsilon_k \to 0$, let $\gamma_k:[-Q_k, 0] \to \infty$ be a sequence of extremal curves such that $\gamma_k(-Q_k) = \zeta_k$, $\gamma_k(0) = \theta$, and 
	\[
		A^{Q_k}_{L_k,c_k}(\zeta_k, \theta)\le h_{L_k, c_k}(\zeta_k, \theta) + \epsilon_k.
	\]
	We note that on each interval $[i, j] \subset [-Q_k, 0]$, we have 
	\begin{equation}
		\label{eq:eps-calibrate}
		\begin{aligned}
			&  A^{j-i}_{L_k, c_k}(\gamma_k(i), \gamma_k(j))  \\
			& = A^{Q_k}(\gamma_k(-Q_k), \gamma_k(0)) - A^{i+Q_k}(\gamma_k(-Q_k), \gamma_k(i)) - A^{- j}(\gamma_k(j), \gamma_k(0)) \\
			& \le h(\zeta_k, \theta) - \epsilon_k - \left( h(\zeta_k, \gamma_k(i)) - h(\zeta_k,\gamma_k(-Q_k)) \right) -  \left( h(\zeta_k, \gamma_k(0)) - h(\zeta_k, \gamma_k(j)) \right) \\
			& \le h_{L_k, c_k}(\zeta_k, \gamma_k(j)) - h_{L_k, c_k}(\zeta_k, \gamma_k(i)) + \epsilon_k,
		\end{aligned}
	\end{equation}
	since $h(\zeta_k, \gamma_k(-Q_k)) = h(\zeta_k, \zeta_k)=0$ and $\gamma_k(0) = \theta$. Note we omit the subscript $L_k, c_k$ in the intermediate calculations.

	Let $i_k, i_k'$ be two consecutive visit of $\gamma_k(i)$ to $U = B_\delta(\mA_L(c))$, we first show that $i_k' - i_k$ must be bounded as $k \to \infty$.  Assume otherwise, then the curves $\gamma_k(t+i_k+1)|[0, i_k' - i_k-2]$ converges in uniformly over compact sets to $\gamma_*:[0, \infty) \to \T^n$. Assume the weak KAM solutions $h_{L_k, c_k}(\zeta_k, \cdot)$ converges uniformly to a weak KAM solution $u$ of $L,c$, taking limit in  \eqref{eq:eps-calibrate} implies $\gamma_*$ must be calibrated by $u$. Therefore $\gamma_*$ must accumulates to $\mA_L(c)$ which is a contradiction. 

	Let $\mS_1, \cdots, \mS_r$ be the static classes of $\mA(L)$.  Denote $U_q = B_\delta(\mS_q)$ and assume $\delta$ is small enough so that $U_q$ are all disjoint. Let us note each $\gamma_k$ determines sequences $q_s \in \{1, \cdots, r\}$, $s = 1, \cdots, r$, and $0 = i_0 \le j_0 \le \cdots \le i_r \le j_r \le Q_k$ as follows.
	\begin{itemize}
		\item Set $i_0 = j_0 = 0$. 
		\item Let $i_1$ be the first visit of $\gamma(-i)$ to $\bigcup_q U_q$ and $U_{q_1}$ is the set that $\gamma(-i_1)$ visits. Let $j_1$ be the last visit to $U_{q_1}$, namely $j_1 = \max\{i:\, \gamma(-i) \in U_{q_1)}\}$. 
		\item  The process stops if $j_{s-1} = - Q_k$, we set then set $i_s = j_s = \cdots = i_r =j_r = Q_k,$, and $q_s= \cdots= q_r = q_{s-1}$.

		Otherwise, let $i_s$ be the first visits to $\bigcup_q U_q$ for $i > j_{s-1}$, and $U_{q_s}$ the set it visits. Define $j_s$ be the last visit to $U_{q_s}$ and continue. 
	\end{itemize}
	Then 
	\begin{equation}
	  \label{eq:h-lower}
	  		\begin{aligned}
			& h_{L_k, c_k}(\zeta_k, \theta) + \epsilon_k \ge A_{L_k, c_k}^{Q_k}(\gamma_k(-Q_k), \gamma_k(0)) \\
			& = \sum_{s=1}^r A^{i_s - j_{s-1}}(\gamma_k(-i_s), \gamma_k(-j_{s-1})) + \sum_{s=1}^r A^{j_s - i_s}(\gamma_k(-j_s), \gamma_k(-i_s)) \\
			& \ge \sum_{s=1}^r A^{i_s - j_{s-1}}(\gamma_k(-i_s), \gamma_k(-j_{s-1})) 
			+ \sum_{s=1}^r  \left( h(\zeta_k, \gamma_k(-i_s)) - h(\zeta_k, \gamma_k(-j_s)) \right) - r\epsilon_k , 
		\end{aligned}
	\end{equation}
	where the subscript $L_k, c_k$ was omitted in the last two lines. 
	By restricting to a subsequence, we may assume that for all $\gamma_k$, the ordering $q_1, \cdots q_r$ are identical. Our previous observation implies for $s = 1, \cdots, r$, $i_s - j_{s-1}$ are bounded as $k \to \infty$. By restricting to another subsequence, we may assume $i_s - j_{s-1}$ is constant for all $k$, and $\gamma_k(-i_s) \to \theta_s$, $\gamma_k(-j_s) \to \theta_s'$ as $k \to \infty$. Note that for $s = 1, \cdots, r$, $\theta_s, \theta_s' \in \overline{B_\delta(\mS_{q_s})}$, therefore, there exists $\eta_s, \eta_s' \in \mS_{q_s}$ such that $\|\theta_s - \eta_s\|, \|\theta_s' - \eta_s'\| \le \delta$. Let us also note, by definition $\theta_0 = \theta_0' = \theta$, $\theta_r = \theta_r' = \zeta_0$.  Define $\eta_0 = \eta_0' = \theta$ and $\eta_r = \eta_r' = \zeta_0$. 
	Up to taking a subsequence, assume the weak KAM solutions $h_{L_k, c_k}(\zeta_k, \cdot) \to u(\cdot)$ uniformly. Take limit as $k \to \infty$ in \eqref{eq:h-lower}, we obtain
	\[
		\begin{aligned}
			& \lim_{k\to \infty} h_{L_k, c_k}(\zeta_k, \theta) \ge \sum_{s=1}^r \left( A_{L,c}^{i_s - j_{s-1}}(\theta_s, \theta_{s-1}') +  u(\theta_s) - u(\theta_s')  \right) \\
			& \ge \sum_{s=1}^r \left( A_{L,c}^{i_s - j_{s-1}}(\eta_s, \eta_{s-1}') + u(\eta_s) - u(\eta_s') - 4C\delta \right) \\
			& = \sum_{s=1}^r \left( A_{L,c}^{i_s - j_{s-1}}(\eta_s, \eta_{s-1}') + h_{L,c}(\zeta_0, \eta_s) - h_{L,c}(\zeta_0, \eta_s') - 4C\delta \right) \\
			& \ge \sum_{s=1}^r \left( h_{L,c}(\zeta_0, \eta_{s-1}') - h_{L,c}(\zeta_0, \eta_s') - 4C\delta  \right) \\
			& = h_{L,c}(\zeta_0, \eta_0') - h_{L,c}(\zeta_0, \eta_s') - 4rC\delta = h_{L,c}(\zeta_0, \theta) - 4rC\delta. 
		\end{aligned}
	\]
	Since $\delta$ is arbitrary, we obtain $\lim_{k\to \infty} h_{L_k, c_k}(\zeta_k, \theta) \ge h_{L,c}(\zeta_0, \theta)$. 
\end{proof}

{\bf Acknowledgement } V.K. has been partially support of the NSF grant DMS-1402164. K.Z. is supported by the NSERC Discovery grant, reference number 436169-2013.

\end{document}